\documentclass[reqno,11pt]{amsart}
\usepackage{amsmath,amssymb,mathrsfs,amsthm,amsfonts}
\usepackage{enumerate}
\usepackage[usenames,dvipsnames]{xcolor}
\usepackage{hyperref}
\usepackage[percent]{overpic}
\usepackage{comment}
\usepackage{algorithm}
\usepackage{algorithmic}
\usepackage{mathtools}
\usepackage{subcaption}
\usepackage{tikz}
\usepackage{bm}
\usetikzlibrary{calc}
\usetikzlibrary{arrows.meta,backgrounds}
\usepackage{multirow,array,longtable,booktabs}
\usetikzlibrary{arrows}

\hypersetup{
	colorlinks=true, linkcolor=blue,
	citecolor=ForestGreen
}
\usepackage[paper=letterpaper,margin=1in]{geometry}

\newtheorem{theorem}{Theorem}[section]
\newtheorem{lemma}[theorem]{Lemma}
\newtheorem{proposition}[theorem]{Proposition}

\newtheorem{prop}[theorem]{Proposition}
\newtheorem{corollary}[theorem]{Corollary}
\theoremstyle{definition}
\newtheorem{definition}[theorem]{Definition}

\newtheorem{remark}[theorem]{Remark}
\numberwithin{equation}{section}
\allowdisplaybreaks
\usepackage{acronym}

\acrodef{LDP}{Large Deviation Principle}

\newcommand{\E}{\mathbf{E}}	

\newcommand{\be}{\begin{equation}}
\newcommand{\ee}{\end{equation}}
\newcommand{\bea} {\begin{array}{rl}}
\newcommand{\eea} {\end{array}}
\newcommand{\bepa}{\left\{ \begin{array}{l}}
\newcommand{\eepa} {\end{array}\right.}
\makeatletter

\newcommand{\R}{\mathbb{R}} 
\newcommand{\N}{\mathbb{N}}

\newcommand{\ep}{\epsilon}

\newcommand{\ds}{\displaystyle}

\newcommand{\T}{\mathbb{T}}

\renewcommand{\hat}{\widehat}

\renewcommand{\bar}{\overline}

\usepackage{graphicx}

\newcommand{\bd}{{\bf d}}
\newcommand{\ov}{\overline}
\newcommand{\bx}{\bm{x}}

\newcommand{\bX}{\bm{X}}

\newcommand{\cP}{\mathcal{P}}

\newcommand{\eps}{\epsilon}

\newcommand{\cA}{\mathcal{A}}

\newcommand{\lip}{\text{Lip}}

\newcommand{\wt}{\widetilde}

\newcommand{\by}{\bm{y}}
\newcommand{\bz}{\bm{z}}

\newcommand{\balpha}{\bm{\alpha}}

\newcommand{\bp}{\bm{p}}

\newcommand{\bzero}{\bm{0}}
\newcommand{\sub}{\cP_{\text{sub}}}

\newcommand{\cad}{\text{c\`{a}dl\`{a}g} }

\newcommand{\btau}{\bm{\tau}}
\numberwithin{equation}{section}
\newcommand{\vs}{\vskip.05in}
\newcommand{\bsp}{\begin{split}}
\newcommand{\espl}{\end{split}}

\title[MFC with stopping]{Mean field control with stopping}

\author[P.\ Cardaliaguet]{Pierre Cardaliaguet}
\address{P.\ Cardaliaguet, 
Ceremade (UMR CNRS 7534), Universite Paris Dauphine-PSL
\newline\hphantom{\quad \ \ P. Cardaliaguet}
Place du Maréchal De Lattre De Tassigny, 75775 Paris Cedex 16,  France
}
\email{cardaliaguet@ceremade.dauphine.fr}

\author[J.\ Jackson]{Joe Jackson}
\address{J.\ Jackson,
Department of Mathematics, University of Chicago,
\newline\hphantom{\quad \ \ J. Jackson}
5734 S.~University Avenue, Chicago, Illinois 60637, USA
}
\email{jsjackson@uchicago.edu}

\author[P.\ E.\ Souganidis]{Panagiotis E. Souganidis}
\address{P.\ E.\ Souganidis, 
Department of Mathematics, University of Chicago,
\newline\hphantom{\quad \ \ J. Jackson}
5734 S.~University Avenue, Chicago, Illinois 60637, USA
}
\email{souganidis@uchicago.edu}

\thanks{\hskip-.125in P. Cardaliaguet was partially supported by P.E.\ Souganidis' Air
Force Office for Scientific Research grant FA9550-18-1-0494 and by the Agence Nationale de la Recherche (ANR),
project ANR-22-CE40-0010 COSS} 

\thanks{\hskip-.125in J. Jackson was supported by the NSF under Grant No. DMS2302703.} 

\thanks{\hskip-.13in P. E. Souganidis was partially  supported by the National Science Foundation grants DMS-2153822 
and DMS- 2452972.}

\begin{document}
	\begin{abstract}
          We study a high-dimensional stochastic optimization problem which features both control and stopping. In particular, a central planner steers a large population of particles, and can also remove particles at any time by paying a penalty. In the limit, we obtain a mean field control problem with discontinuous dynamics, in the sense that the controlled Fokker-Planck equation can have jumps. The value of the $N$-particle problem is characterized by a hierarchy of non-linear obstacle problems. The value of the limiting problem, meanwhile, solves an infinite-dimensional quasi-variational inequality (QVI). We introduce a notion of viscosity solution for this QVI, and obtain a comparison principle. Together with various regularity estimates, this comparison principle allows us to characterize the mean field value function as the unique viscosity solution of the QVI, and to establish the convergence of the $N$-particle value functions to the mean field value function.
	
	\end{abstract}

	
	\maketitle
		
	\setcounter{tocdepth}{1}
\tableofcontents

\section{Introduction} 

We study an optimization problem in which a central planner controls a large number of particles, and aims to minimize a symmetric cost function. The central planner chooses the drift of each particle, and can also choose to remove a particle at any time by paying a certain penalty. This is a continuation of the program initiated in \cite{CJSabsorption}, the goal of which is to establish convergence results for mean field control problems and mean field games in which the number of agents in the population can change over time.

\subsection*{Statement of the problem} 
We work for simplicity on the $d$-dimensional flat torus $\T^d$, and on a fixed probability space hosting independent Brownian motions $(W^i)_{i \in \N}$. We fix a time horizon $T > 0$. For each $N\in \N$ (the initial number of particles) and each $K\in \{1, \dots, N\}$ (the number of remaining particles), we study the stochastic control problem whose value function
\[ V^{N,K} : [0,T]\times (\T^d)^K \to \R\]
is given,  for each $t_0 \in [0,T]$ and $\bx_0=(x^1_0, \ldots, x^K_0) \in (\T^d)^K$, by the formula
\begin{align} \label{def.vnk.intro}
V^{N,K}(t_0,\bx_0) = \inf_{(\balpha, \btau) } \E\bigg[ \int_{t_0}^T \frac{1}{N} &\sum_{i = 1}^K L\big(X_t^i, \alpha_t^i, m_t^{N,K} \big) 1_{t < \tau^i} dt  \nonumber \\
&+ \frac{1}{N} \sum_{i = 1}^K \Psi\big(X_{\tau^i}^i, m_{\tau^i-}^{N,K}\big) 1_{\tau^i \leq T} + G(m_T^{N,K})\bigg].
\end{align}
In \eqref{def.vnk.intro}, the infimum is taken over all tuples
\begin{align*}
    \balpha = (\alpha^1,...,\alpha^K), \quad \btau = (\tau^1,...,\tau^K)
\end{align*}
where each $\alpha^i$ is a square-integrable, $\R^d$-valued process progressively measurable with respect to the filtration generated by the independent Brownian motions $(W^i)_{i = 1,...,K}$, and $\tau^i$ is a $[t_0,T] \cup \{+\infty\}$-valued stopping time (with respect to the same filtration). We note that the stopping time $\tau^i$ represents the time at which particle $i$ is removed, and $\tau^i = +\infty$ indicates that the particle is never removed. 
\vs
The $(\T^d)^K-$valued state process $\bX=(X^1,\ldots,X^K)$ is determined from the control $\balpha$ and the initial condition $\bx_0$ via
\be\label{phl0}
dX^i_t=\alpha^i_t dt +\sqrt{2} dW^i_t \quad t_0 \leq t \leq T, \quad X^i_{t_0} =x^i_0.
\ee
Finally, $m_t^{N,K}$ and $m_{t-}^{N,K}$ are random elements of $\sub= \cP_{\text{sub}}(\T^d)$, the space of sub-probability measures on $\T^d$, given by
\be\label{def.mnn}
 m_t^{N,K}=\frac{1}{N} \sum_{i = 1}^K \delta_{X_t^i} 1_{t < \tau^i} \ \ \text{and} \ \    m_{t-}^{N,K} = \lim_{s \uparrow t} \;m_t^{N,K} = \frac{1}{N} \sum_{i  = 1}^K \delta_{X_t^i} 1_{t \leq \tau^i}. 
\ee 
It will be convenient to also use the notation 
\begin{align*}
   V^{N,0}(t) = G(\bzero) \ \text{for } \ t \in [0,T], 
\end{align*}
where $\bzero \in \sub$ denotes the zero measure on $\T^d$. This should be interpreted as the value when all of the particles have already been removed. 
\vs
We emphasize that we are primarily interested in the optimization problem whose value is $V^{N,N}$, but we will need the full collection $(V^{N,K})_{K = 0,...,N}$ in order to obtain a closed PDE system; more on this below.
\vs
The data of the problem consist of three functions
\begin{align*}
    L : \T^d \times \R^d \times \sub \to \R, \quad \Psi : \T^d \times \sub  \to \R, \quad G : \sub \to \R.
\end{align*}
We refer to $L$ as the Lagrangian or running cost, $\Psi$ as the stopping penalty, and $G$ as the terminal cost. We also define a Hamiltonian $H:\T^d\times \R^d \times \sub \to \R$ from the Lagrangian $L$ in the usual way, that is,
\be\label{ath200}
H(x,p,m)=\sup_{a \in \R^d} \Big\{-p\cdot a -L(x,a, m)\Big\}.
\ee

\subsection*{The hierarchy of HJB equations}
Standard arguments from the theory of optimal stopping and viscosity solutions show that, for each fixed $N$, we have the system
\begin{equation} \label{stopping.gen2} \tag{$\text{HJB}_{N,K}$}
    \begin{cases}
      & \ds \max\bigg\{  - \partial_t V^{N,K} (t,\bx)- \sum_{i = 1}^K \Delta_{x^i} V^{N,K} (t,\bx) +
     \frac{1}{N} \sum_{i = 1}^K H\big( x^i,N D_{x^i} V^{N,K}(t,\bx), m_{\bx}^{N,K} \big) ,\\[1.75mm] 
      & \ds \qquad \underset{S \subset [K]\backslash\{\emptyset\}} \max \Big\{ V^{N,K}(t,\bx)-
      V^{N,K-|S|}(t,\bx^{-S}) -\frac{1}{N} \sum_{i \in S} \Psi(x^i,m_{\bx}^{N,K}) \Big\} \bigg\}=0  \ \ \text{in} \ \  [0,T) \times (\T^d)^K, \vspace{.2cm} \\[2.5mm]
     &   V^{N,K}(T,\bx) = G_{\Psi}^{N,K}(\bx) \quad  \bx \in  (\T^d)^K, 
    \end{cases}
\ee
where  $G^{N,K}_{\Psi} : (\T^d)^K \to \R$ given by 
\begin{align} \label{def.discretemonotoneenvelop}
    G^{N,K}_{\Psi}(\bx) = \inf_{S \subset [K]} \Big\{ G\big( m_{\bx^{-S}}^{N,K-|S|}\big) + \frac{1}{N} \sum_{i \in S} \Psi(x^i,m_{\bx}^{N,K}) \Big\}.
\end{align}
We remark that the terminal condition $G^{N,K}_{\Psi}$ of \eqref{stopping.gen2} is  not $ G^{N,K}(\bx) = G\big(m_{\bx}^{N,K}\big)$. From a PDE perspective, the new terminal data should be understood as the weak formulation of the terminal value problem. From a control-theoretic viewpoint, the new terminal data  indicates that by removing particles at the terminal time $T$, the central planner can essentially replace $G^{N,K}$ by $G_{\Psi}^{N,K}$. We collect in the Appendix some important facts about $G_{\Psi}^{N,K}$. 
\vs
In \eqref{stopping.gen2}, we are using the following notational conventions: First, given $N \in \N$, $K \in \{1,...,N\}$ and $\bx \in (\T^d)^K$, we set $m_{\bx}^{N,K} = \frac{1}{N} \sum_{i = 1}^K \delta_{x^i}$. For $\bx \in (\T^d)^K$ and $S \subset [K] \coloneqq \{1,...,K\}$ with $|S| < K$, we denote by $\bx^{-S}$ the element of $\R^{K - |S|}$ obtained by deleting the coordinates of $\bx$ with indices in $S$; for example, if $K = 3$ and $\bx = (x^1,x^2,x^3)$, and $S = \{2\}$, then $\bx^{-S} = (x^1,x^3) \in (\T^d)^2$. We also allow the case $S = [K]$ in \eqref{stopping.gen2}, in which case $V^{N,K-|S|}(t,\bx^{-S})$ is interpreted to mean $V^{N,0}(t) = G(\bzero)$.
\vs

\subsection*{The mean field problem}
The limit problem corresponds (at least formally) to the following variational problem.
 For $(t_0,m_0)\in [0,T]\times \sub$, we define $\cA_{t_0,m_0}$ to be the set of triples $(m,\alpha, \mu)$, such that $[t_0,T] \ni t \mapsto m_t \in \sub$ is \cad, $\alpha : [t_0,T] \times \T^d \to \R^d$ is a measurable function satisfying 
    \begin{align}
        \int_{t_0}^T \int_{\T^d} |\alpha(t,x)|^2 dm_t(x) dt < \infty,
    \end{align}
 and $\mu$ is a non-negative Borel measure on $[t_0,T] \times \T^d$, such that the equation
    \begin{align} \label{meqn0}
        \partial_t m - \Delta m + \text{div}(m\alpha) = - \mu \ \ \text{in} \ \ [t_0,T] \times \T^d \ \ \text{and} \ \  m_{t_0-} = m_0.
    \end{align}
is satisfied in the sense of distributions on $[t_0,T] \times \T^d$, that is, 
for any  test function $\phi\in C^\infty([t_0,T]\times \T^d)$,  we have 
\be\label{meqn}
\begin{split} 
&\int_{\T^d} \phi(T,x) dm_T(x) - \int_{\T^d} \phi(t_0,x) dm_0(x) =\\
&\int_{t_0}^T \int_{\T^d} \Big( \partial_t \phi(t,x) + \Delta \phi(t,x) + D\phi(t,x) \cdot \alpha(t,x) \Big)dm_t(x) dt - \int_{[t_0,T] \times \T^d} \phi(t,x) d\mu(t,x). 
\end{split}
\ee
\vs
We define the mean field value function $U : [0,T] \times \sub \to \R$ by 
\begin{align*}
    U(t_0,&m_0) = \inf_{(m,\alpha,\mu) \in \cA_{t_0,m_0}} J_{t_0,m_0}(m,\alpha,\mu), \text{ where }
    \\
    & J_{t_0,m_0}(m,\alpha,\mu) = \int_{t_0}^{T} \int_{\T^d} L\big( x, \alpha(t,x), m_t \big) dm_t(x) dt + \int_{[t_0,T] \times \T^d} \Psi\big(x, m_{t-}\big) d\mu(t,x) + G\big( m_T \big).
\end{align*}
A formal dynamic programming argument indicates that $U$ is $\Psi$-non-increasing, in the sense that, for all  $(t,m) \in [0,T] \times \sub$ and all $m',m\in \sub$ such that $m'\leq m$,
\be\label{def.psinon-increasing}
    U(t,m) \leq U(t,m') + \int_{\T^d} \Psi(x,m) d(m-m')(x).
\ee
Here we use the notation $m' \leq m$ to mean that $m'(A) \leq m(A)$ for all Borel subsets $A$ of $\T^d$. The inequality \eqref{def.psinon-increasing} can be deduced from the observation that for any $m' \leq m$, the controller may choose to jump from $m$ to $m'$ immediately by removing the mass $m - m'$ at time $t$, which corresponds to choosing a measure $\mu$ such that $\mu(\{t\} \times A) = (m - m')(A)$ for $A \subset \T^d$. 
\vs
Moreover, $U$ should satisfy the inequality
\begin{equation} \label{u.subsol.intro}
\begin{split}
   -\partial_t U -\int_{\T^d} \Delta_x \frac{\delta U}{\delta m}(t,m,x)m(dx) +\int_{\T^d} H\Big(x, D_x \frac{\delta U}{\delta m}(t,m,x),m \Big)m(dx) \leq 0 \ \ \text{ on \ \ $[0,T) \times \sub$}
     \end{split}
\end{equation}
in the viscosity sense, a fact that can be deduced 
by noting that the controller may choose $\mu = 0$, in which case the problem reduces to a more standard mean field control problem. 
\vs
It also follows that $U$ should satisfy 
\be\label{phl11}
   -\partial_t U -\int_{\T^d} \Delta_x \frac{\delta U}{\delta m}(t,m,x)m(dx) +\int_{\T^d} H\Big(x, D_x \frac{\delta U}{\delta m}(t,m,x),m \Big)m(dx) = 0 \ \ \text{in } \mathcal{D}
   \ee
where    
   \be\label{phl11.1}
 \begin{split}   
 & \mathcal{D} = \Big\{ (t,m) : U(t,m) < U(t,m') + \int_{\T^d} \Psi(x,m)  (m - m')(dx) \\
 &\hspace{6cm} \ \text{for all} \  m' \ \text{such that}  \ m' \leq m, \, m' \neq m \Big\},
\end{split}
\ee
as can be seen, at least formally,  from the fact that if $(t,m) \in \mathcal{D}$, it is not optimal to remove mass immediately. 
\vs
Finally, we expect that $U(T,\cdot) = G_{\Psi}$, where $G_{\Psi}$ is the ``$\Psi$-monotone envelope" of $G$, given by
\begin{align} \label{def.monotoneenvelope}
    G_{\Psi}(m) = \inf_{m' \leq m} \Big\{ G(m') + \int_{\T^d} \Psi(x,m) (m-m')(dx) \Big\}.
\end{align}
\vs
We summarize these four conditions in the infinite-dimensional quasi-variational inequality
\be \label{eq.HJstoppingPSI} \tag{$\text{HJB}_{\infty}$}
\begin{cases} 
& \displaystyle \max\bigg\{    -\partial_t U -\int_{\T^d} \Delta_x \frac{\delta U}{\delta m}(t,m,x)m(dx) +\int_{\T^d} H\Big(x, D_x \frac{\delta U}{\delta m}(t,m,x),m \Big)m(dx), \vspace{.1cm} \\[1.2mm] 
& \displaystyle \qquad  \sup_{m' \leq m, m' \neq m}  \text{sgn} \Big( U(t,m) - U(t,m') - \int_{\T^d} \Psi(y,m) (m-m')(dy) \Big) \bigg\}= 0 \; \text{in}  \;  [0,T) \times \sub, \vspace{.1cm}  \\[1.5mm]
& \ds  U(T,m)=G_{\Psi}(m), 
\end{cases}
\ee
where we have set $\text{sgn}(x)$ to be $1$ when $x>0$, $0$ when $x=0$ and $-1$ for $x<0$.

\subsection*{The main results}
We now state our main results, although we prefer to defer the precise assumptions on $H$, $\Psi$, and $G$ until Section \ref{sec.prelim} below. In particular, these are stated in \eqref{phl1}, \eqref{phl2}, and \eqref{phl3}. Our first main result, which is a crucial ingredient for the convergence and uniqueness results which will be presented shortly,  is a Lipschitz bound for both $V^{N,K}$ and $U$, expressed in terms of a bounded Lipschitz distance on $\sub$: 
$$
\bd(m,n)=  \sup_{\|f \|_{W^{1,\infty}(\T^d)} \leq 1} \int_{\T^d} f d(m- n) \  \text{  for any \ $m,n\in \sub $}. 
$$

\begin{theorem} \label{thm.main.lip} 
Assume  \eqref{summarize}.  Then there exists  a constant $C>0$ such that, for each $t,s \in [0,T]$ and $m,n \in \sub$, 
    \begin{align} \label{ulip}
        | U(t,m) - U(s,n) | \leq C \Big( \bd(m,n) + |t-s|^{1/2} \Big), 
    \end{align}
    and, for each $t,s \in [0,T]$, $N \in \N$, $K,J \in \{1,...,N\}$, $\bx \in (\T^d)^K$ and $\by \in (\T^d)^J$, 
    \begin{align} \label{vnklip}
        \big| V^{N,K}(t,\bx) - V^{N,J}(s,\by) \big| \leq C \Big( \bd\big( m_{\bx}^{N,K}, m_{\by}^{N,J} \big) + |t-s|^{1/2} \Big).
    \end{align}
\end{theorem}
\vs
The estimate \eqref{vnklip} is proved in several steps. First, we use a doubling of variables argument to prove in Lemma \ref{lem.derivativebound2} that,  for $t \in [0,T]$ and $\bx, \by \in (\T^d)^K$, 
\begin{align} \label{vnklip1.intro}
    |V^{N,K}(t,\bx) - V^{N,K}(t,\by)| \leq \frac{C}{N} \sum_{i = 1}^K |x^i - y^i|.
\end{align}

Together with the symmetry of $V^{N,K}$, this implies that \eqref{vnklip} holds when $t = s$ and $K = J$. 
\vs
Next, in Lemma \ref{lem.removingmass}, we obtain the estimate 
\begin{align} \label{removemass.intro}
    |V^{N,K-1}(t,\bx^{-i}) - V^{N,K}(t,\bx)| \leq C/N. 
\end{align}
In some sense, this shows that \eqref{vnklip} holds when $t = s$, $J = K-1$, and $\by = \bx^{-i}$, that is,  we have the required spatial regularity when a single particle is removed from the system. 
\vs
Finally, we study the properties of the distance $\bd$ when restricted to empirical measures to deduce that \eqref{vnklip1.intro} and \eqref{removemass.intro} together imply the spatial regularity in \eqref{vnklip}, and then infer the time regularity from the space regularity through dynamic programming.
\vs 
The bound \eqref{ulip} is much more delicate because of the fact that the controlled dynamics are discontinuous. Our strategy is to infer the regularity for $U$ from the optimality conditions for the corresponding mean field control problem; very formally, $U$ should be Lipschitz if optimal controls $\alpha$ are uniformly bounded. 
\vs
Making sense of optimality conditions for $U$ is challenging in view of the discontinuous dynamics, so we first regularize the problem in several ways, then study optimality conditions for the regularized problem (see Proposition \ref{prop.dualdeltaAppend}). Then we use these optimality conditions to obtain a Lipschitz bound on the regularized problem,  which is independent of the regularization parameters (Proposition \ref{prop.UthetadeltaLip}). Finally, we prove that the value functions of the regularized problem converge to $U$ (Lemmas \ref{lem.UthetatoU} and \ref{lem.UthetadeltatoUtheta}) to complete the proof of \eqref{ulip}. 
\vs
Our second main result characterizes $U$ as the unique Lipschitz continuous viscosity solution of \eqref{eq.HJstoppingPSI}, a fact which is  a  consequence of the  comparison principle in  Theorem \ref{thm.comparison}. To state the uniqueness result precisely, we say that a function $V : [0,T] \times \sub \to \R$ is uniformly Lipschitz in $m$,  if there exists a constant $C$ such that, for all $t \in [0,T]$, $m,n \in \sub$, 
\begin{align*}
    |V(t,m) - V(t,n)| \leq C \bd(m,n).
\end{align*}

\begin{theorem} \label{thm.uniqueness}
Assume  \eqref{summarize}. Then the value function $U$ is a viscosity solution of the quasi-variational inequality \eqref{eq.HJstoppingPSI} in the sense of Definition \ref{def.viscositysoln} below. It is the unique solution which is uniformly Lipschitz in $m$. 
\end{theorem}
The key comparison result, Theorem \ref{thm.comparison}, is proved via a doubling of variables argument. The main challenge is to build appropriate test functions. In particular, to ensure that the optima in our doubling of variables procedure are achieved at points where the supersolution is not touching the barrier, we must build test functions which are ``$\Psi$-decreasing" in an appropriate sense; see Definition \ref{def.psinon-increasing}. This is a difficult task. In the end, we managed to prove comparison by studying maxima of the expression
\be \label{doubling.intro}
\begin{split}
&V^-(t,m)-V^+(s,n) -\frac{1}{\ep}\Phi(m-n) 
 \\
&\qquad  - \int_{\T^d} \Psi(x,m) (m-n)(dx) -\frac{1}{2\ep}(t-s)^2 -\delta( \|m\|^2_2+ \|n\|^2_2) -\theta(T-t) - \theta n(\T^d)
\end{split}
\ee
for some subsolution $V^-$ and supersolution $V^+$ of \eqref{eq.HJstoppingPSI}, and $\delta, \eps, \theta > 0$, where $\Psi$ is exactly the stopping penalty and $\Phi : H^{-1} \to \R$ is a tailor-made non-increasing function defined in \eqref{defPhiPsi}. 
\vs 
 The non-increasing test function $\Phi$, which we build,  is not a smooth on $\sub(\T^d)$ in the sense of Definition \ref{def.testfunction}. Instead,  $\Phi \in C^{1,1}(H^{-1})$, where $H^{-1} = \big(H^1(\T^d)\big)^*$. The term $\delta( \|m\|^2_2+ \|n\|^2_2)$ appearing in \eqref{doubling.intro} guarantees that optimizers of that expression are in $L^2(\T^d)$, that is, they have square-integrable densities, so that we can make sense of the differential of $\Phi$. However, working with these less regular test functions is not without complications. Indeed, for our application to to the mean field convergence, it is important that Definition \ref{def.viscositysoln} (our definition of viscosity sub/super-solutions) involves only smooth test functions. Thus, in order to make sense of \eqref{doubling.intro}, we must execute a careful mollification argument, which shows that, if $U$ is a viscosity sub/super-solution with respect to smooth test functions, then it is also a sub/super-solution with respect to more singular test functions, like those appearing in \eqref{doubling.intro}. This is the subject of Proposition \ref{prop.subsol.equiv}. 
\vs 
The use of singular penalization terms in the study of viscosity solutions of Hamilton-Jacobi equations is classical. In the setting of equations on spaces of probability measures, we refer to \cite{daudinseeger, BertucciLionsSoug} for a doubling of variables argument involving the $2$-Wasserstein distance and an entropy penalization. Similar ideas have also appeared in \cite{Lionsvideo}. The closest precedent for the doubling of variables \eqref{doubling.intro} is our previous work \cite{CJSabsorption}, where we used a similar strategy to treat an infinite-dimensional Hamilton-Jacobi equation with non-standard boundary conditions.
\newline \newline 
Finally, we have our main convergence result, which shows that the family  $(V^{N,K})_{K = 1,...,N}$ converges as $N \to \infty$ towards $U$. 
\begin{theorem} \label{thm.convergence}
Assume  \eqref{summarize}. Then 
    \begin{align*}
        \lim_{N \to \infty} \max_{K = 1,...,N} \sup_{(t,\bx) \in [0,T] \times (\T^d)^K} \big| V^{N,K}(t,\bx) - U(t,m_{\bx}^{N,K}) \big| = 0.
    \end{align*}
\end{theorem}
The proof of Theorem \ref{thm.convergence} follows from a compactness/uniqueness argument. Theorem \ref{thm.main.lip} shows that $(V^{N,K})_{K= 1,...,N}$ are equicontinuous in an appropriate sense, and so we can extract subsequential limit points. By showing that every subsequential limit point is a viscosity solution of \eqref{eq.HJstoppingPSI}, we can deduce Theorem \ref{thm.convergence} from Theorem \ref{thm.uniqueness}. The challenge here is that the test functions appearing in the doubling of variables argument \eqref{doubling.intro} are singular; well-defined only on $\sub \cap L^2$. But to prove that limit points of the $V^{N,K}$'s  are solutions, we would like to evaluate the relevant test functions at empirical measures. This difficulty is circumvented by Proposition \ref{prop.subsol.equiv}, which shows that a sub/supersolution with respect to smooth test functions is also a sub/supersolution with respect to the relevant class of singular test functions.
\vskip-1in
\subsection*{Related literature}
Our results are connected to several active areas of research, in particular mean field models involving optimal stopping, the well-posedness of Hamilton-Jacobi-Bellman equations set on spaces of measures, and the convergence problem in mean field control.
\vs 
In the absence of stopping, the issues studied in the present paper are by now reasonably well-understood. In this case, the value functions $(V^{N,K})_{K = 1,...,N}$ are replaced by a single value function $V^N : [0,T] \times (\T^d)^N \to \R$, and the limiting value function $U : [0,T] \times \cP(\T^d) \to \R$ solves an HJB equation on $\cP(\T^d)$, without any ``obstacle". The limiting HJB equation in this case has received huge attention in recent years; see e.g. \cite{BayraktarEkrenZhangCPDE, BayraktarEkrenHeZhang, BCEQTZ, touzizhangzhou, bertucci2023, confortihj, CKT23b, CKTT24, daudinseeger, DJS2025, silbersteintonon, bertuccisilberstein}. The uniqueness proof developed here is closest in spirit to the ones presented in the papers \cite{Lionsvideo, daudinseeger, CJSabsorption, BertucciLionsSoug}, which also use some sort of ``singular penalization" in their doubling of variables arguments. The qualitative convergence of $V^N$ to $U$ has been obtained in various settings, through both PDE and probabilistic methods in \cite{budhiraja2012, Lacker2017, DjetePossamaiTan, djete2022extended, GangboMayorgaSwiech, mayorgaswiech, SW}. These results have been quantified in some cases in \cite{bayraktarcecchinchakrabory, cdjs2023, cjms2023, BayraktarEkrenZhangQuant, ddj2023, CDJM, DJSrates}.
\vs 
Mean field games (rather than mean field control problems) with optimal stopping have been thoroughly studied. Particular models were proposed in \cite{nutzstopping} and \cite{cdlstopping}. Meanwhile, \cite{bertuccistopping} pursued a PDE approach, characterizing optimizers in terms of a novel forward-backward PDE system and introducing the notion of mixed solutions. Alternative approaches were later developed in \cite{Bouveret2020, Dumitrescu2021}. These works focused on studying the limiting model, and did not address the connection with finite-player games. 
\vs 
As far as mean field control problems involving stopping, the only work we are aware of is the series of three papers \cite{TTZ1,TTZ2,TTZ3} by Talbi, Touzi, and Zhang. In \cite{TTZ1}, an equation similar to \eqref{eq.HJstoppingPSI} is proposed for the limiting value function, in \cite{TTZ2} a notion of viscosity solutions is introduced and a comparison principle is proved, and in \cite{TTZ3}, a convergence result is obtained. While the Talbi-Touzi-Zhang model shares several features with the model proposed here, there are also key differences. For example, in \cite{TTZ1} particles are ``frozen" rather than ``removed" by the controller, which changes the natural state space of the problem. In addition, the dynamics are very general but uncontrolled, so that the optimization is only over stopping times and the relevant PDEs are linear obstacle problems. 
\vs 
More important than the difference in the models is the difference in the methodology.  First, in \cite{TTZ1, TTZ2, TTZ3}, the limit HJ equation is stated in an extended state-space, including information on the state of the particle (stopped or not). Second, the notion of viscosity solutions is inspired by the one for the path dependent PDEs, and the techniques of proof (comparison, mean field limit) rely on stochastic calculus. In particular, the convergence of the values in \cite{TTZ3} relies on propagation of chaos; see Step 3 in the proof of Theorem 4.2. This connection to stochastic calculus can be seen already at the level of the definition of viscosity solution, which depends on the underlying dynamics of the model, through the ``neighborhoods" $\mathcal{N}_{\delta}(t,m)$ appearing in \cite[Section 2.3]{TTZ3}.
\vs
In contrast, our limit Hamilton-Jacobi equation is stated in the initial state space. Thus, while we lose information (we do not keep the knowledge that a particle has been stopped or not), our HJ equation is the ``natural counterpart'' of the obstacle problem \eqref{stopping.gen2} satisfied by the particle system. Our definition of viscosity solution is also closer to the one employed in the classical theory of viscosity solutions in infinite dimensional spaces (using smooth test functions defined on the whole space) and our techniques of proof for the comparison and the mean field limit are of a purely analytical flavor. We note that defining viscosity solutions in terms of smooth test functions has several advantages, one of the most important one being the stability properties of the equation, which in turn play a role in our convergence proof.

\subsection*{Organization of the paper} The paper is organized as follows. In Section \ref{sec.prelim} we introduce our notation, and state precisely our standing assumptions. Section \ref{sec.vnk} is devoted to the proof of the Lipschitz bounds \eqref{vnklip} for $V^{N,K}$ stated in Theorem \ref{thm.main.lip}. In Section \ref{sec.U} we obtain the corresponding regularity estimates~\eqref{ulip} for the value function $U$. In Section \ref{sec.viscosity} we define precisely a notion of viscosity solution for \eqref{eq.HJstoppingPSI}. We prove a comparison result (Theorem \ref{thm.comparison}) for \eqref{eq.HJstoppingPSI} when the Hamiltonian is globally Lipschitz continuous and then use this together with the Lipschitz bound on $U$ to prove Theorem \ref{thm.uniqueness}. Section \ref{sec.convergence} is devoted to the proof of Theorem \ref{thm.convergence}. Finally, in the Appendix we deal with some technical issues regarding mollifications and $\Psi$-monotone envelopes of the terminal condition $G$.

\section{Notation, preliminaries, and assumptions} \label{sec.prelim}

\subsection*{Notation} We recall that $d \in \N$ and $T \in (0,\infty)$ are fixed throughout the paper. We denote by $\T^d$ the $d$-dimensional flat torus, and by $\sub$ or $\sub(\T^d)$ the space of sub-probability measures on $\T^d$, that is, the space of non-negative Borel measures on $\T^d$ with total mass at most one. Given $m,n \in \sub$, we write $m \leq n$ if $m(A) \leq n(A)$ for all Borel subsets $A$ of $\T^d$. We write $m \lneq n$ if $m \leq n$ and $m \neq n$. 
\vs
We denote by $W^{1,\infty} = W^{1,\infty}(\T^d)$ the space of bounded and Lipschitz functions on $\T^d$ with norm 
\begin{align*}
    \| f \|_{W^{1,\infty}} = \|f\|_{1,\infty} = \| f\|_{\infty} + \| Df \|_{\infty}, 
\end{align*}
with $\|\cdot\|_{\infty}$ being the standard $L^{\infty}$ norm. We write $H^1$ for the space of square-integrable functions with square-integrable first derivatives with norm $\| \cdot \|_{H^1}$ inherited from the standard inner product
\begin{align*}
   \langle f,g \rangle_{H^1} = \int_{\T^d} f(x) g(x) dx + \int_{\T^d} Df(x) \cdot Dg(x) dx.
\end{align*}
Then, $H^{-1}$ denotes the space of bounded linear functionals on $H^1$ with the inner product $\langle \cdot, \cdot \rangle_{H^{-1}}$ and norm $\| \cdot \|_{H^{-1}}$ inherited from duality with $H^1$. 
\vs
We mostly  work with the distance on $\sub$ inherited from duality with $W^{1,\infty}$, that is, we define a metric $\bd$ on $\sub$ via
\be\label{ord1}
\bd(m,n)= \|m-n\|_{(W^{1,\infty})^*} = \sup_{\|f \|_{1,\infty} \leq 1} \int_{\T^d} f d(m- n) \  \text{  for any \ $m,n\in \sub $. }
\ee
As usual, $\bd_1$ denotes the Monge-Kantorovitch distance between two measures with the same mass given by 
$$
\bd_1(m,n)= \sup_{\phi\in W^{1, \infty}, \; \|D\phi\|_\infty\leq 1} \int_{\T^d} \phi(m-n)\qquad \forall m,n \in \sub \; \text{with}\; \int_{\T^d}(m-n)=0.
$$
We will often use the  facts that 
\be\label{ord2}
\begin{split}
& \bd(m, n) \geq  |n(\T^d)-m(\T^d)| \ \ \text{for all} \ \  m,n\in \sub \, ,\\ 
& \bd(m, n) =  |n(\T^d)-m(\T^d)| \ \ \text{ if \ $ m\leq n$ or $n\leq m$}.
\end{split}
\ee
The lower bound on $\bd$ can be obtained by considering the test function $f(x) = 1$, and the equality when $m \leq n$ or $n \leq m$ comes from the fact that in this case, $f(x) = 1$ or $f(x) = -1$ is optimal for the maximization problem which defines $\bd(m,n)$. 
\vs
We say that a function $\Phi : \sub \to \R$ is non-increasing if $\Phi(m) \leq \Phi(n)$ whenever $n \leq m$. We say that a function $\Phi : \sub \to \R$ is $C^1$,  if there exists a continuous function $\dfrac{\delta \Phi}{\delta m} : \sub \times \T^d \to \R$, often called the linear derivative of $\Phi$,   with the property that 
\begin{align*}
    \Phi(m) - \Phi(n) = \int_0^1 \int_{\T^d} \frac{\delta \Phi}{\delta m}\big( (1-t) n + t m, x \big) d(m - n)(x) \quad \text{ for all } m,n \in \sub. 
\end{align*}
\vs
We use similar notation if $\Phi$ depends on an additional finite-dimensional parameter, that is,  for $\Phi : [0,T] \times \sub \to \R$. At times, we will also work with functions $\Phi : H^{-1} \to \R$  or $\Phi : [0,T] \times H^{-1} \to \R$. In this case, we use $D_{H^{-1}} \Phi$ to denote the Fr\'echet derivative of $\Phi$, and $\nabla_{H^{-1}} \Phi$ to denote the Hilbertian gradient of $\Phi$, that is,  for $q \in H^{-1}$, $\nabla_{H^{-1}} \Phi(q)$ is defined by 
\begin{align*}
    \Phi(p) = \Phi(q) + \langle \nabla_{H^{-1}} \Phi(q), p - q \rangle_{H^{-1}} + o(\|p-q\|_{H^{-1}}), 
\end{align*}
and $D_{H^{-1}} \Phi(q) \in H^1$ is defined by 
\begin{align*}
    \Phi(p) = \Phi(q) + \langle D_{H^{-1}} \Phi(q), p - q \rangle_{H^{-1},H^1} + o(\|p-q\|_{H^{-1}}) = \Phi(q) + \big(p-q\big) \big( D_{H^{-1}} \Phi(q) \big) + o(\|p-q\|_{H^{-1}}).
\end{align*}
\vs
Finally, we emphasize that, throughout the paper, we will write $C$ for constants which depend on the data and may change from line to line. 

\subsection*{Standing assumptions}

We now state our main assumptions. First, throughout the paper, we make the following assumption on the Hamiltonian $H$:
\be\label{phl1}
\begin{cases}
&\text{the Hamiltonian $H$ is jointly $C^2$, and there is a constant $C$} \\[1.5mm] &\text{   such that for all $x, y \in \T^d, p \in  \R^d$ and $m,n \in  \sub $,}\\[1.5mm] 
&\hskip.25inC^{-1} |p|^2 - C \leq H(x,p,m) \leq C |p|^2 + C , \qquad C^{-1} I_{d \times d} \leq  D_{pp}H(x,p,m) \leq C I_{d \times d},\\[1.5mm]
&\hskip.5in   \left|H(x,p,m)-H(y,p,n)\right| \leq C\big(1 + |p|\big)\Big(|x-y|+\bd(m,n) \Big), \\[1.5mm]
 &\hskip.5in \ds   \Big| \frac{\delta H}{\delta m}(x,p,m,y) \Big| +  \Big| D_y \frac{\delta H}{\delta m}(x,p,m,y) \Big| \leq C\big(1 + |p|^2\big).
\end{cases}
\ee

In \eqref{phl1}, jointly $C^2$ means that mixed second derivatives involving $x$, $p$, $m$ exist and are continuous, the derivatives with respect to $m$ being linear derivatives as defined above, and in addition $D_y \frac{\delta H}{\delta m}(x,p,m,y)$ exists and is continuous. 
\vskip.05in

For the stopping penalty and terminal cost  we assume that 
\begin{equation}\label{phl2}
\begin{cases} 
&\text{the cost $\Psi$ and its linear derivative $\ds \frac{\delta \Psi}{\delta m} : \T^d \times \sub \times \T^d \to \R$ are both}\\[1.2mm] 
&\text{globally
Lipschitz continuous, there is a constant $C$ such that}\\[1.2mm]
& \hskip1.25in \underset{m \in \sub} \sup \| \Psi(\cdot, m)\|_{C^2(\T^d)} \leq C,\\[1.5mm]
& \text{and, for each } x \in \T^d,\\[1.2mm]
& \hskip1.25in   m \mapsto \Psi(x,m) \text{ is non-increasing,}
\end{cases}
\ee
and
\be\label{phl3}
\text{the terminal cost $G$ is Lipschitz continuous with respect to $\bd$}.
\ee
In what follows, we summarize  these assumptions  in
\be\label{summarize}
\eqref{phl1}, \eqref{phl2} \ \text{and} \  \eqref{phl3} \  \text{hold}.
\ee
Most of the conditions appearing in \eqref{phl1}, \eqref{phl2}, and \eqref{phl3} are technical in nature, and, for example, the regularity conditions on $H$ and $\Psi$ could likely be relaxed in various directions. 
\vskip.075in
The monotonicity of the cost $\Psi$, however, is a key structural condition. It ensures that, for any $m_2\leq m_1\leq m_0$, we have 
$$
\int_{\T^d}\Psi(x,m_0)(m_0-m_1)(dx) +\int_{\T^d}\Psi(x,m_1)(m_1-m_2)(dx) \geq \int_{\T^d}\Psi(x,m_0)(m_0-m_2)(dx);
$$
in other words, it is less costly to make one big jump from $m_0$ to $m_2$ than two consecutive ones.
\vs

The conditions on $H$ appearing in \eqref{phl1} imply several corresponding estimates on the Lagrangian $L$. Since $L$ is continuous and convex in its second argument, we have 
\begin{align} \label{lagrangianduality}
L(x,a,m)= \sup_{p\in \R^d} \Big( - p\cdot a- H(x,p,m) \Big)=  D_a L(x,a,m) \cdot a - H\big(x,-D_aL(x,a,m),m\big).
\end{align}
Thus, we see that \eqref{phl1} implies that
\be\label{hypL}
C^{-1}|\alpha|^2-C \leq L(x,\alpha,m)  \leq C(|\alpha|^2+1), \qquad C^{-1} I_{d \times d} \leq D_{aa} L(x,a,m) \leq C I_{d \times d}.
\ee
Moreover, since $-D_a L(x,a,m)$ is the minimizer in \eqref{lagrangianduality}, the assumption \eqref{phl1} on $H$ also implies that
\begin{align} \label{DaLgrowth}
    |D_a L(x,a,m)| \leq C \big(1 + |a| \big).
\end{align}
Finally, differentiating \eqref{lagrangianduality} shows that 
\begin{align*}
    \frac{\delta L}{\delta m}(x,a,m,y) = - \frac{\delta H}{\delta m}\big( x, - D_aL(x,a,m),m, y \big), 
\end{align*}
so that \eqref{phl1} and \eqref{DaLgrowth} together imply that 
\be\label{hypdeltaL}
\left| \frac{\delta L}{\delta m} (x,\alpha, m,y) \right| + \left| D_y \frac{\delta L}{\delta m} (x,\alpha, m,y) \right|  \leq C(|\alpha|^2+1).
\ee

\section{The regularity of $V^{N,K}$} \label{sec.vnk}

The proof of  the estimate~\eqref{vnklip} in Theorem~\ref{thm.main.lip} is based on a number of technical lemmata which we state and prove first. 
\vs

\begin{lemma} \label{lem.derivativebound2}
Assume  \eqref{summarize}.  Then there exists a constant $C$ such that, for each $N \in \N$, each $K \in \{1,...,N\}$, $t \in [0,T]$ and $\bx,\by \in (\T^d)^K$, we have 
     \begin{align*}
         |V^{N,K}(t,\bx) - V^{N,K}(t,\by)| \leq \frac{C}{N} \sum_{i = 1}^K |x^i - y^i|.
     \end{align*}
 \end{lemma}

 \begin{proof} It follows from the  classical parabolic regularity theory that  $V^{N,K} \in C^{1,2}$ in the set 
 \begin{align*}
     \Big\{(t,\bx), \; V^{N,K}(t,\bx)< \inf_{S \subset [K]} \Big( V^{N,K - |S| }(t,\bx^{-S}) + \frac{1}{N} \sum_{i \in S} \Psi(x^i,m_{\bx}^{N,K}) \Big) \Big\}.
 \end{align*} 
      Fix $N \in \N$, fix $\eps > 0$ and let  $(\ov{K}, \ov{t},\ov{\bx}, \ov{\by})$ be a maximum point for  the optimization problem
     \begin{align} \label{opt.lipbound}
         \max_{K = 1,...,N} \sup_{t \in [0,T], \, \bx,\by \in (\T^d)^K} \Big\{ V^{N,K}(t,\bx) - V^{N,K}(t,\by) - \Big(\lambda(t) + \frac{C_0 K}{N}\Big) \frac{1}{N} \sum_{i = 1}^K \big(|x^i - y^i|^2 + \eps \big)^{1/2} \Big\},
     \end{align}
     where $\lambda : [0,T] \to (0,\infty)$ and $C_0 > 0$ are respectively a smooth function and a constant to be decided later. 
     Moreover, we assume that $\ov{K}$ is minimal, in the sense that any other maximizer $(\ov{K}', \ov{t}',\ov{\bx}', \ov{\by}')$ satisfies $\ov{K} \leq \ov{K}'$. 
     \vs
     We use Lemma \ref{lem.envelopeLip} and choose $\lambda(T)$ large enough so  that,    for each $N \in \N$, $K \in \{1,...,N\}$, $\bx,\by \in (\T^d)^K$, 
     \begin{align*}
         |G_{\Psi}^{N,K}(\bx) - G_{\Psi}^{N,K}(\by)| \leq \frac{\lambda(T)}{N} \sum_{i = 1}^K |x^i - y^i|. 
     \end{align*}
    This implies that 
     \begin{align*}
         \max_{K = 1,...,N} \sup_{ \bx,\by \in (\T^d)^K} \Big\{ V^{N,K}(T,\bx) - V^{N,K}(T,\by) - \Big(\lambda(T) + \frac{C_0 K}{N} \Big) \frac{1}{N} \sum_{i = 1}^K \big(|x^i - y^i|^2 + \eps \big)^{1/2} \Big\} \leq 0, 
     \end{align*}
     and, thus, if $\ov{t} = T$, then we must have,   for all $K,t,\bx,\by$,
     \begin{align} \label{epsbound}
         V^{N,K}(t,\bx) - V^{N,K}(t,\by) \leq \Big(\lambda(t) + \frac{C_0 K}{N} \Big)\frac{1}{N}  \sum_{i = 1}^K \big(|x^i - y^i|^2 + \eps \big)^{1/2},
     \end{align}
  
     On the other hand, if $\ov{t} < T$, then by the minimality of $\ov{K}$, we see that, for any $S \subset [\ov K] \backslash\{\emptyset\}$, we have 
     \be\label{vnkminus}
     \begin{split} 
         &V^{N,\ov K}(\ov t, \ov \bx) - V^{N,\ov K}(\ov t, \ov \by) - \Big( \lambda(\ov t) + \frac{C_0 \ov K}{N} \Big) \frac{1}{N} \sum_{i = 1}^{\ov K} \big( |\ov{x}^i - \ov{y}^i|^2 + \eps \big)^{1/2}\\
         &> V^{N,\ov K- |S|}(\ov t, \ov \bx^{-S}) - V^{N,\ov K- |S|}(\ov t, \ov \by^{-S}) - \Big( \lambda(\ov t) + \frac{C_0 (\ov K - |S|)}{N} \Big) \frac{1}{N} \sum_{i \notin S} \big( |\ov x^i - \ov y^i|^2 + \eps \big)^{1/2}.
     \end{split}
     \ee
Let $C_{\lip, \Psi}$ be  the Lipschitz constant of $\Psi$.  Then, rearranging \eqref{vnkminus}, using the equation for $V^{N,K}$ and, in particular, the fact that it lies above the obstacle in \eqref{stopping.gen2}, and then the Lipschitz continuity of $\Psi$, we obtain
     \begin{align*}
         V^{N,\ov K}(\ov t, \ov \by) &< V^{N,\ov K - |S|}(\ov t, \ov \by^{-S}) + V^{N,K}(\ov t, \ov \bx) - V^{N, \ov K - |S|}(\ov t, \ov \bx^{-S}) 
         \\
         &\quad - \frac{C_0 |S|}{N} \frac{1}{N} \sum_{i = 1}^{\ov K} \big( |\ov x^i - \ov y^i|^2 + \eps \big)^{1/2} - \lambda(\ov t) \frac{1}{N} \sum_{i \in S} \big(|\ov x^i - \ov y^i|^2 + \eps)^{1/2}
         \\
         &\leq V^{N, \ov K - |S|}(\ov t, \ov \by^{-S}) + \frac{1}{N} \sum_{i \in S} \Psi(\ov x^i,m_{\ov{\bx}}^{N,K})
         \\
         &\quad - \frac{C_0 |S|}{N} \frac{1}{N} \sum_{i = 1}^{\ov K} \big( |\ov x^i - \ov y^i|^2 + \eps \big)^{1/2} - \lambda(\ov t) \frac{1}{N} \sum_{i \in S} \big(|\ov x^i - \ov y^i|^2 + \eps)^{1/2}
         \\
         &\leq V^{N,\ov K - |S|}(\ov t, \ov \by^{-S}) +  \frac{1}{N} \sum_{i \in S} \Psi( \ov y^i, m_{\ov \by}^{N,K}) + C_{\lip, \Psi} \Big(\frac{1}{N} \sum_{i \in S} |\ov x^i- \ov y^i| + {\frac{|S|}{N} \frac{1}{N}} \sum_{i = 1}^N |\ov x^i - \ov y^i| \Big)
         \\
         &\quad - \frac{C_0 |S|}{N} \frac{1}{N} \sum_{i = 1}^{\ov K} \big( |\ov x^i - \ov y^i|^2 + \eps \big)^{1/2} - \lambda(\ov t) \frac{1}{N} \sum_{i \in S} \big(|\ov x^i - \ov y^i|^2 + \eps)^{1/2}. 
     \end{align*}
    Choosing  $\lambda$ and $C_0$ in such a way that 
     \begin{align*}
         \lambda(t) \geq C_{\lip,\Psi}, \ \text{for all } t\in [0,T] \ \ \text{and} \ \  C_0 \geq C_{\lip,\Psi}, 
     \end{align*}
      we find that,   for all $S \subset [\ov K]$, 
     \begin{align*}
          V^{N,\ov K}(\ov t, \ov \by) < V^{N, \ov K - |S|}(\ov t, \ov \by^{-S}) + \frac{1}{N} \sum_{i \in S} \Psi(\ov y^i, m_{\ov \by}^{N,K}).
     \end{align*}
     In particular, this means that $V^{N,\ov K}$ is smooth in a neighborhood of $(\ov t, \ov \by)$, and the equation is satisfied there, that is, 
     \begin{align*}
         - \partial_t V^{N, \ov K}(\ov t, \ov \by) - \sum_{i = 1}^{\ov K} \Delta_{x^i} V^{N,\ov K}(\ov t,\ov \by) + \frac{1}{N} \sum_{i = 1}^{\ov K} H\bigg(\ov y^i, \Big( \lambda(\ov t) + \frac{C_0 \ov K}{N} \Big) \frac{(\ov x^i - \ov y^i)}{\big(|\ov x^i - \ov y^i|^2 + \eps\big)^{1/2}}, m_{\ov \by}^{N,\ov K} \bigg) = 0.
     \end{align*}
    For the subsolution at $(\ov{t},\ov \bx)$, we argue as if $V^{N,\ov K}$ is $C^{1,2}$ in a neighborhood of $(\ov t, \ov \bx)$ to simplify the presentation; the general case can be treated by using a standard viscosity solutions argument as in, for example, \cite[Theorem 8.3]{usersguide}. It follows that 
     \begin{align*}
         - \partial_t V^{N,\ov K}(\ov t,\ov \bx) - \sum_{i = 1}^{\ov K} \Delta_{x^i} V^{N,\ov K}(\ov t,\ov \bx) + \frac{1}{N} \sum_{i = 1}^{\ov K} H\bigg(\ov x^i, \Big( \lambda(\ov t) + \frac{C_0 \ov K}{N} \Big) \frac{(\ov x^i - \ov y^i)}{\big(|\ov x^i - \ov y^i|^2 + \eps\big)^{1/2}} , m_{\ov \bx}^{N,\ov K} \bigg) \leq 0.
     \end{align*}
     Subtracting the two equations and using that
\[
         \partial_t V^{N, \ov K}(\ov t, \ov \bx) - \partial_t V^{N, \ov K}(\ov t, \ov \by) = \lambda'(\ov t) \frac{1}{N} \sum_{i  = 1}^{\ov K} \big(|\ov x^i - \ov y^i|^2 + \eps \big)^{1/2} \ \text{and} \  D_{x^ix^i} V^{N,\ov K}(t,\bx) \leq  D_{x^ix^i} V^{N,\ov K}(t,\by),
         \] 
     we find, using \eqref{phl1}, that 
     \begin{align*}
         &- \lambda'(\ov t) \frac{1}{N} \sum_{i = 1}^{\ov K} \big(|\ov x^i - \ov y^i|^2 + \eps\big)^{1/2} 
         \\
         &\qquad \leq  \frac{1}{N} \sum_{i = 1}^{\ov K} \bigg[ H\bigg(\ov y^i, \Big( \lambda(\ov t) + \frac{C_0 \ov K}{N} \Big) \frac{(\ov x^i - \ov y^i)}{\big(|\ov x^i - \ov y^i|^2 + \eps\big)^{1/2}}, m_{\ov \by}^{N,\ov K} \bigg) 
         \\
         &\qquad \qquad \qquad \qquad \qquad  - H\bigg(\ov x^i, \Big( \lambda(\ov t) + \frac{C_0 \ov K}{N} \Big) \frac{(\ov x^i - \ov y^i)}{\big(|\ov x^i - \ov y^i|^2 + \eps\big)^{1/2}}, m_{\ov \bx}^{N,\ov K} \bigg) \bigg] 
         \\
         &\qquad \leq  \frac{C(1 + \lambda(\ov t)) }{N} \sum_{i = 1}^{\ov K} \frac{|\ov x^i - \ov y^i|^2}{\big(|\ov x^i - \ov y^i|^2 + \eps\big)^{1/2}}{+
         \frac{C}{N}\sum_{i=1}^K |\ov x^i - \ov y^i|(1+ \lambda(\ov t)) }
         \\
         &\qquad \leq \frac{C(1 + \lambda(\ov t)) }{N} \sum_{i = 1}^{\ov K} \big(|\ov x^i - \ov y^i|^2 + \eps \big)^{1/2},
     \end{align*}
where $C$ depends on $H$ and on the constant $C_0$, which has already been fixed and depends only on $C_{\lip,\Psi}$. We reach a contradiction if we choose $\lambda$ such that 
     \begin{align*}
         - \lambda'(t) > C(1 + \lambda(t)).
     \end{align*}
     In conclusion, we can choose a constant $C_0$ and a smooth function $\lambda$ independent of $N$ such that, for all $\eps > 0$, any optimizer $(\ov t, \ov K, \ov \bx, \ov \by)$ of \eqref{opt.lipbound} satisfies $\ov{t} = T$. Thus the bound \eqref{epsbound} holds for all $\eps > 0$. Sending $\eps \to 0$ completes the proof. 
 \end{proof}

In the next two steps, we introduce a new distance $\rho$ on $\sub$, which, according to Lemma~\ref{lem.dequrho}, is 
equivalent to the distance $\bd$, and, as shown in Lemma~\ref{lem.wminus.discrete}, is  convenient to estimate distances between empirical measures. 
\vs
We define, for $m,n \in \sub $ with $m(\T^d) \leq n(\T^d)$, 
 \begin{align*}
     \rho(m,n) = n(\T^d) - m(\T^d) + \inf \big\{ \bd(m , n') : n' \leq n, \,\, n'(\T^d) = m(\T^d) \big\}.
 \end{align*}
 If $n(\T^d) \leq m(\T^d)$, $\rho$ is defined in a symmetric way. We show first that $\rho$ is equivalent to $\bd$.

 \begin{lemma}\label{lem.dequrho}
     There is a constant $C$ such that, for $m,n \in \sub$,
     \begin{align*}
         \bd(m , n) \leq \rho(m,n) \leq 3  \bd(m , n).
     \end{align*}
 \end{lemma}

 \begin{proof}
    Suppose without loss of generality that $m(\T^d) \leq n(\T^d)$. For any test function $\phi$ with $\|\phi\|_{1,\infty} \leq 1$, and any $n' \leq n$ with $n'(\T^d) = m(\T^d)$, we have 
    \begin{align*}
        \int \phi d(n - m) &= \int \phi d(n' - m) + \int \phi d(n - n') \leq \bd(n',m) + \|\phi\|_{\infty} (n-n')(\T^d) 
        \\
        &\leq \bd(n',m) + (n-m)(\T^d). 
    \end{align*}
    Taking first the sup over $\phi$ and then the  inf over $n'$, we get $\bd(m,n)\leq \rho(m,n)$. 
  
    Next, we note that, for any signed measure $\mu$ on $\T^d$, $\|\mu\|_{(W^{1,\infty})^*} \leq |\mu|(\T^d)$, and so, for any $n'$ with $n'(\T^d) = m(\T^d)$ and $n' \leq n$, 
    \begin{align*}
        \bd(m,n') & + (n-m)(\T^d) = \|m - n + (n-n') \|_{(W^{1,\infty})^*} + (n-m)(\T^d) 
        \\
        &\leq \|m-n\|_{(W^{1,\infty})^*} + \|n-n'\|_{(W^{1,\infty})^*} + (n-m)(\T^d)
        \\
        &\leq \|m-n\|_{(W^{1,\infty})^*} + 2(n-m)(\T^d) \leq 3 \|m-n\|_{(W^{1,\infty})^*}=3 \bd(m,n). 
    \end{align*}
 \end{proof}

We show next that $\rho$ can be easily estimated on empirical measures.

 \begin{lemma} \label{lem.wminus.discrete}
     There is a constant $C$ such that,  for all $N \in \N$ and $K,M$ with $K + M \leq N$, $\bx \in (\T^d)^K$ and $\by \in (\T^d)^{K+M}$, 
     \begin{align*}
         \rho(m_{\bx}^{N,K}, m_{\by}^{N,K+M}) \leq \frac{M}{N} + \inf_{\bz \in (\T^d)^K, \, m_{\bz}^{N,K} \leq m_{\by}^{N,K+M}} \frac{1}{N} \sum_{i = 1}^K |x^i - z^i| \leq \text{diam}(\T^d) \rho(m_{\bx}^{N,K}, m_{\by}^{N,K+M}).
     \end{align*}
 \end{lemma}

 \begin{proof}
 From the definition of $\rho$, we see that it suffices to consider the case $M = 0$. That is, we need to show that, for $K \leq N$ and $\bx,\by \in (\T^d)^K$, 
 \begin{align} \label{empricialsufficient}
    \bd(m_{\bx}^{N,K}, m_{\by}^{N,K}) \leq \inf_{\sigma} \frac{1}{N} \sum_{i = 1}^K |x^i - y^{\sigma(i)}| \leq \text{diam}(\T^d)\rho(m_{\bx}^N, m_{\by}^{N,K+M}), 
 \end{align}
with the infimum taken over permutations $\sigma$ of $\{1,...,K\}$. For this, we first note that 
\begin{align} \label{d1.empirical}
    \inf_{\sigma} \frac{1}{N} \sum_{i = 1}^K |x^i - y^{\sigma(i)}| = \frac{K}{N} \inf_{\sigma} \frac{1}{K} \sum_{i = 1}^K |x^i - y^{\sigma(i)}|=\frac{K}{N} \bd_1\Big(\frac{1}{K} \sum_{i = 1}^K \delta_{x^i}, \frac{1}{K} \sum_{i = 1}^K \delta_{y^i} \Big).
\end{align}
It is clear that, for two measures $m$ and $n$ with equal mass, 
\be
\begin{split} \label{dandd1comp}
   & \bd_1(m,n) = \sup_{\phi \text{ 1-Lipschitz}, \, \phi(0) = 0}  \int \phi d(m-n)\\
   &\leq \text{diam}(\T^d) \sup_{\| \phi \|_{1,\infty} \leq 1} \int \phi d(m-n) = \text{diam}(\T^d) \bd(m,n).
\end{split}
\ee
Combining \eqref{d1.empirical} and \eqref{dandd1comp} gives \eqref{empricialsufficient}, and completes the proof.
 \end{proof}

The next step towards the proof of estimate~\eqref{vnklip} in Theorem~\ref{thm.main.lip} is to  estimate the distance between $V^{N,K}$ and $V^{N,K-1}$. 

 \begin{lemma} \label{lem.removingmass}
Assume  \eqref{summarize}. Then there exists  a constant $C$ such that, for each $N \in \N$, $K \in \{1,...,N\}$,  $(t,\bx) \in [0,T] \times (\T^d)^K$ and $i = 1,...,K$, we have 
     \begin{align*}
         \Bigl| V^{N,K-1}(t,\bx^{-i}) - V^{N,K}(t,\bx)\Bigr| \leq C/N. 
     \end{align*}
 \end{lemma}

 \begin{proof} Since $V^{N,K}$ is $\Psi$-non-increasing, we have 
 \begin{align*}
     V^{N,K}(t,\bx) \leq V^{N,K-1}(t,\bx^{-i}) + \frac{1}{N} \Psi(x^i,m_{\bx}^N) \leq V^{N,K-1}(t,\bx^{-i}) + \frac{\|\Psi\|_{\infty}}{N}.
 \end{align*}
We will prove the other inequality by showing the existence of a constant  $\lambda>0$ such that, for any $N$, $K$, $i$, $t$ and $\bx$, 
 $$
 V^{N,K-1}(t, \bx^{-i})-V^{N,K}(t, \bx) -  \frac{\lambda}{N} (1+\frac{K}{N})(T-t+1) \leq 0.
 $$
Arguing by contradiction, we assume that, for some $N>0$,  the map
 $$
(M, k, t,\bx)\to  V^{N,M-1}(t, \bx^{-k})-V^{N,M}(t,\bx)- \frac{\lambda}{N} (1+\frac{M}{N})(T-t+1)
 $$
 has a positive maximum $\mathcal M$  over $1\leq M\leq N$, $1\leq k\leq M$, and $(t,\bx)\in [0,T]\times (\T^d)^M$.  
 \vs
 
Denote by  $(K,i,\bar t, \ov{x})$  a maximum point. We first claim that $(\bar t, \ov{\bx})$ belongs to the set 
           $$
     \mathcal O = \Big\{ (t,\bx) \in [0,T]\times (\T^d)^K :  V^{N,K}(t,\bx)< V^{N,K-|S|}(t,\bx^{-S})  + \frac{1}{N} \sum_{i \in S} \Psi(x^i,m_{\bx}^{N,K}), \; \forall S\subset  [K], \; S\neq \emptyset \Big\}.
     $$
 Indeed, otherwise, there exists $S\subset  [K]$ with $|S|\geq 1$ such that 
 \be\label{defSkelsrd}
 V^{N,K}(\bar t,\ov\bx)= V^{N,K-|S|}(\bar t,\ov\bx^{-S})  + \frac{1}{N} \sum_{j \in S} \Psi(\bar x^j,m_{\ov{\bx}}^{N,K}) . 
 \ee
If $i\notin S$, then, by the optimality of $(K,i,\bar t, \ov{x})$,  we have 
\begin{align*}
& V^{N,K-|S|-1}(\bar t, \ov{\bx}^{-(S\cup\{i\})})-V^{N,K-|S|}(\bar t,\ov{\bx}^{-S})- \frac{\lambda }{N}(1+\frac{K-|S|}{N}) (T-\bar t+1) \\
& \qquad\qquad\qquad  \leq V^{N,K-1}(\bar t, \ov{\bx}^{-i})-V^{N,K}(\bar t,\ov{\bx})- \frac{\lambda}{N} (1+\frac{K}{N})(T-\bar t+1).
\end{align*}
Using \eqref{defSkelsrd},  the  $\Psi-$monotonicity of $V^{N,K}$ for the second inequality, and  the fact that, for any $K,i$ and $\bx$,  $\bd(m_{\bx^{-i}}^{N,K-1}, m_{\bx}^{N,K}) = \frac{1}{N}$, and denoting by  $\text{Lip}(\Psi)$ the Lipschitz constant of $\Psi$ in $m$ with respect to the metric $\bd$, we find 
\begin{align*}
& V^{N,K-|S|-1}\big(\bar t, \ov{\bx}^{-(S\cup\{i\})}\big)  \leq V^{N,K}(\bar t, \ov{\bx}^{-i})-\frac{1}{N} \sum_{j \in S} \Psi(\bar x^j,m_{\ov{\bx}}^{N,K})  -\lambda  \frac{|S|}{N^2}(T - \ov t - 1) \\
& \qquad  \leq V^{N,K-|S|-1}\big(\bar t, \ov{\bx}^{-(S\cup\{i\})}\big)   +\frac{1}{N} \sum_{j\in S} \Psi(\bar x^j,m_{\ov{\bx}^{-i}}^{N,K-1})-\frac{1}{N} \sum_{j \in S} \Psi(\bar x^j,m_{\ov{\bx}}^{N,K})- \lambda\frac{|S|}{N^2} (T-\bar t+1)
\\
& \qquad  \leq V^{N,K-|S|-1}\big(\bar t, \ov{\bx}^{-(S\cup\{i\})}\big) + \text{Lip}(\Psi) \frac{|S|}{N} \bd\big( m_{\bx^{-i}}^{N,K-1}, m_{\bx}^{N,K}\big) - \lambda \frac{|S|}{N^2} (T - \ov t +1)
\\
&\qquad  = V^{N,K-|S|-1}\big(\bar t, \ov{\bx}^{-(S\cup\{i\})}\big) + \text{Lip}(\Psi) \frac{|S|}{N^2} - \lambda \frac{|S|}{N^2} (T- \ov t + 1), 
\end{align*}
which is  a contradiction provided that we choose $\lambda > \text{Lip}(\Psi)$.
\vs

We assume now that $i\in S$. Then, using the  $\Psi-$monotonicity of $V^{N, K-1}$ and the fact that $\bd(m_{\bx^{-i}}^{N,K-1}, m_{\bx}^{N,K}) = \frac{1}{N}$, we have
\begin{align*}
V^{N,K-1}(\bar t, \ov{\bx}^{-i}) &  \leq V^{N,K-|S|}(\bar t, \ov{\bx}^{-S}) + \frac{1}{N} \sum_{j\in S\backslash\{i\}} \Psi(\bar x^j, m^{N,K-1}_{\bar \bx^{-i}}) \\
& \leq V^{N,K-|S|}(\bar t, \ov{\bx}^{-S}) + \frac{1}{N} \sum_{j\in S\backslash\{i\}} \Psi(\bar x^j, m^{N,K}_{\bar \bx}) + \frac{|S|}{N^2} \text{Lip}(\Psi). 
\end{align*}
By \eqref{defSkelsrd}, this implies that 
$$
V^{N,K-1}(\bar t, \ov{\bx}^{-i}) \leq V^{N,K}(\bar t,\ov\bx) -\frac{1}{N} \Psi(\bar x^i,  m^{N,K}_{\bar \bx})+ \frac{|S|}{N^2} \text{Lip}(\Psi).
$$
But then 
$$
\mathcal M= V^{N,K-1}(\bar t, \ov{\bx}^{-i})- V^{N,K}(\bar t,\ov\bx) -\frac{\lambda}{N}(1+\frac{K}{N})(T-\bar t+1) \leq - \frac{\lambda}{N}(1+\frac{K}{N})+ \frac{\|\Psi\|_\infty}{N}  + \frac{|S|}{N^2} \text{Lip}(\Psi), 
$$
and the right-hand side is negative as soon as $\lambda$ is larger than $\|\Psi\|_\infty + \text{Lip}(\Psi)$. This contradicts the assumption that $\mathcal M$ is positive and proves that $(\bar t, \ov{\bx})$ belongs to ${\mathcal O}$. 
 \vs
 
Next we show that we cannot have $\bar t <T$. Arguing again by contradiction, we note that, since  $(\bar t, \bar x)\in \mathcal O$ with $\bar t<T$, the map $V^{N,K}$ satisfies \eqref{stopping.gen2} with an equality  in a neighborhood of $(\bar t, \ov{\bx})$ and thus is of class $C^{1,2}$  in this neighborhood by classical parabolic regularity. 
\vs
We now view $\ov{x}^i$ as fixed, and note that the map
\begin{align*}
    [0,T] \times (\T^d)^{K-1} \ni (t,\by) \mapsto V^{N,K-1}(t,\by) - V^{N,K}\big(t, \by \oplus_i \ov{x}^i\big) - \frac{\lambda}{N}\Big(1 + \frac{K}{N} \Big)(T - t + 1)
\end{align*}
has a maximum at $(\ov{t}, \ov{\bx}^{-i})$, where we define $\by \oplus_i \ov{x}^i$ to be the element of $(\T^d)^K$ obtained by inserting $\ov{x}^i$ at slot $i$, i.e. 
\begin{align*}
    \by \oplus_i \ov{x}^i = (y^1,...,y^{i-1}, \bar x^i, y^i,...,y^{K-1}). 
\end{align*}
We now use the fact that $V^{N,K-1}$ is a global subsolution of \eqref{stopping.gen2} (with $K$ replaced by $K-1$). Setting $\theta= \frac{\lambda }{N}(1+\frac{K}{N})$ to simplify the expressions, we find that 
     \begin{align*}
  0& \geq       \theta- \partial_t V^{N,K} (\bar t,\ov{\bx})- \sum_{j \in [K]\backslash\{ i\} } \Delta_{x^j} V^{N,K} (\bar t,\ov{\bx})+ \frac{1}{N} \sum_{j \in [K]\backslash\{ i\}}H\big( \bar x^j,N D_{x^j} V^{N,K}(\bar t,\ov{\bx}), m_{\ov{\bx}^{-i}}^{N,K-1} \big).
     \end{align*}
  Plugging the equation for $V^{N,K}$ into this inequality, we obtain
     \begin{align*}
0  & \geq  \theta + \Delta_{x^i} V^{N,K} (\bar t,\ov{\bx}) + \frac{1}{N} \sum_{j \in [K]\backslash\{ i\}}  H\big( \bar x^j,N D_{x^j} V^{N,K}(\bar t,\ov{\bx}), m_{\ov{\bx}^{-i}}^{N,K-1} \big)\\
  & \qquad -  \frac{1}{N} \sum_{j=1}^K H\big( \bar x^j,N D_{x^j} V^{N,K}(\bar t,\ov{\bx}), m_{\ov{\bx}}^{N,K} \big)  .
     \end{align*}
Lemma \ref{lem.derivativebound2} yields that $\|N D_{x^j} V^{N,K}\|_\infty$ is bounded independently of $N$, $K$ and $j$. 
Moreover, by the optimality of $(\bar t, \ov{\bx})$, we have $D_{x^i} V^{N,K}(\bar t, \ov{\bx})=0$ and $D^2_{x^ix^i} V^{N,K}(\bar t, \ov{\bx})\geq0$.
 \vs
Using \eqref{phl1} and  the fact that $\bd(m_{\ov{\bx}^{-i}}^{N,K-1}, m_{\ov{\bx}}^{N,K}) = 1/N$, we obtain 
   \begin{align*}
  0&  \geq \theta  + \frac{1}{N} \sum_{j \in [K]\backslash\{ i\}}  \Bigl( H\big( \bar x^j,N D_{x^j} V^{N,K}(\bar t,\ov{\bx}), m_{\ov{\bx}^{-i}}^{N,K-1} \big)- H\big( \bar x^j,N D_{x^j} V^{N,K}(\bar t,\ov{\bx}), m_{\ov{\bx}}^{N,K} \big)\Bigr) \\
  & \qquad \hspace{5cm}  - \frac{1}{N} H\big( \bar x^i,0, m_{\ov{\bx}}^{N,K} \big)\\
  &   \geq \frac{\lambda }{N}(1+\frac{K}{N})  - \frac{C (K-1)}{N^2}- \frac{C}{N},  
     \end{align*}
     \vs
 which is  a contradiction if $\lambda$ is large enough. 
  \vs
  We have now shown that $\ov{t} = T$. To complete the proof, we note that, in view of the terminal condition for the $V^{N,K}$ and Lemma \ref{lem.envelopeLip},
  \begin{align*}
  \mathcal M &= V^{N,K-1}(T, \ov{\bx}^{-i})- V^{N,K}(T,\ov\bx) -\frac{\lambda}{N}(1+\frac{K}{N}) = G_{\Psi}^{N,K-1}(\ov \bx^{-i})-G_{\Psi}^{N,K}(\ov \bx)- \frac{\lambda}{N}(1+\frac{K}{N}) \\
  &\qquad \qquad \qquad  \leq \frac{C}{N} -\frac{\lambda}{N}(1+\frac{K}{N}), 
  \end{align*}
  the right-hand side being negative if $\lambda$ is larger than the constant $C$ appearing in Lemma \ref{lem.envelopeLip}. So we find again a contradiction and $\mathcal M$ cannot be positive. 

 \end{proof}

In order to infer time-regularity from space regularity, we need the following lemma. Since its proof follows from a standard verification argument together with the Lipschitz regularity in Lemma \ref{lem.derivativebound2}, it is omitted.

 \begin{lemma} \label{lem.partialdp}
Assume  \eqref{phl1}, \eqref{phl2} and \eqref{phl3}. Then there exists a constant $C$ such that, for each $N \in \N$, $K \in \{1,...,N\}$,  $(t_0,\bx_0) \in [0,T) \times (\T^d)^K$, and  $h \in (0,T-t_0)$, there exists an admissible control $(\balpha, \btau)$ for the problem defining $V^{N,K}$ such that
     \begin{align*}
         \|\alpha^i\|_{\infty} \leq C \quad \text{ for all } i = 1,...,K,
    \end{align*}
    as well as 
    \begin{align*}
     V^{N,K}(t_0, \bx_0) &= \E\bigg[ V^{N,K-|S|}\big(t_0 + h, \bX_{t_0 + h}^{-S} \big) + \int_{t_0}^{t_0 + h} \frac{1}{N} \sum_{i = 1}^K L\big(X_t^i, \alpha_t^i, m_t \big) 1_{t \leq \tau^i} dt + \frac{1}{N} \sum_{i \in S} \Psi\big(X_{\tau^i}^i, m_{\tau^i-} \big) \bigg ], 
     \end{align*}
     where 
     \begin{align*}
      m_t  = \frac{1}{N} \sum_{i = 1}^K \delta_{X_t^i} 1_{t < \tau^i}, \,\, m_{t-} = \frac{1}{N} \sum_{i = 1}^K \delta_{X_t^i} 1_{t \leq \tau^i} \ \ \text{and} \ \  S = \{i : \tau^i < t_0  + h\}. 
     \end{align*}
 \end{lemma}

We now complete the proof of the regularity of $V^{N,K}$.

\begin{proof}[Proof of the estimate \eqref{vnklip} in Theorem~\ref{thm.main.lip}]
    Combining Lemma~\ref{lem.derivativebound2} and Lemma~\ref{lem.removingmass}, we deduce that,  for all $K \leq M$, $\bx \in (\T^d)^M$, and  $\by \in (\T^d)^K$, and for any $\bx' \in (\T^d)^K$ with $m_{\bx'}^{N,K} \leq m_{\bx}^{N,M}$, 
    \begin{align*}
        |V^{N,M}(t,\bx) - V^{N,K}(t,\by)| &\leq \frac{C(M-K)}{N} + |V^{N,K}(t,\bx') - V^{N,K}(t,\by)|\\
         &\leq \frac{C(M-K)}{N} + \frac{C}{N} \sum_{i = 1}^K |x^{'i} - y^i|. 
    \end{align*}
    Taking an infimum over $\bx'$ and applying Lemma \ref{lem.wminus.discrete} completes the proof of the regularity in the space variable. 
\vs
   We address next the time regularity of $V^{N,K}$.  Recall that $V^{N,K}$ satisfies, in the  viscosity sense and  globally on $(0,T) \times (\T^d)^K$, 
   \begin{align*}
       - \partial_t V^{N,K} - \sum_{i = 1}^K \Delta_{x^i} V^{N,K} + \frac{1}{N} \sum_{i = 1}^K H\big(x^i,ND_{x^i} V^{N,K}, m_{\bx}^{N,K}\big) \leq 0.
   \end{align*}
   In light of Lemma \ref{lem.derivativebound2} and the fact that $H(x,p,m) \lesssim 1 + |p|^2$, there is a constant $C$, which is independent of $N$ and $K$, such that $V^{N,K}$ satisfies in the viscosity sense 
   \begin{align*}
      - \partial_t V^{N,K} - \sum_{i = 1}^K \Delta_{x^i} V^{N,K} \leq C.
   \end{align*}
   Thus, for any $(t_0,\bx_0)\in [0,T)\times (\T^d)^K$ and $h\in (0,T-t_0)$, we have 
   \begin{align} \label{comp.vnk.w}
       V^{N,K}(t_0,\bx_0) \leq W(t_0,\bx_0),
   \end{align}
   where $W$ solves 
   \begin{align*}
       - \partial_t W - \sum_{i = 1}^K \Delta_{x^i} W = 0 \text{ in } (t_0,t_0+h) \times (\T^d)^K, \quad W(t_0+h,\cdot) = V^{N,K}(t_0 + h,\cdot). 
   \end{align*}
   Now, again using Lemma \ref{lem.derivativebound2}, we find that there is a constant $C$ which is independent of $N,K$, $t_0$ and $h$ such that $$W(t_0,\bx_0) = \E\Big[ V^{N,K}\big(t_0 + h, \bx_0 + W_h^i\big)\Big] \leq V^{N,K}(t_0 + h, \bx_0) + \frac{C}{N} \sum_{i = 1} \E[|W_h^i|] \leq V^{N,K}(t_0 + h, \bx_0) + C \sqrt{h}.$$
   It then follows from  \eqref{comp.vnk.w} that 
$$
V^{N,K}(t_0, \bx_0) \leq V^{N,K}(t_0+h, \bx_0) + Ch^{1/2}.
$$
We now prove the converse inequality, that is, the existence of a constant $C$ which is independent of $(t_0,\bx_0)$, $K$ and $N$ such that 
\be\label{cvineqlkjsndfcv}
V^{N,K}(t_0, \bx_0) \geq V^{N,K}(t_0+h, \bx_0) - Ch^{1/2}.
\ee
Let $(\balpha, \btau)$, $\bX_t$, $m_t$, and $m_{t-}$ be as in the statement of Lemma \ref{lem.partialdp}, and set 
\begin{align*}
    \wt{m}_t = \frac{1}{N} \sum_{i = 1}^K \delta_{X_t^i}. 
\end{align*}
Then, using the fact that $\Psi$ is non-increasing in $m$, the boundedness of $\alpha^1,...,\alpha^K$ and the $\Psi$-monotonicity of $(V^{N,K})_{K = 1,...,N}$, we find 
\be \label{dpptimebound}
\begin{split}
    V^{N,K}(t_0, \bx_0) &= \E\bigg[ V^{N,K-|S|}\big(t_0 + h, \bX_{t_0 + h}^{-S} \big) + \int_{t_0}^{t_0 + h} \frac{1}{N} \sum_{i = 1}^N L\big(X_t^i, \alpha_t^i, m_t\big) 1_{t \leq \tau^i} dt \\
    &\hskip1in + \frac{1}{N} \sum_{i \in S} \Psi\big(X_{\tau^i}^i, m_{\tau^i-} \big) \bigg ]\\
    &\geq \E\bigg[ V^{N,K-|S|}\big(t_0 + h, \bX_{t_0 + h}^{-S} \big) + \frac{1}{N} \sum_{i \in S} \Psi\big(X_{\tau^i}^i, m_{\tau^i-} \big) \bigg ] - Ch\\
    &\geq \E\bigg[ V^{N,K}\big(t_0 + h, \bX_{t_0 + h} \big) - \frac{1}{N} \sum_{i \in S} \Psi\big(X_{t_0 + h}^i, \wt{m}_{t_0 + h}\big) + \frac{1}{N} \sum_{i \in S} \Psi\big(X_{\tau^i}^i, m_{\tau^i - } \big) \bigg] - Ch, 
\end{split}
\ee
Next, we note that, since  $\Psi$ is non-increasing and Lipschitz,  
\begin{align*}
    \Psi(X_{\tau^i}^i, m_{\tau^i}) \geq \Psi(X^i_{\tau^i}, \wt{m}_{\tau}^i) \geq \Psi(X^i_{t_0 + h}, \wt{m}_{t_0 + h}) - C \Big( |X_{\tau^i}^i - X_{t_0+h}^i| + \bd\big(\wt{m}_{\tau^i}, \wt{m}_{t_0 + h}\big) \Big).
\end{align*}
Returning to \eqref{dpptimebound} and using that, in view of bound on the  $\alpha^i$'s, we have $\E\big[\sup_{t_0 \leq s \leq t_0+ h} |X_s^i - x^i| \big] \leq C \sqrt{h}$, we obtain
\begin{align*}
    V^{N,K}(t_0, \bx_0) &\geq \E\Big[ V^{N,K}\big(t_0 + h, \bX_{t_0 + h} \big) - C\Big( \frac{1}{N} \sum_{i \in S}  |X_{\tau_i}^i - X_{t_0 + h}^i| + \bd(\wt{m}_{t_0 + h}, \wt{m}_{\tau}) \Big) \Big] - Ch 
    \\
    &\geq \E\Big[ V^{N,K}\big(t_0 + h, \bX_{t_0 + h} \big) \Big] - C (h + \sqrt{h})
 \geq V^{N,K}(t_0 + h, \bx) - C(h + \sqrt{h}). 
\end{align*}
The proof is now complete. 

\end{proof}

\section{The mean field problem} \label{sec.U}

We investigate here some  properties of the mean field control problem and its value function $U$.
 
\subsection*{Existence of minimizers and basic properties of the
limiting problem}

Recall that $\cA_{t_0,m_0}$ is the set of triples $(m,\alpha, \mu)$ where $m$ is a \cad path taking values in $\sub$, $\alpha$ is a measurable function, $\mu$ is a non-negative measure on $[t_0,T] \times \T^d$, which satisfy the equation \eqref{meqn0}. For the purposes of compactness arguments, it is easier to view $m$ as a measure on $[t_0,T] \times \T^d$ and  to facilitate this, we make use of the following lemma.

\begin{lemma}\label{lem.cdlag} Let $(t_0,m_0) \in [0,T) \times \sub$. Suppose that $(m',m'_T, \alpha, \mu)$ are such that $m'$ and $\mu$ are non-negative measures on $[t_0,T] \times \T^d$, $m'_T \in \sub$, and $\alpha : [t_0,T] \times \T^d \to \R^d$ is a measurable function satisfying 
\begin{align}\label{ord10}
    \int_{[t_0,T] \times \T^d} |\alpha(t,x)|^2 dm'(t,x) < \infty
\end{align}
and, for every test function $\phi \in C^{\infty}([t_0,T] \times \T^d)$, 
\be\label{meqn2}
\begin{split}
& \int_{\T^d} \phi(T,x) dm_T'(x) - \int_{\T^d} \phi(t_0,x) dm_0(x)\\ 
&= \int_{t_0}^T \int_{\T^d} \Big( \partial_t \phi(t,x) + \Delta \phi(t,x) + D\phi(t,x) \cdot \alpha(t,x) \Big)dm'(t, x) - \int_{[t_0,T] \times \T^d} \phi(t,x) d\mu(t,x).
\end{split}
\ee

Then there is a unique \cad path $[t_0,T] \ni t \mapsto m_t \in \sub$ such that  $(m,\alpha,\mu)\in \mathcal A_{t_0,m_0}$ and 
\begin{align*}
    dm'(t,x) = d m_t(x) dt, \ \ m_{t_0} \leq m_0 \ \text{and} \  m_T = m_T'.
\end{align*}
Moreover, if $(m_t)_{t_0 \leq t \leq T}$ is extended to all  $[0,T]$ by setting $m_t = m_0$ for $t < t_0$, and we define, for $t\in [0,T]$, $m_{t-} = \lim_{s \uparrow t} m_s,$
then, for each $t \in [t_0,T]$, we have 
\begin{align} \label{jump.char}
    m_{t-} - m_t = \mu(\{t\}, \cdot  ),  
\end{align}
where we denote by $\mu(\{t\} ,\cdot )$ the measure $A \mapsto \mu(\{t\} \times A )$. 
\end{lemma}

\begin{proof} Let  $(m',m'_T,\alpha,\mu)$ satisfy \eqref{ord10} and \eqref{meqn2}. 
It is known that $dm'=\rho dx dt$ with $\rho\in L^q([t_1,t_2]\times \T^d)$ for any $q\in [1, (d+2)')$ and any $t_0<t_1<t_2<T$,  where $(d+2)'$ indicates the conjugate exponent of $d+2$ (see for instance \cite[Theorem 2.2.1]{BKR}). Thus, we can define
\begin{align*}
    d\wt{m}_t(x) = \begin{cases}
        \rho(t,x) dx   \ \ \text{if} \ \  t \in [t_0,T), 
        \\
        dm'_T(x)  \ \ \text{if} \ \  t = T.
    \end{cases}
\end{align*}

We now claim that, for any $C^2$ test function $\phi : \T^d \to \R$, we have, for a.e. $t \in [t_0,T]$ and for $t = T$,
\be \label{phi.expansion}
\begin{split}
    \int_{\T^d} \phi(x) d\wt{m}_t(x) = \int_{\T^d} \phi(x) dm_0(x) &+ \int_{t_0}^t \int_{\T^d} \big( \Delta \phi(x) + D \phi(x) \cdot \alpha \big) d\wt{m}_t(x)dt \\
    &- \int_{[t_0,t] \times \T^d} \phi(x) d\mu(t,x).
\end{split}
\ee
Indeed, for $t = T$, this follows from the definition of $\wt{m}_T$ and \eqref{meqn2}. For $t < T$, we consider a test functions of the form $\phi_{\delta}(s,x) = \kappa_{\delta}(s) \phi(x)$, where 
$\kappa_{\delta} : [0,T] \to [0,1]$ is a smooth function satisfying 
\begin{align*}
    \kappa_{\delta}(s) = 1 \text{ for } 0 \leq s \leq t, \quad \kappa_{\delta}(s) = 0 \text{ for } s \geq t + \delta, \quad - \frac{2}{\delta} \leq \kappa'(s) \leq 0, \quad s \in [0,T].
\end{align*}
Then by \eqref{meqn2}, for $\delta < T-t$, we have 
\begin{align*}
    - \int_{\T^d} \phi(t_0,x) dm_0(x) &= \int_{t_0}^T \kappa(s)\bigg[ \int_{\T^d} \Big( \Delta \phi(x) + D\phi(x) \cdot \alpha(s,x) \Big) d\wt{m}_s(x) \bigg] ds
    \\
    &\qquad + \int_{t_0}^T \kappa'(s) \Big[ \int_{\T^d} \phi(x) d\wt{m}_s(x) \bigg] ds - \int_{[t_0,T] \times \T^d} \rho_{\delta}(s) \phi(x) d\mu(s,x)
    \\
    &= I_{\delta} + II_{\delta} + III_{\delta}. 
\end{align*}
Notice that $\rho_{\delta} \to 1_{[0,t]}$ pointwise, so by the bounded convergence theorem we find
\begin{align*}
    I_{\delta} \xrightarrow{\delta \downarrow{0}} \int_{t_0}^t \int_{\T^d} \Big(\Delta \phi(x) + D\phi(x) \cdot \alpha(s,x) \Big) d\wt{m}_s(x) ds, \quad III_{\delta} \xrightarrow{\delta \downarrow{0}} - \int_{[t_0,t] \times \T^d}  \phi(x) d\mu(s,x).
\end{align*}
Meanwhile, using the properties of $\kappa_{\delta}$, one can check that 
\begin{align*}
    II_{\delta} \xrightarrow{\delta \downarrow{0}} - \int_{\T^d} \phi(x) d\wt{m}_t(x)
\end{align*}
for any $t$ which is a Lebesgue point of the bounded and measurable function 
\begin{align*}
    t \mapsto \int_{\T^d} \phi(x) d\wt{m}_t(x).
\end{align*}
So \eqref{phi.expansion} holds for every such Lebesgue point, and thus for a.e. $t$, by the Lebesgue differentiation theorem.  

In particular, by considering a countable and dense subset of $C^2(\T^d)$, we deduce that there exists a set $A \subset [t_0,T]$ of full measure which contains $T$, such that, for each test function $\phi$ and each $t \in A$, $\eqref{phi.expansion}$ holds. 
\vs
Testing \eqref{phi.expansion} with $\phi\equiv 1$, we see that the map $M : A \to [0,1]$ given by $M(t) = \wt{m}_t(\T^d)$ satisfies, for all $t \in A$,  
\begin{align} \label{mprop}
    M(t) = m_0(\T^d) - \mu([t_0,t] \times \T^d), 
\end{align}
and, since the right-hand side of \eqref{mprop} easily extends to all of $[t_0,T]$, we can define  $M$ as a  \cad and non-increasing function on $[t_0,T]$ satisfying, for all $s,t \in [t_0,T]$ with $s < t$, 
\begin{align} \label{mincrememnts}
   M(s) - M(t) = \mu\big( (s,t] \times  \T^d \big). 
\end{align}
Coming back to \eqref{phi.expansion}, we see that, if $\phi$ is positive, then for $t,s \in A$ with $s < t$ and a constant $C$ depending on $\int_{t_0}^T \int_{\T^d}|\alpha|^2 d m' <\infty$, 
\begin{align}\label{eljkzhrsdnfg}
\int_{\T^d} \phi d\wt{m}_t &  \leq \int_{\T^d} \phi d\wt{m}_s +\int_s ^t \int_{\T^d}( \Delta \phi -D\phi\cdot \alpha) d\wt{m}_r dr 
\leq \int_{\T^d} \phi d\wt{m}_s +C (t-s)^{1/2} \|\phi\|_{C^2}.
\end{align}

Thus if $\xi:\R^d\to\R$ is a smooth and compactly supported non-negative  kernel and $\xi_\eta=\eta^{-d}\xi(\cdot/\eta)$, for any $\phi$ with $\|\phi\|_{W^{1,\infty}}\leq 1$ and $\phi\geq 0$,  we have
\begin{align*}
\int_{\T^d} \phi d(\wt{m}_t - \wt{m}_s) \, &= \int_{\T^d} \xi_\eta \ast \phi d(\wt{m}_t - \wt{m}_s) +\int_{\T^d} (\phi - \xi_\eta \ast \phi)d(\wt{m}_t - \wt{m}_s) \\
& \leq \int_{\T^d} \xi_\eta \ast \phi d(\wt{m}_t - \wt{m}_s) +2 \|\xi_\eta \ast \phi-\phi\|_\infty\\
& \leq C (t-s)^{1/2} \|\xi_\eta \ast \phi\|_{C^2} + 2 \eta\;  \leq \;  C (t-s)^{1/2}\eta^{-1} + 2 \eta,
\end{align*}
where the third inequality comes from \eqref{eljkzhrsdnfg} and the last one from the assumption on $\phi$. 
\vs
Choosing $\eta = (t-s)^{1/4}$, we find 
\be\label{erlfgjfgkjnl}
\sup_{\begin{array}{c} \|\phi\|_{W^{1,\infty}}\leq 1, \\ \phi\geq 0 \end{array}} \int_{\T^d} \phi (\wt{m}_t-\wt{m}_s) \leq C (t-s)^{1/4}.
\ee
Next, we notice that, for any $s ,t \in A$, $s < t$, we have 
\begin{align*}
\bd(\wt{m}_t,\wt{m}_s) & \leq  
\sup_{\begin{array}{c} \|\phi\|_{W^{1,\infty}}\leq 1 \end{array}} \int_{\T^d} \phi d\big(\wt{m}_t +M(s) \text{Leb}-M(t) \text{Leb} -\wt{m}_s \big) + M(s)-M(t) \\
& = \sup_{\begin{array}{c} \|\phi\|_{W^{1,\infty}}\leq 1 \end{array}} \int_{\T^d} (\phi +1) d\big(\wt{m}_t +M(s) \text{Leb}-M(t) \text{Leb} -\wt{m}_s \big) + M(s) - M(t) \\ 
& \leq 2 \sup_{\begin{array}{c} \|\phi\|_{W^{1,\infty}}\leq 1, \\ \phi\geq 0 \end{array}}  \int_{\T^d} \phi d\big(\wt{m}_t +M(s) \text{Leb}-M(t) \text{Leb} -\wt{m}_s \big)  + M(s)-M(t) \\ 
& \leq  C (t-s)^{1/4} + 3(M(s)-M(t)).
\end{align*}
\vs
Since $M$ is \cad, this shows that the limit $\lim_{s \downarrow t, s \in A} \wt{m}_s$ exists for any $t \in [t_0,T]$. It is then easy to check that the map 
\begin{align} \label{def.wtm}
    [t_0,T] \ni t \mapsto m_t \coloneqq \lim_{s \downarrow t, s \in A} \wt{m}_s
\end{align}
is \cad and agrees with $m_t$ on $A$, that is, $m_t = \wt{m}_t$ for $t \in A$. It follows that $dm_t(x) dt = dm'(t,x)$, and that $m_T = \wt{m}_T$ (since $T \in A$).
\vs
In addition, since $\wt{m}_t = m_t$ for $t \in A$, we find that 
\be \label{phi.expansion2}
\begin{split}
    \int_{\T^d} \phi(x) dm_t(x) = \int_{\T^d} \phi(x) dm_0(x) & + \int_{t_0}^t \int_{\T^d} \big( \Delta \phi(x)  + D \phi(x) \cdot \alpha \big) dm_t(dx)dt \\
    & - \int_{[t_0,t] \times \T^d} \phi(x) d\mu(t,x),
\end{split}
\ee
for each $t \in A$. By right continuity of $t \mapsto m_t$, we deduce that in fact \eqref{phi.expansion2} holds for all $t \in [t_0,T]$, and thus, for $t_0 \leq s < t \leq T$,
\begin{align} \label{phi.expansion.3}
    \int_{\T^d} \phi(x) dm_t(x) - \int_{\T^d} dm_s(x) = \int_{s}^t \int_{\T^d} \big( \Delta \phi(x)  + D \phi(x) \cdot \alpha \big) dm_t(dx)dt - \int_{(s,t] \times \T^d} \phi(x) d\mu(t,x).
\end{align}
Choosing $s = t - h$ and sending $h \to 0$ gives \eqref{jump.char}. Finally, combining \eqref{phi.expansion2} and \eqref{phi.expansion.3} with $s = t_0$ and $\phi \geq 0$ gives $m_{t_0} \leq m_0$, and uniqueness of the \cad representative $m$ is classical.

\end{proof}

In the next lemma we isolate an estimate on a solution $(m, \alpha, \mu)$ to \eqref{meqn2} established in \eqref{erlfgjfgkjnl} during the proof of Lemma \ref{lem.cdlag}. 

\begin{lemma}\label{lem.estisol} For any  $(m, \alpha,\mu) \in \cA_{t_0,m_0}$, there is a constant $C$, which  depends only  $\int_{t_0}^T\int_{\T^d}|\alpha|^2 m$, such that, for any $t_0\leq s\leq t\leq T$,  
$$
\sup_{\begin{array}{c} \|\phi\|_{W^{1,\infty}}\leq 1, \\ \phi\geq 0 \end{array}} \int_{\T^d} \phi (dm_t - dm_{s-}) \leq C (t-s)^{1/4}.
$$
\end{lemma}

We now discuss in the next proposition  the compactness of the set of solutions to \eqref{meqn2}. 

\begin{proposition}\label{cor.compactmn} Assume that the sequence $(m^n, \alpha^n,\mu^n) \in \cA_{t_0,m_0}$  has uniformly bounded energies, that is,
\be\label{cond.alphan}
\sup_n \int_{t_0}^T\int_{\T^d}|\alpha^n(t,x)|^2dm^n_t(x) <\infty. 
\ee
Then there exists a subsequence (denoted in the same way) and $(m, \alpha,\mu) \in \cA_{t_0,m_0}$ such that, weakly$-\star$,
\begin{align*}
    dm^n_t(x)  dt \to dm_t(x )dt, \quad  \alpha^n(t,x) dm^n_t(x) dt \to \alpha(t,x) dm_t(x) dt, \quad \mu^n \to \mu, \text{  and  } m_T^n \to m_T,
\end{align*} 
 and
   \be\label{lscL}
\liminf_n \int_{t_0}^T\int_{\T^d}L\big(x,\alpha^n(t,x),m^n_t)dm^n_t(x) dt \geq \int_{t_0}^T\int_{\T^d}L\big(x, \alpha(t,x), m_t) dm_t(x) dt.
\ee

In addition, along a converging subsequence, if a sequence $(t^n)_{n\in \N}$ in $[t_0,T]$ converges to $t$, then any cluster point $\rho$ (for the convergence of measures in $\sub$) of the $m^n_{t_n}$'s  satisfies  $m_{t} \leq \rho \leq m_{t-}$, and, thus, for almost every $t \in [t_0,T]$, $m_t^n \to m_t$. 
\end{proposition}

\begin{proof} Define 
\begin{align*}
    dm^{n,'}(t,x) = dm^n_t(x) dt \ \ \text{and} \ \ dE^n(t,x) = \alpha^n(t,x) dm^{n,'}(t,x),
\end{align*}
where $m^{n,'}$  and $E^n$ are  viewed respectively as a non-negative Borel measure on $[t_0,T] \times \T^d$ and  an $(\R^d)$-valued vector measure on $[t_0,T] \times \T^d$. 
\vs
We first note that the total variation of $m^{n,'}$, $E^n$, $\mu^n$ and $m^n_T$ are bounded uniformly in $n$. Indeed, the bound for  $m^{n,'}$ is straightforward, while the bound for $\mu^n$ is easily obtained by integrating \eqref{meqn2} in time-space . Finally, the bound on the total variation of $E^n$ is a consequence of the Cauchy-Schwarz inequality combined with  \eqref{cond.alphan}. 
\vs
Thus,  there exists a subsequence (denoted in the same way) and $(m',m'_T, E, \mu)$ such that, in the weak$-\star$ topology
\begin{align*}
    m^{n,'} \to m', \quad m^n_T \to m'_T, \quad E^n \to E, \quad \mu^n \to \mu. 
\end{align*}
Moreover, since \eqref{cond.alphan} holds, it is known that $E\ll m'$ (see for instance Theorem 5 in \cite{Rock71}). 
\vs

Let $\alpha$ be the Radon-Nikodym derivative of $E$ with respect to $m'$, and note that, since $(m',m_T',\alpha,\mu)$ satisfy \eqref{meqn2}, by Lemma \ref{lem.cdlag}, there exists a \cad function $m$ such that $(m,\alpha,\mu) \in \cA_{t_0,m_0}$, $dm'(t,x) = m_t(dx)dt$, and $m_T = m_T'$. 
\vs
Assume now that  $t^n\to t$ and let $\rho$ be a cluster point of $(m^n_{t_n})_{n\in \N}$. To simplify the notation, we argue as if the whole sequence $(m^n_{t_n})$ converges to $\rho$. 
\vs
Fix $\ep>0$ and, using the  right continuity of $m$, choose $h>0$ small enough such that $\bd(m_s,m_t)\leq \ep$ for $s\in (t, t+h)$. 
It follows from  Lemma \ref{lem.estisol} that, for any $s\in [t^n, t^n+h]$ and any $\phi\in W^{1,\infty}(\T^d)$ with $\phi\geq0$, 
$$
\int_{\T^d} \phi d m^n_s \leq \int_{\T^d} \phi d m^n_{t_n} + C\|\phi\|_{W^{1,\infty}}h^{1/4}.
$$
Fix a continuous nonnegative function $\xi=\xi(t)$ with support in $(t, t+h)$ and such that $\int_t^{t+h}\xi>0$, and,  for $n$ large enough, integrate the inequality above against $\xi$ and pass to the limit using the convergence of the  $m^n$'s to obtain
$$
\int_{t}^{t+h} \int_{\T^d} \phi(x) \xi(s) dm_s(x) ds \leq (\int_{t}^{t+h}\xi(s)ds ) \left( \int_{\T^d} \phi(x) d\rho(x) + C\|\phi\|_{W^{1,\infty}}h^{1/4}\right).
$$
As $\bd(m_s, m_t) \leq \ep$ for $s\in (t, t+h)$, we get, after dividing by $\int_{t}^{t+h}\xi(s) ds$, 
$$
\int_{\T^d} \phi dm_t - \ep \|\phi\|_{W^{1,\infty}} \leq  \int_{\T^d} \phi(x) d\rho(x) + C\|\phi\|_{W^{1,\infty}}h^{1/4}. 
$$
Letting $h\to 0$ and then $\ep\to 0$ we conclude that $\rho\geq  m_t$. The proof that $\rho\leq  m_{t-}$ can be obtained by symmetric arguments. 
\vs
By duality, we have, for any $\phi\in C^0([t_0,T]\times \T^d, \R^d)$,  
\begin{align*} 
\int_0^T \int_{\T^d} L(x,\alpha^n, m^n_t)dm^n_t(x)dt \geq 
- \int_0^T \int_{\T^d} \phi(t,x)\cdot dE^n(t, x) - \int_0^T \int_{\T^d} H(x, \phi(t,x), m^n_t)dm^n_t(x)dt
\end{align*}
Note that 
$$
\lim \int_{t_0}^T \int_{\T^d} \phi(t,x)\cdot dE^n(t, x) = \int_0^T \int_{\T^d} \phi(t,x)\cdot dE(t, x) ,
$$
while, for a.e. $t\in [t_0,T]$, $m^n_t$ converges to $m_t$ so that, by the continuity of $H$,  
$$
\lim  \int_{\T^d} H(x, \phi(t,x), m^n_t)dm^n_t(x) =  \int_{\T^d} H(x, \phi(t,x), m_t)dm_t(x). 
$$
Then using  dominated convergence we get 
$$
\lim \int_0^T \int_{\T^d} H(x, \phi(t,x), m^n_t)dm^n_t(x)dt = \int_0^T \int_{\T^d} H(x, \phi(t,x), m_t)dm_t(x)dt,
$$
and, hence, 
\begin{equation*}
\begin{split}
\liminf \int_0^T \int_{\T^d} L(x,\alpha^n, m^n_t)m^n_t(dx)dt \geq  &- \int_0^T \int_{\T^d} \phi(t,x) \cdot E(dt,dx)\\
& - \int_0^T \int_{\T^d} H(x, \phi(t,x), m_t)dm_t(x)dt.
\end{split}
\end{equation*}
As this holds for any $\phi\in C^0([t_0,T]\times \T^d, \R^d)$, we infer from the representation formula in Theorem~5 of \cite{Rock71} that 
\begin{align*}
&\liminf \int_{t_0}^T \int_{\T^d} L(x,\alpha^n, m^n_t)m^n_t(dx)dt 
\\
&\geq \sup_{\phi\in C^0} \Big(- \int_0^T \int_{\T^d} \phi(t,x) \cdot E(dt,dx) - \int_0^T \int_{\T^d} H(x, \phi(t,x), m_t)dm_t(x)dt\Big)\\
&=  \int_{t_0}^T \int_{\T^d} L(x,\alpha, m_t)dm_t(x)dt. 
\end{align*}

\end{proof}

In a series of lemmata, we next discuss various properties of the value function $U$. The first claim, which we state without a proof since it is a classical fact, is that $U$ satisfies the following dynamic programming property.

\begin{lemma}\label{dp}
Assume  \eqref{summarize}. For any $0\leq t_0\leq t_1\leq T$,  
\begin{align*}
U(t_0,m_0)&= \inf_{(\alpha, m,\mu) \in \cA_{t_0,m_0}} \bigg\{ \int_{t_0}^{t_1}  \int_{\T^d} L\big(x,\alpha(t,x),m_t \big) dm_t(x) dt
\nonumber \\
&\qquad \qquad \qquad \qquad \qquad \qquad  \nonumber  +\int_{[t_0,t_1] \times \T^d} \Psi(x,m_{t-})\mu(dt,dx) + U(t_1, m_{t_1})\bigg\}
\nonumber \\
&= \inf_{(\alpha, m,\mu) \in \cA_{t_0,m_0}} \bigg\{ \int_{t_0}^{t_1}  \int_{\T^d} L\big(x,\alpha(t,x),m_t \big) dm_t(x) dt
\nonumber \\
&\qquad \qquad \qquad \qquad  \qquad \qquad  \nonumber +\int_{[t_0,t_1) \times \T^d} \Psi(x,m_{t-})\mu(dt,dx) + U(t_1, m_{t_1-})\bigg\}.
\end{align*}
\end{lemma}
\vs
The second claim is the $\Psi-$non-increasing property.
 
\begin{lemma} Assume  \eqref{summarize}. The value function  $U$ is {\it $\Psi-$non-increasing}, i.e. satisfies \eqref{def.psinon-increasing}. 
\end{lemma}

\begin{proof} Fix $\ep>0$, let $(\alpha, m, \mu) \in \cA_{t_0,n_0}$ be $\ep-$optimal for $U(t_0,n_0)$ where $n_0 \leq m_0$, and  define $(\tilde m, \tilde \alpha, \tilde \mu)$ by $\tilde m_t = m_t$ for $t \geq t_0$, $\tilde m_{t_0 -} = m_0$, $\tilde \alpha= \alpha$ and $\tilde \mu= \mu + \delta_{t_0}(m_0-n_0)$. 
\vs
Since $(m,\alpha, \mu)$ solves \eqref{meqn}, for any smooth test function $\phi$ with a compact support in $[t_0,T)\times \T^d$, we have 
\begin{align*}
& \int_{\T^d} \phi(t_0,x)m_0(dx) +\int_{t_0}^T \int_{\T^d} (\partial_t \phi+ \Delta \phi +D\phi\cdot \alpha) \tilde m - \iint_{[t_0,T)\times \T^d} \phi \tilde \mu \\ 
& = \int_{\T^d} \phi(t_0,x)m_0(dx) +\int_{t_0}^T \int_{\T^d} (\partial_t \phi+ \Delta \phi +D\phi\cdot \alpha) m -  \iint_{[t_0,T)\times \T^d} \phi \mu - \int_{\T^d}\phi(t_0,x)(m_0-n_0)(dx) \\ 
&= 0.
\end{align*}
It follows that  $(\tilde m, \tilde \alpha, \tilde \mu) \in \cA_{t_0,m_0}$. Moreover, 
\begin{align*}
U(t_0,m_0)& \leq J_{t_0,m_0}(\tilde m, \tilde \alpha, \tilde \mu) = J_{t_0,n_0}(m,\alpha, \mu)+ \int_{\T^d} \Psi(x,m_0)(m_0-n_0)(dx) \\
& \leq U(t_0,n_0)+\ep +  \int_{\T^d} \Psi(x,m_0)(m_0-n_0)(dx).
\end{align*}
Letting $\ep\to 0$ yields \eqref{def.psinon-increasing} and, hence, the claim.

\end{proof}

Next we show the existence of minimizers.

\begin{prop}\label{prop.existsOS} Assume  \eqref{summarize}. For any initial condition $(t_0,m_0)\in [0,T)\times \sub$, there exists at least a minimizer for $U(t_0, m_0)$. 
\end{prop}

\begin{proof} Let $(m^n, \alpha^n, \mu^n)$ be a minimizing sequence for $U(t_0, m_0)$. Then, in view of the coercivity condition \eqref{hypL},  \eqref{cond.alphan} holds. 
\vs
By Proposition \ref{cor.compactmn}, there exists a subsequence (denoted in the same way) and $(m, \alpha,\mu) \in \cA_{t_0,m_0}$ such that the sequences $(m^n)_{n\in \N}$, $(\alpha^nm^n)_{n\in \N}$ and $(\mu^n)_{n\in \N}$ converge in measure to $m$, $\alpha m$ and $\mu$ respectively, with 
\be\label{lakejznrdfgkl010}
\liminf_{n \to \infty} \int_{t_0}^T\int_{\T^d}L(x,\alpha^n,m^n)m^n \geq \int_{t_0}^T\int_{\T^d}L(x, \alpha, m) m.
\ee
\vs
In addition, if a sequence $(t^n)_{n\in \N}$ in $[t_0,T]$ converges to $t$, then any cluster point $\rho$ (for the convergence of measures) of the $m^n_{t^n}$'s satisfies $$m_t\leq \rho\leq m_{t-}.$$

The continuity and non-increasing property of $\Psi$ imply  that, for any sequence $(t^n,x^n)_{n\in \N}$ such that  $t^n \to t$ and $x^n \to x$, we have 
\be\label{lkzjaenrsdftgm1}
\liminf_{t^n\to t, \; x^n\to x} \Psi(x^n, m^n_{t^n}) \geq \Psi(x, m_{t-}) .
\ee
We claim that 
\be
\begin{split}\label{lakejznrdfgkl01} 
&\liminf_{n \to \infty} \iint_{[t_0,T]\times \T^d} \Psi(x, m^n_{t-}) d\mu^n(t, x) +G(m^n_T)\\
& \qquad \qquad \qquad \qquad \geq \iint_{[t_0,T]\times \T^d} \Psi(x,m_{t-}) d\mu(t,x)+G(m_T).
\end{split}
\ee 
Indeed, we first note that
\be\label{lakejznrdfgkl1}
 \liminf_{(t^n,x^n) \to (t,x)} \Psi(x^n, m^n_{t^n}) = \lim_{k\to\infty, \; \ep\to 0^+} \inf_{n\geq k, \; (s,y)} \Psi(y,m^n_{s}) +\frac{1}{\ep} |(s,y)-(t,x)|, 
\ee
\vs
where the $\liminf$ is taken over all subsequences $t^n \to t$, $x^n \to x$. 
\vs
Indeed, set
$$
f_{k,\ep}(t,x)= \inf_{n\geq k, \; (s,y)} \Psi(y,m^n_s) +\frac{1}{\ep} |(s,y)-(t,x)|.
$$
Then, for any $(k, \ep)$ and any $(t^n,x^n) \to (t,x)$,  
$$
\liminf_{t^n\to t, \; x^n\to x} \Psi(x^n, m^n_{t^n}) \geq f_{k,\ep}(t,x).
$$ 
Conversely, let $n^{k,\ep} \geq k$, $(s^{k,\ep}, y^{k,\ep})$ be $\ep-$optimal for $f_{k,\ep}(t,x)$ and note that, since $\Psi$ is bounded, the $(s^{k,\ep}, y^{k,\ep})$'s tend  as $\ep\to 0$ to $(t,x)$. 
\vs
It follows that 
$$
\lim_{k\to\infty, \; \ep\to 0^+} f_{k,\ep}(t,x) \geq \liminf_{k\to\infty, \; \ep\to 0^+} \Psi(y^{k,\ep}, m^{n^{k,\ep}}_{s^{k,\ep}}) \geq 
 \liminf_{(t^n,x^n) \to (t,x)} \Psi(x^n, m^n_{t^n}), 
$$
and, hence,  \eqref{lakejznrdfgkl1}. 
\vs
Next, we note that the facts that  $f_{k,\ep}$ is Lipschitz continuous and, if  $n\geq k$, then $f_{k,\ep}(t,x)\leq \Psi(x, m^n_{t-})$, yield that 
\begin{align*}
\liminf_n  \iint_{[t_0,T]\times \T^d} \Psi(x, m^n_{t-}) d\mu^n(x,t) & \geq  \liminf_n \iint_{[t_0,T]\times \T^d} f_{k,\ep}(t,x) d\mu^n(t, x) \\
&  = \int_{[t_0, T]\times \T^d} f_{k,\ep} (t,x) d\mu(t, x). 
\end{align*}
\vs
Using dominated convergence, \eqref{lakejznrdfgkl1} and  \eqref{lkzjaenrsdftgm1}, we find 
\begin{align*}
\lim_{k\to \infty, \; \ep\to 0^+} \iint_{[t_0,T]\times \T^d}  f_{k,\ep}(t,x) d\mu(t, x)  &  = \iint_{[t_0,T]\times \T^d}  \liminf_{ n \to \infty, (t^n,x^n) \to (t,x)} \Psi(x^n, m^n_{t^n})  d\mu(t, x) \\
& \geq \iint_{[t_0,T]\times \T^d}  \Psi(x,m_{t-}) d\mu(t,x). 
\end{align*}
Together with the continuity of $G$, this proves \eqref{lakejznrdfgkl01}.
\vs
In conclusion, we infer by \eqref{lakejznrdfgkl010} and \eqref{lakejznrdfgkl01}  that $(m, \alpha,  \mu)$ is a minimizer of $J$. 

\end{proof}

\subsection*{The continuity  of the limit problem} 

Since  the dynamics \eqref{meqn} are discontinuous, to prove the continuity of $U$ requires considerable effort and  several approximation procedures. The first step is to  regularize in time the discontinuous  term $\Psi(x,m_{t-})$ in the cost function $J$. This leads to the new value function $U^\theta$ which approximates  $U$ (Lemma \ref{lem.UthetatoU}). Thanks to this extra gain of regularity, we regularize the dynamics \eqref{meqn} to get  a penalized cost functional  $J^{\theta, \delta}$.  Its  value function $U^{\theta,\delta}$ is, in view of the optimality condition (Proposition \ref{prop.dualdeltaAppend}),  Lipschitz in the measure argument (Proposition \ref{prop.UthetadeltaLip}).
 Then we show that $U^{\theta,\delta}$ is indeed an  approximation of $U$ (Lemma \ref{lem.UthetadeltatoUtheta}). Passing to the limit in the various penalizations we  eventually obtain the regularity of $U$ (Proposition \ref{prop.reguU}). 
\vs
To execute the strategy outlined above, we will need to assume that the terminal condition $G$ is both smooth and  $\Psi$-monotone, that is,
\be\label{assump.Gsmooth}
\begin{cases} & G \in C^1 \ \text{with} \  \dfrac{\delta G}{\delta m} \text{ jointly Lipschitz in $(m,x)$, and, for all $m,n \in \sub \text{ with } n \leq m$},  \vspace{.1cm}\\  
 & \ds \hskip1in       G(m) \leq G(n) + \int_{\T^d} \Psi(x,m) d(m - n). 
\end{cases}
\ee

Fortunately, Lemma \ref{lem.envelopeLip} and Proposition \ref{prop.mollification} from the Appendix allow us to remove these assumptions in the end. 
\vs

In what follows we summarize the assumptions as
\be\label{ath1}
 \eqref{phl1}, \eqref{phl2}, \eqref{phl3} \ \text{and} \ \eqref{assump.Gsmooth} \ \text{hold}.
\ee

For the first layer of approximation,  we fix $(\xi_{\theta})_{\theta > 0}$ to be a smooth approximation to the identity on $\R$ with the property that,  for each $\theta > 0$, the support of $\xi_{\theta}$ is contained in $[0,\theta]$. 
\vs

For $\theta > 0$ and $(t_0,m_0) \in [0,T) \times \sub$, we introduce the cost functional  $J^{\theta}_{t_0,m_0} : {\cA}_{t_0,m_0} \to \R$ by 
\begin{align*}
    J^{\theta}_{t_0,m_0}(m,\alpha,\mu) = \int_{t_0}^T  \int_{\T^d} L\big(x,\alpha(t,x),m_t \big) dm_t(x) dt +\int_{[t_0,T]\times \T^d} \Psi\big(x, (\xi_\theta \ast m)_t) \mu(dt,dx) +G(m_T), 
\end{align*}
where by convention we extend $m$ to $m_{t_0-} = m_0$ on $(-\infty, t_0)$, and $\xi_{\theta} \ast m$ is the  continuous in $\sub$ path  given by
\begin{align*}
    (\xi_{\theta} \ast m)_t = \int_0^{\theta} \xi_{\theta}(s) m_{t - s} ds.
\end{align*}
We then define the value function
$$
U^\theta(t_0,m_0)= \inf_{(m,\alpha, \mu) \in \cA_{t_0,m_0}} J^\theta_{t_0,m_0}(m,\alpha, \mu). 
$$

\begin{lemma}\label{lem.UthetatoU} Assume \eqref{ath1}. Then, for any $(t_0,m_0)\in [0,T]\times \sub$, 
$$
\lim_{\theta\to 0^+} U^\theta(t_0,m_0)= U(t_0,m_0).
$$
\end{lemma} 

\begin{proof} We first claim that 
\be\label{limsupUtheta}
\limsup_{\theta\to 0^+} U^\theta(t_0,m_0) \leq U(t_0,m_0).
\ee
Indeed, let $(m,\alpha, \mu) \in \cA_{t_0,m_0}$   be optimal for $U(t_0,m_0)$. Since $m$ is c\`{a}dl\`{a}g, ${\rm spt}(\xi_\theta)\subset [0,\theta]$ and $\Psi$ is continuous, we have
$$
\lim_{\theta\to 0^+} \Psi\big(x,(\xi_\theta\ast m)_t\big) = \Psi(x, m_{t-}).
 $$
Thus, by dominated convergence,
$$
\lim_{\theta\to 0^+} J^\theta_{t_0,m_0}(m,\alpha, \mu) = J_{t_0,m_0}(m,\alpha, \mu) .
$$
It follows  that 
$$
\limsup_{\theta\to 0^+} U^\theta(t_0,m_0) \leq \lim_{\theta\to 0^+} J_{t_0,m_0}^\theta(m,\alpha, \mu) = J_{t_0,m_0}(m,\alpha, \mu) = U(t_0,m_0),
$$
and \eqref{limsupUtheta} holds. 
\vs
Let now $\theta^n\to 0$ and $(m^n, \alpha^n, \mu^n)$ be $(1/n)-$optimal for $U^{\theta^n}(t_0,m_0)$. By Proposition \ref{cor.compactmn} and up to a subsequence labeled in the same way, $(m^n, \alpha^n, \mu^n)$ converges to some $(m,\alpha, \mu) \in \cA_{t_0,m_0}$, 
$$
\liminf \int_0^T \int_{\T^d} L\big(x,\alpha^n(t,x), m^n_t)dm^n_t dt \geq \int_0^T \int_{\T^d} L\big(x,\alpha(t,x), m_t)dm_t(x) dt,
$$
and, if $t^n\to t$, then it can be checked as in the proof of Proposition \ref{cor.compactmn} that any cluster point $\hat m$ of the sequence $(\xi_{\theta^n} \ast m^n)_{t_n}$ satisfies $m_t \leq \hat m \leq m_{t-}$.
 \vs
 It follows from the above and the fact that $\Psi$ is non-increasing  that,  for any $(t,x)$ and any sequence $(t^n,x^n) \to (t,x)$, 
\be\label{lkzjaenrsdftgm}
\liminf_{n \to \infty} \Psi\big(x^n, (\xi_{\theta^n}\ast m^n)_{t^n} \big) \geq \Psi(x,m_{t-}). 
\ee
The proof  of 
\begin{equation}\label{lakejznrdfgkl0} 
\begin{split}
& \liminf \int_{t_0}^T \int_{\T^d} \Psi\big(x, (\xi_{\theta^n}\ast m^n)_t \big) d\mu^n(x, t) + G(m^n_T)\\
& \qquad \qquad \geq \iint_{[t_0,T] \times \T^d} \Psi(x,m_{t-}) d\mu(x,t) +G(m_T) 
\end{split}
\ee
 follows exactly the same argument as the proof of \eqref{lakejznrdfgkl01} above, so we omit it. \vs
We can conclude that 
$$
\liminf U^{\theta^n}(t_0,m_0) = \liminf J_{t_0,m_0}(m^n,\alpha^n, \mu^n)  \geq J_{t_0,m_0}(m,\alpha, \mu) \geq U(t_0,m_0).
$$ 
\end{proof}

Next we introduce a further penalized problem.  First, for $(t_0,m_0) \in [0,T) \times \sub$, we consider the set  $\wt{\cA}_{t_0,m_0}$  of triples $(m,\alpha,\beta)$ such that $[t_0,T] \ni t \mapsto m_t \in \sub$ is continuous, $\alpha : [t_0,T] \times \T^d \to \R^d$ and $\beta : [t_0,T] \times \R$ satisfy $\beta \geq 0$, 
    \begin{align*}
        \int_{t_0}^T \int_{\T^d} \Big( |\alpha(t,x)|^2 + |\beta(t,x)|^2 \Big) dm_t(x) dt < \infty,
    \end{align*}
    and, in the sense of distributions,  the equation 
    \begin{align*}
        \partial_t m - \Delta m + \text{div}(m\alpha) = - \beta m \ \text{in} \ (t_0,T) \times \T^d \ \text{and} \  m_{t_0} = m_0.
    \end{align*}
    For $\delta, \theta > 0$ and $(t_0,m_0) \in [0,T) \times \sub$, we define the cost functional $J^{\theta,\delta}_{t_0,m_0} : \wt{\cA}_{t_0,m_0} \to \R$ by 
\begin{align*}
&J^{\theta,\delta}_{t_0,m_0}(m,\alpha,\beta) = J^{\theta}_{t_0,m_0}(m,\alpha, \beta m) + \frac{\delta}{2} \int_{t_0} \int_{\T^d} |\beta(t,x)|^2 dm_t(x) dt
\\
&\qquad \qquad \qquad =   \int_{t_0}^T \int_{\T^d} L\big(x,\alpha(t,x),m_t \big)dm_t(x) dt  + \int_{t_0}^T \int_{\T^d} \Psi\big(x,(\xi_\theta \ast m)_t ) \beta(t,x) dm_t(x) dt
 \\
 &\qquad \qquad \qquad \qquad \qquad \qquad +\frac{\delta}{2} \int_{t_0}^T \int_{\T^d} |\beta(t,x)|^2  dm_t(x) dt + G(m_T).
\end{align*}
Finally the value function $U^{\theta,\delta} : [0,T] \times \sub \to \R$ is given by
$$
U^{\theta,\delta}(t_0,m_0) =\inf_{(m,\alpha, \beta) \in \wt{\cA}_{t_0,m_0}} J^{\theta,\delta}_{t_0,m_0} (m,\alpha, \beta) \geq U^\theta(t_0,m_0).
$$
We show in Lemma \ref{lem.UthetadeltatoUtheta} below that, as  $\delta\to 0$  and then $\theta\to 0$,   the $U^{\theta,\delta}$'s converge to $U$. 
\vs 

For the rest of this section, $C$ denotes a constant which depends on $G$ only through its Lipschitz constant and is independent of $\theta$, $\delta$, and of the initial condition $(t_0,m_0)$.
\vs

Next we discuss an optimality condition for  $U^{\theta,\delta}$.

\begin{proposition}\label{prop.dualdeltaAppend} Assume \eqref{ath1}.  Then, for any initial condition $(t_0,m_0)$, there exists a minimizer $(\bar m, \bar \alpha,\bar \beta)$ for $U^{\theta,\delta}(t_0,m_0)$. Moreover, for any minimizer $(\bar m, \bar \alpha, \bar \beta)$, we have, a.e. with respect to $dt dm_t(x)$,  
\begin{align*}
    \ov{\alpha}(t,x) = - D_p H\big( x, Du(t,x), \ov{m}_t\big) \ \text{and} \ \ov{\beta}(t,x) = \frac{1}{\delta} \Big( u(t,x) - \Psi\big(x, (\xi_{\theta} * \ov{m})_t \big) \Big)_+,
\end{align*}
where $u$ is a classical solution to 
\begin{align} \label{eq.HJuthetadelta}
\begin{cases}
  \ds  - \partial_t u(t,x) - \Delta u(t,x)+ H\big(x, Du(t,x), \ov{m}_t\big) \\[1.2mm]
 \hskip1in  + \dfrac{1}{2\delta} \Big( u(t,x) - \Psi\big(x, (\xi_{\theta} * \ov{m})_t \big) \Big)_+^2 = F(t,x)  \ \text{in}  \  [t_0,T) \times \T^d, 
    \\[1.2mm]
 \ds    u(T,\cdot) = \frac{\delta G}{\delta m}(\ov{m}_T,\cdot), 
\end{cases}
\end{align}
with the function $F$ given by 
\begin{align*}
    F(t,x) &= \int_{\T^d} \frac{\delta L}{\delta m}\big( y, \ov{\alpha}(t,y), \ov{m}_t,x \big) d\ov{m}_t(y) \\
    &\hskip.5in + \int_0^{\theta \wedge (T-t)} \int_{\T^d} \xi_{\theta}(s) \frac{\delta \Psi}{\delta m} \Big(y, (\xi * \ov{m})_{t+s}, x \Big) \ov{\beta}(t+s,y) \ov{m}_{t+s}(dy)ds
    \\[1.2mm]
    &= \int_{\T^d} \frac{\delta L}{\delta m}\big( y, \ov{\alpha}(t,y), \ov{m}_t,x \big)d \ov{m}_t(y)
    \\
    &\qquad \qquad +\frac{1}{\delta} \int_{t}^{(t+\theta)\wedge T} \int_{\T^d} \xi_\theta(s-t) \frac{\delta \Psi}{\delta m}\Big( y, (\xi_{\theta} * \ov{m})_s,x \Big)  \Big(u(s,y) - \Psi\big(y, (\xi_{\theta} * \ov{m})_s \big) \Big)_+ d\bar m_s(y) ds. 
\end{align*}
Finally, $u$  and $Du$ are bounded by a constant, which depends on $G$ only through its Lipschitz constant and is  independent of $\theta$, $\delta$ and  the initial condition $(t_0,m_0)$.
\end{proposition}

\begin{proof} The proof of the existence of a minimizer and of the optimality condition follows a standard argument. Indeed,  it suffices to notice that, since the non-local dependence on $\bar m$ is smoothing,  the Hamilton-Jacobi-Bellman equation \eqref{eq.HJuthetadelta} possesses a classical solution, 
and then proceed for instance as in  \cite{BrianiCardaliaguet}. 
\vs
\vs
Next we show two preliminary estimates, which will be used to bound $u$ and $Du$, namely,
\be\label{qlejksnfdgc}
\int_{t_0}^T \int_{\T^d}  \delta^{-1}  \Big( u(t,x)- \Psi\big(x, (\xi_{\theta} * \ov{m})_t\big) \Big)_+ d\bar m_t(x)dt \leq 1, 
\ee
and 
\be\label{qlejksnfdgc2}
 \int_{t_0}^T\int_{\T^d} \left( \left\| \frac{\delta L}{\delta m}(x, \bar \alpha(t,x),\ov{m}_t,\cdot)\right\|_\infty +
  \left\| D_y \frac{\delta L}{\delta m}(x, \bar \alpha(t,x),\ov{m}_t,\cdot)\right\|_\infty \right) d\ov{m}_t(x) \leq C.
\ee
\vs
The first follows by integrating the equation 
\begin{align*}
    \partial_t \ov m - \Delta \ov m + \text{div}\big(\ov m \ov \alpha\big) = - \ov{\beta} \ov{m} \text{ in } (t_0,T) \times \T^d, \quad \ov{m}_{t_0} = m_0
\end{align*} 
 in time and space and using the  equality  $\bar \beta = \delta^{-1}  \Big(u-\Psi\big(x,(\xi_{\theta} * \ov{m})_t \big) \Big)_+$. The second  is a direct consequence of the bound on $\int_{t_0}^T \int_{\T^d} |\bar \alpha|^2\bar m$, which follows from the coercivity assumption \eqref{hypL} and the bound on $U^{\theta,\delta}$, combined with the growth conditions on $\dfrac{\delta L}{\delta m}$ and $D_y \dfrac{\delta L}{\delta m}$.
\vs
We note that \eqref{qlejksnfdgc} and \eqref{qlejksnfdgc2} together with the regularity of $\dfrac{\delta \Psi}{\delta m}$ imply that there is a constant $C$ such that
\begin{align*}
    \int_t^T  \|F(s,\cdot)\|_{W^{1,\infty}} ds \leq C.
\end{align*}
 We now show that $u$ is uniformly bounded  independently of $\delta$. Since the proofs of the lower and upper bounds are similar, here we only check that $u$ is bounded above. 
 \vs
 For this we note that the map $t\to w(t)$ given by 
\begin{align*}
 w(t) & = \left\|\frac{\delta G}{\delta m}\right \|_\infty + \int_t^T \left \|H(\cdot, 0, \bar m_s)\right                                                                                                                                    \|_\infty ds + 
 \int_t^T \left\|F(s, \cdot)\right\|_{\infty} ds.
\end{align*}
 is a supersolution to  \eqref{eq.HJuthetadelta}. Thus, by comparison,  
\begin{align*}
\max_{t,x} u(t,x) &  \leq \|\frac{\delta G}{\delta m}\|_\infty + \int_{t_0}^T \|H(\cdot, 0, \bar m(s))\|_\infty ds + 
 \int_{t_0}^T \|F(s,\cdot)\|_{\infty} ds  \leq \; C.
\end{align*}
 To  check that $Du$ is, independently of $\delta$, globally bounded, we  fix $z\in \R^d$, set $u_z= Du\cdot z$  and note that $u_z$ solves  
\begin{align*}
-\partial_t u_z -\Delta u_z + H_x\cdot z + H_p\cdot Du_z + \frac{1}{\delta} \Big(u-\Psi\big(x,(\xi_{\theta} * \ov{m})_t \big) \Big)_+(u_z-D_x\Psi\cdot z)= D_xF(t,x) \cdot z. 
\end{align*}
Thus $u_z$ is in the set $\{u_z> \|D_x\Psi\|_\infty |z|\}$ a subsolution to the equation 
\begin{align*}
&-\partial_t u_z -\Delta u_z + H_x\cdot z + H_p\cdot Du_z  \leq   |z|  |D_x F(t,x)|.
\end{align*}
Recalling \eqref{phl1} to bound the term $H_x \cdot z$, we infer by comparison that, for any $t\in [t_0, T]$,  
\begin{align*}
& \max_{x\in\T^d} \max\{ u_z(t,x), \|D_x\Psi\|_\infty |z|\} \\
&  \leq |z| \|D_x \frac{\delta G}{\delta m}\|_\infty + C|z|(\int_t^T \|Du(s,\cdot)\|_\infty ds +1)+ 
|z| \int_{t_0}^T  \|DF(s,\cdot)\|_{\infty} ds.
\end{align*}
Taking the sup over $|z|\leq 1$ we conclude by Gronwall Lemma that $\|Du\|_\infty$ is bounded uniformly in $\delta$, $\theta$, $t_0$ and $m_0$. 

\end{proof}

We now show the uniform Lipschitz continuity  of $U^{\theta,\delta}$ in the measure argument. 

\begin{prop}\label{prop.UthetadeltaLip} Assume \eqref{ath1}. Then the map $U^{\theta,\delta}$ is, uniformly in $t$, $\theta$ and $\delta$, Lipschitz continuous in $m$ with a constant which depends on $G$ only through its Lipschitz constant.
\end{prop}

\begin{proof} 
The classical Schauder estimates imply that, for given $(t_0, m_0)$ and $(\bar m,\bar \alpha,\bar \beta)$ optimal for $U^{\theta,\delta}(t_0,m_0)$, the  $u$ given by Proposition \ref{prop.dualdeltaAppend} is in $C^{1+\gamma/2, 2+\gamma}$ for any $\gamma\in (0,1)$ with a  norm depending on $\delta$. 
\vs
Given $m_0'\in\sub $, let $m'$ be the solution to 
$$
\partial_t m'-\Delta m' +{\rm div}(m'\bar \alpha)= -\bar \beta m' \ \text{in} \ (t_0,T) \times \T^d \ \  \text{and} \  \  m'_{t_0} =m'_0.
$$
Since  $\bar \alpha = - D_pH(x, Du, \ov{m}_t)$, with $Du$ bounded and Lipschitz in space, and $\bar \beta = \delta^{-1} (u-\Psi)_+$  is bounded and Lipschitz (with Lipschitz constants depending on $\delta$), we have 
\be\label{ksjfdlqhjsdkd}
\sup_{t\in [t_0,T]} \bd(m_t, m'_t) \leq C_\delta \bd(m_0, m_0').
\ee
Indeed, this can be proved via a duality argument: for $\phi \in W^{1,\infty}$, we have 
\begin{align*}
    \int_{\T^d} \phi \, dm'_t(x) = \int_{\T^d} f(t_0,x) \, dm'_0(x), \quad \int_{\T^d} \phi \, dm_t(x) = \int_{\T^d} f(t_0,x) \,dm_0(x),
\end{align*}
where $f$ solves 
\begin{align*}
    \partial_s f + \Delta f + \alpha \cdot Df = \beta f, \text{  $(s,x) \in (t_0,t) \times \T^d$}, \quad f(t,x) = \phi(x). 
\end{align*}
It is straightforward to verify that 
\begin{align*}
    \|f(t_0,\cdot)\|_{W^{1,\infty}} \leq C_{\delta} \|\phi\|_{W^{1,\infty}}, 
\end{align*}
for $C_{\delta}$ depending on $\delta$ through the quantities $\sup_{t_0 \leq s \leq t} \|\ov \alpha(s,\cdot) \|_{W^{1,\infty}}$ and $\sup_{t_0 \leq s \leq t} \|\ov \beta(s,\cdot)\|_{W^{1,\infty}}$. We thus find
\begin{align*}
    \int_{\T^d} \phi \, d(m_t - m_t') = \int_{\T^d} f(t_0,\cdot) \, d(m_0 - m_0') \leq C_{\delta} \|\phi\|_{W^{1,\infty}} \bd(m_0,m_0'), 
\end{align*}
hence \eqref{ksjfdlqhjsdkd}.

We set $\rho=m'-\bar m$ and note that $\rho$ solves 
$$
\partial_t \rho-\Delta \rho +{\rm div}(\rho\bar \alpha)= -\bar \beta \rho  \ \text{in} \ (t_0,T] \ \  \text{and} \  \  \rho_{t_0} =m'_0-m_0.
$$
Then, with $m^{\tau} = (1-\tau)\bar m+\tau m'$, 
\begin{align*}
& U^{\theta,\delta}(t_0,m_0') -U^{\theta,\delta}(t_0,m_0)\leq J^{\theta,\delta}_{t_0,m_0'}(m',\bar  \alpha,\bar  \beta)-J_{t_0,m_0}^{\theta,\delta}(m, \bar \alpha, \bar \beta)\notag \\
&  =  \int_{t_0}^T \int_{\T^d} \Big( L\big(x,\bar \alpha(t,x),\ov{m}_t\big) +\frac{\delta}{2} \big| \ov{\beta}(t,x) \big|^2 \Big)d\rho_t(x) dt
\\
&\qquad + \int_{t_0}^T \int_{\T^d} \Big( \Psi\big(x,(\xi_\theta\ast m')_t\big)\bar \beta(t,x) dm'_t(x)-\Psi\big(x,(\xi_\theta\ast \bar m)_t \big)\bar \beta(t,x)  d\ov{m}_t(x) \Big) dt\notag\\
& \qquad + \int_{t_0}^T \int_{\T^d}  \Big( L\big(x,\bar \alpha(t,x), m'_t\big)-L\big(x,\bar \alpha(t,x),\bar m_t\big) \Big) dm'_t(x)  dt +G(m'_T)-G(m_T) \notag \\ 
&=  \int_{t_0}^T  \int_{\T^d} \Bigl(L(x,\bar \alpha(t,x),\bar m_t) +\frac{\delta}{2}  \big| \ov{\beta}(t,x) \big|^2+ \Psi\big(x,(\xi_\theta\ast \ov{m})_t \big) \ov{\beta}(t,x)  \Bigr) d\rho_t(x) dt \notag \\ 
& \qquad  +  \int_{t_0}^T \int_{\T^d} \int_0^1 \int_{\T^d} \frac{\delta \Psi}{\delta m}\Big(x,\big(\xi_\theta\ast m^{\tau} \big)_t, y \Big) \ov{\beta} (t,x) \big(\xi_\theta \ast \rho\big)_t(dy) d\tau  dm'_t(y) dt \notag \\
& \qquad + \int_{t_0}^T \int_{\T^d} \int_0^1 \int_{\T^d} \frac{\delta L}{\delta m}\Big(x,\bar \alpha, m^{\tau}_t, y \Big)  d\rho_t(y) \, d\tau \, dm_t'(x) \, dt +\int_0^1 \int_{\T^d} \frac{\delta G}{\delta m}\big( m^{\tau}_T,x \big) d\rho_T(x) d\tau.
\end{align*}

Writing out $\big(\xi_{\theta} * m^{\tau} \big)_t = \int_0^{\theta} m^{\tau}_{t-s} \xi_{\theta}(s)$ and then exchanging the roles of $t$ and $s$ in the second integral, we find 
\be\label{kjhezrnfd}
\begin{split}
& U^{\theta,\delta}(t_0,m_0') -U^{\theta,\delta}(t_0,m_0)\\ 
& \leq  \int_{t_0}^T  \int_{\T^d} \Bigl(L(x,\bar \alpha(t,x),\bar m_t) +\frac{\delta}{2}  \big| \ov{\beta}(t,x) \big|^2+ \Psi\big(x,(\xi_\theta\ast \ov{m})_t\big) \ov{\beta}(t,x) \Bigr) d\rho_t(x) dt \\ 
& \quad   +  \int_{t_0}^T  \int_{\T^d} \int_0^1 \int_{t}^{(t+\theta)\wedge T}\int_{\T^d} \frac{\delta \Psi}{\delta m}\Big(x, \big(\xi_\theta\ast m^\tau\big)_s, y \Big)\xi_\theta(s-t)  \bar \beta (s,x) m_s'(dx)\, ds\, d\tau  \, d\rho_t(y) \,dt  \\
& \quad  + \int_{t_0}^T \int_{\T^d} \int_0^1 \int_{\T^d} \frac{\delta L}{\delta m}\Big(x,\bar \alpha, m^{\tau}_t, y \Big)  d\rho_t(y) \, d\tau \, dm_t'(x) \, dt +\int_0^1 \int_{\T^d} \frac{\delta G}{\delta m}\big( m^{\tau}_T,x \big) d\rho_T(x) d\tau. 
\end{split}
\ee

Using  the equations satisfied by $u$ and $\rho$ and the fact that $\bar \beta=\delta^{-1}(u-\Psi)_+$ and $\bar \alpha = -D_pH(x,Du, \bar m)$, we have 
\begin{align*}
&\int_{t_0}^T \int_{\T^d} \Bigl(L(x,\bar \alpha(t,x),\bar m_t) +\frac{\delta}{2} \big| \ov{\beta}(t,x)\big|^2 +  \Psi\big(x,(\xi_\theta\ast \ov{m})_t\big) \bar \beta(t,x)+ \int_{\T^d}  \frac{\delta L}{\delta m}(y,\bar \alpha, \bar m_t,x))d\bar m_t(y)    \Bigr) d\rho_t(x) \; dt \\ 
&  +  \int_{t_0}^T  \int_{\T^d}  \int_{t}^{(t+\theta)\wedge T}\int_{\T^d} \frac{\delta \Psi}{\delta m}\Big(y,(\xi_\theta\ast\bar m)_s,x \Big)\xi_\theta(s-t)  \bar \beta (s,y) \ov{m}_s(dy) \, ds  \, d\rho_t(x) \, dt
\\
&\qquad \qquad 
+ \int_{\T^d} \frac{\delta G}{\delta m}\big(\bar m_T,x \big) d\rho_T(x) \notag\\
&=  \int_{t_0}^T  \int_{\T^d} \Big(L(x,\bar \alpha, \bar m) +\frac{\delta}{2} \bar \beta^2- \partial_t u-\Delta u+H(x,Du, \bar m)+\frac{1}{2\delta} (u-\Psi)_+^2+\bar \beta \Psi \Big) d\rho_t(x)  \; dt 
\\
&\qquad \qquad + \int_{\T^d} \frac{\delta G}{\delta m}(m_T,x) d\rho_T(x) \notag\\ 
&=  \int_{t_0}^T  \int_{\T^d} \Big(L(x,\bar \alpha, \bar m) +H(x,Du, \bar m) + \delta \bar \beta^2+\bar \beta\Psi \Big) d\rho_t(x)  dt 
+ \int_{t_0}^T  \int_{\T^d} u( \partial_t \rho-\Delta \rho) \,dx \, dt 
\\
&\qquad \qquad + \int_{\T^d} u(t_0,x)d\rho_{t_0}(x) \notag\\ 
&=  \int_{t_0}^T  \int_{\T^d} \Big(L(x,\bar \alpha, \bar m) +H(x,Du, \bar m) +\delta\bar \beta^2+Du\cdot \bar \alpha -\bar \beta (u-\Psi) \Big) \rho_t(dx) dt 
\\
&\qquad \qquad + \int_{\T^d} u(t_0,x)d\rho_{t_0}(x) \notag \\
& =  \int_{\T^d} u(t_0,x) d(m'_0-m_0)(x)\; \leq \; \|u(t_0, \cdot)\|_{W^{1,\infty}} \bd(m_0,m_0'). 
\end{align*}

Proposition \ref{prop.dualdeltaAppend} states that $\|u(t_0, \cdot)\|_{W^{1,\infty}}$ is bounded by a constant $C_0$ that is independent of the initial condition and of $\delta$. 
\vs
Plugging the previous inequality into \eqref{kjhezrnfd}, controlling the differences between $m^\tau$ and $\bar m$, and using the regularity of $\dfrac{\delta \Psi}{\delta m}$, $\dfrac{\delta L}{\delta m}$, and $\dfrac{\delta G}{\delta m}$ implied by \eqref{ath1},  we obtain an estimate of the form
    \begin{align*}
        U^{\theta,\delta}(t_0,m_0') - U^{\theta,\delta}(t_0,m_0) &\leq \int_{\T^d} u(t_0,x) d(m_0'-m_0)(dx) + C \sup_{t_0 \leq t \leq T} \sup_{0 \leq \tau \leq 1} \bd\big(\ov{m}_t, m^{\tau}_t\big)
        \\
        &= \int_{\T^d} u(t_0,x) d(m_0'-m_0)(dx) + C \sup_{t_0 \leq t \leq T} \bd\big(\ov{m}_t, m_t'\big).
    \end{align*}

Applying \eqref{ksjfdlqhjsdkd} to control the second term, we obtain 
\be\label{zknjnsknk}
 U^{\theta,\delta}(t_0,m_0') -U^{\theta,\delta}(t_0,m_0)\leq C_0 \bd(m_0,m_0')\big(1+ C_\delta  \bd(m_0,m_0') \big). 
\ee
 Let us finally check that this inequality, which holds for any $m_0,m_0'\in \sub$, implies that $U^{\theta,\delta}(t_0,\cdot)$ is $C_0-$Lipschitz continuous, where $C_0$ is independent of $\delta, \theta, t_0$. Indeed, let $n$ be a large integer and, for $k\in \{0, \dots, n\}$, let us set $m_0^k= (1-k/n)m_0+(k/n)m_0'$. Note that, by the definition of $\bd$, 
$$
\bd(m_0^k, m_0^{k+1})= \bd(m_0,m_0')/n\qquad \forall k\in \{0, \dots, n-1\},
$$
so that, applying inequality \eqref{zknjnsknk} between $m_0^k$ and $m_0^{k+1}$, we find: 
\begin{align*}
    U^{\theta,\delta}(t_0,m_0') -U^{\theta,\delta}(t_0,m_0) & = \sum_{k=0}^{n-1}U^{\theta,\delta}(t_0,m_0^{k+1}) -U^{\theta,\delta}(t_0,m_0^k)\\
    & \leq C_0\sum_{k=0}^{n-1}\bd(m_0^k,m_0^{k+1})\big(1+ C_\delta  \bd(m_0^k,m_0^{k+1}) \big)\\
    & = C_0\bd (m_0,m_0') \Big(1+ C_\delta \frac{\bd (m_0,m_0')}{n}\Big).
\end{align*}
Letting $n\to\infty$ completes the proof. 
\end{proof}

We next check that $U^{\theta,\delta}$ is an approximation of $U^\theta$. For this, we need the following preliminary fact.

\begin{lemma} \label{lem.rc}
 Assume \eqref{ath1}. Then, for each $(t_0,m_0) \in [0,T) \times \sub$, there is an optimal control $(m,\alpha, \mu) \in \cA_{t_0,m_0}$ for the problem defining $U(t_0,m_0)$ such that $t \mapsto m_t$ is continuous at time $T$. 
\end{lemma}

 \begin{proof} 
    Let $(m,\alpha, \mu) \in \cA_{t_0,m_0}$ optimal for the problem defining $U(t_0,m_0)$ (such a control exists in view of Proposition \ref{prop.existsOS}). Define $\wt{m}$ and $\wt{\mu}$ by 
    \begin{align*}
        \wt{m}_t = \begin{cases}
            m_t & t < T, 
            \\
            m_{T-} & t = T, 
        \end{cases}
        \qquad \wt{\mu} = \mu|_{[0,T)}.
    \end{align*}
Then Lemma \ref{lem.cdlag} allows us to conclude that $(\wt{m}, \alpha, \wt{\mu}) \in \cA_{t_0,m_0}$ and  the assumption that $G$ is $\Psi$-non-increasing implies that: 
$$
J(\wt{m},\alpha, \wt{\mu})-  J(m,\alpha,\mu) = 
G(m_{T-}) - G(m_T) - \int_{\T^d} \Psi(x, m_{T-}) d(m_{T-} - m_T) \leq 0.
$$ 
Hence $(\wt{m}, \alpha, \wt{\mu})$ is also optimal for the problem defining $U(t_0,m_0)$ and is continuous  at $T$. 
   \end{proof}

\begin{lemma}\label{lem.UthetadeltatoUtheta} For any $(t_0,m_0)\in [0,T]\times \sub$, 
$$
\limsup_{\theta\to 0^+}\limsup_{\delta\to 0^+} U^{\theta,\delta}(t_0,m_0) =
\liminf_{\theta\to 0^+}\liminf_{\delta\to 0^+} U^{\theta,\delta}(t_0,m_0) =U(t_0,m_0).
$$
\end{lemma}

\begin{proof} We already know that $U^{\theta,\delta}(t_0,m_0) \geq U^\theta(t_0,m_0)$. Hence, by Lemma \ref{lem.UthetatoU},
$$
\liminf_{\theta\to 0^+}\liminf_{\delta\to 0^+} U^{\theta,\delta}(t_0,m_0) \geq U(t_0,m_0).
$$
\vs
We use Lemma \ref{lem.rc} to find $(m, \alpha,\mu) \in \cA_{t_0,m_0}$, which is optimal for the problem defining $U(t_0,m_0)$ and continuous at time $T$. Fix $\eps > 0$. By Lemma \ref{lem.UthetatoU} we can choose $\theta>0$ small enough such that $(m, \alpha,\mu) \in \cA_{t_0,m_0}$ is $\ep-$optimal for the problem defining $U^\theta(t_0,m_0)$.   Then, using  $(m, \alpha,\mu)$, we construct, for $\eta, \sigma > 0$,   admissible controls $(m^{\eta,\sigma}, \alpha^{\eta,\sigma}, \beta^{\eta,\sigma})$  such that  $m^{\eta,\sigma}_0:=m^{\eta,\sigma}_{t_0}$ converges to $m_0$ as $\eta,\sigma$ tend to $0$, and 
\be\label{limreguJetasigma}
 \limsup_{\eta\to 0}   \limsup_{\sigma\to 0}  \limsup_{\delta\to 0}   J^{\theta,\delta}_{t_0,m^{\eta, \sigma}_{0}}(m^{\eta,\sigma}, \alpha^{\eta,\sigma}, \beta^{\eta,\sigma})\leq  J^\theta_{t_0,m_0} (m, \alpha, \mu).
\ee
Then the uniform continuity of $U^{\theta,\delta}$ in the measure argument yields that 
\begin{align*}
& \limsup_{\delta\to 0} U^{\theta,\delta}(t_0, m_0) =   \limsup_{\eta\to 0} \limsup_{\sigma\to 0}  \limsup_{\delta\to 0}  U^{\theta, \delta} (t_0, m^{\eta,\sigma}_0)  \\
& \qquad \leq  \limsup_{\eta\to 0} \limsup_{\sigma\to 0}  \limsup_{\delta\to 0}  J^{\theta,\delta}_{t_0,m^{\eta, \sigma}}(m^{\eta,\sigma}, \alpha^{\eta,\sigma}, \beta^{\eta,\sigma}) \leq J^\theta_{t_0,m_0}( m, \alpha, \mu)
\\
&\qquad \leq U^{\theta}(t_0,m_0) + \eps. 
\end{align*}
Sending $\theta\to 0^+$ and then $\eps \to 0$ implies that
$$
\limsup_{\theta\to 0^+}\limsup_{\delta\to 0^+} U^{\theta,\delta}(t_0,m_0) \leq U(t_0,m_0),
$$
which completes the proof. 
\vs
To produce the $m^{\eta, \sigma}$, $\alpha^{\eta, \sigma}$ and $\beta^{\eta, \sigma}$ needed to establish \eqref{limreguJetasigma}, we first extend  $(m, \alpha,\mu)$ to $[t_0,\infty)$ by setting $\alpha=\mu=0$ on $[T,\infty)$ and let $m$ be the solution of the heat equation on $(T,\infty)$ with initial condition $m_T$. 
\vs

Let $\Gamma_t$ be the heat kernel on $\T^d$. For $\eta>0$ small, we let $m^\eta = \Gamma_\eta \ast m$, $\alpha^\eta= (\Gamma_\eta \ast (\alpha m))/m^\eta$ and $\mu^\eta= \Gamma_\eta \ast \mu$, where the convolution is in space only. 
\vs
Then $(m^\eta, \alpha^\eta, \mu^\eta)$ satisfies \eqref{meqn} with an initial condition  $m^\eta_0= \Gamma_\eta\ast m_0$ and, 
for $\eta$ small enough, 
\be\label{qlhskdfnglk0}
J^\theta_{t_0,m^{\eta}_0}(m^\eta, \alpha^\eta, \mu^\eta)\leq J^\theta_{t_0,m_0}( m, \alpha, \mu) + \ep/2.
\ee

We extend $(m^\eta, \alpha^\eta, \mu^\eta)$ to $[t_0-\eta,\infty)$ by setting $\alpha^\eta=\mu^\eta=0$ and $m^\eta(t)= \Gamma_{t-t_0+\eta}\ast m_0$ on $[t_0-\eta,t_0)\times \T^d$. Then $(m^\eta, \alpha^\eta, \mu^\eta)$ still satisfies \eqref{meqn} on the time interval $(t_0-\eta,\infty)$ with initial condition $m_0$. 
\vs
For $\sigma \in (0,\eta)$, let $m^{\eta,\sigma}= \xi_\sigma \ast m^\eta$, $\alpha^{\eta,\sigma}= ( \xi_\sigma \ast (\alpha^\eta m^\eta))/m^{\eta,\sigma}$, $\beta^{\eta,\sigma}= \xi_\sigma \ast \mu^\eta/ m^{\eta,\sigma}$, where the convolution is now in time only with the kernel $\xi_\sigma$ as in the definition of $U^\theta$. 
The triple $(m^{\eta,\sigma},\alpha^{\eta,\sigma},\beta^{\eta,\sigma})$  solves \eqref{meqn} on the time interval $(t_0,T)$ with initial condition $m^{\eta,\sigma}_0= m^{\eta,\sigma}_{t_0}$. 
\vs
We  claim that, for $\sigma$ first  and then $\delta$ small enough, 
\be\label{qlhskdfnglk}
J^{\theta,\delta}_{t_0,m^{\eta, \sigma}_0}( m^{\eta,\sigma}, \alpha^{\eta,\sigma}, \beta^{\eta,\sigma}) \leq  J_{t_0,m^{\eta}_0}^\theta(m^\eta, \alpha^\eta, \mu^\eta)+\ep/2. 
\ee

The inequality
\be\label{qlhskdfnglkstep1}
 \int_{t_0}^T \int_{\T^d} L(x, \alpha^{\eta,\sigma},m^{\eta,\sigma})m^{\eta,\sigma}
 \leq 
 \int_{t_0}^T \int_{\T^d} L(x, \alpha^{\eta},m^{\eta})m^{\eta}+\ep/8
\ee
holds for $\sigma$ small enough by Jensen inequality, the regularity of $L$ and the fact that $m^{\eta,\sigma}_t$ converges to $m^{\eta}_t$ for a.e. $t\in [t_0,T]$. 
\vs

It also follows that, for  $\sigma$ small enough, 
\be\label{ikhjzljenlPsi}
\begin{split}
 \int_{t_0}^T \int_{\T^d} \Psi\big(x,(\xi_\theta \ast & m^{\eta,\sigma})_t \big)\beta^{\eta,\sigma}(t,x)m^{\eta,\sigma}_t(dx) dt\\
& \leq  \int_{t_0}^T \int_{\T^d} \Psi\big(x,(\xi_\theta \ast m^{\eta})_t \big)\mu^{\eta}(dt, dx)dxdt+\ep/8.
 \end{split}
 \ee
 Indeed, we note that
\begin{align*}
I_\sigma &= \int_{t_0}^T \int_{\T^d} \Psi\big(x,(\xi_\theta \ast m^{\eta,\sigma})_t \big)\beta^{\eta,\sigma}(t,x)dm^{\eta,\sigma}_t(x) dt
= \int_{t_0}^T \int_{\T^d} \Psi\big( t,x,(\xi_\theta \ast m^{\eta,\sigma})_t\big) d(\xi_\sigma \ast \mu^\eta)_t(x) dt \\ 
& \qquad \qquad = \int_{t_0}^T  \int_{\T^d}  \int_{t-\sigma}^t \xi_\sigma (t-s)  \Psi\big(x, (\xi_\theta \ast m^{\eta,\sigma})_t\big) d\mu^\eta(s, x) dt \\
& \qquad \qquad = \int_{t_0-\sigma}^T  \int_{\T^d} \Bigl(\int_{s\vee t_0}^{(s+\sigma)\wedge T}  \xi_\sigma (t-s)  \Psi\big(x, (\xi_\theta \ast m^{\eta,\sigma})_t\big) dt\Bigr) d\mu^\eta (s,x)  . 
\end{align*}
\vs
Since the $m^{\eta,\sigma}_t$'s converge, as $\sigma\to0$  for a.e. in $t$, to $m^\eta_t$,  the $\xi_\theta \ast m^{\eta,\sigma}$'s converge uniformly to $\xi_\theta \ast m^{\eta}$ and, hence, the integral 
$$
\int_{s\vee t_0}^{(s+\sigma)\wedge T}  \xi_\sigma (t-s)  \Psi\big(x,(\xi_\theta \ast m^{\eta,\sigma})_t)dt
$$
converges uniformly to $\Psi\big(x, (\xi_\theta \ast m^{\eta})_t\big)$, which guarantees the convergence of $I_\sigma$ to 
$$
\int_{t_0}^T \int_{\T^d} \Psi\big(x, (\xi_\theta \ast m^{\eta})_t\big) d\mu^{\eta}(t,x),
$$
and, thus \eqref{ikhjzljenlPsi} holds. 
\vs
The continuity of  $m$ at time $T$ implies that $m^{\eta}$ is also continuous at time $T$. Then, in view of the construction of the convolution kernel $\xi_\sigma$,  $m^{\eta,\sigma}_T \to m_T^{\eta}$ as $\sigma \to 0$.
\vs
Thus, for $\sigma$ small, we have
$$
G(m^{\eta,\sigma}_T) \leq G( m^{\eta}_T) +\ep/8. 
$$
\vs
Combining this last inequality with \eqref{qlhskdfnglkstep1} and \eqref{ikhjzljenlPsi}, we find, for sufficiently small $\sigma$, that  
$$
J^{\theta,\delta}_{t_0, m^{\eta,\sigma}_0}(m^{\eta,\sigma}, \alpha^{\eta,\sigma}, \beta^{\eta,\sigma}) \leq  J^{\theta}_{t_0,m_0}(m^{\eta}, \alpha^{\eta}, \mu^{\eta})+ \int_{t_0}^T \int_{\T^d} \frac{\delta}{2} |\beta^{\eta,\sigma}|^2 m^{\eta,\sigma} + 3\ep/8. 
$$
We finally choose $\delta>0$ small enough to obtain \eqref{qlhskdfnglk}. 
\vs
In view of the fact that, as $\eta\to 0$, the $m^{\eta}_0$'s tend to $m_0$,  to complete the proof we just need to check that 
\begin{equation*}
\lim_{\sigma\to 0} m^{\eta,\sigma}_0= m^{\eta}_0.
\end{equation*}
This last claim follows from the observation that, since $m^{\eta}$ is c\`{a}dl\`{a}g,  for any test function $f\in C(\T^d)$, we have 
$$
\lim_{\sigma\to0}\int_{\T^d} f m^{\eta,\sigma}_0=\lim_{\sigma\to0} \int_{\T^d} \int_{t_0-\sigma}^{t_0} f(x) \xi_\sigma(t_0-s) m^{\eta}(ds,x)dx = 
\int_{\T^d} f m^{\eta}(t_0^-,x)dx = \int_{\T^d} f m^{\eta}_0(x)dx. 
$$
\end{proof}

We finally obtain the following result about the joint continuity of $U$ in $t$ and $m$. 

\begin{prop}\label{prop.reguU} Assume \eqref{ath1}. Then there exist   constants $C$ and $R_0$, which depend on $G$ only through its Lipschitz constant,  such that 
\begin{align*}
    \big| U(t,m) - U(s,n) \big| \leq C \Big( \bd(m,n) + |t-s|^{1/2} \Big), 
\end{align*}
and, for any $(t_0,m_0)\in [0,T]\times \sub$, there is an optimal solution $(m,\alpha, \mu)$ for $U(t_0,m_0)$ such that \[\|\alpha\|_{L^\infty}\leq R_0.\]
\end{prop}

\begin{proof} The Lipschitz continuity of $U$ in the measure variable is a straightforward consequence of  Proposition \ref{prop.UthetadeltaLip} and Lemma \ref{lem.UthetadeltatoUtheta}. 
\vs
Now let $(\bar m^{\theta, \delta},\bar \alpha^{\theta, \delta}, \bar \beta^{\theta, \delta})$ be optimal  for $U^{\theta, \delta}(t_0,m_0)$. Then the  $\bar \alpha^{\theta, \delta}=-H_p(x,Du, \bar m^{\theta, \delta})$'s  are  bounded in $L^\infty$ by a constant $R_0$ depending only on the bound on $Du$ given in Proposition \ref{prop.dualdeltaAppend} and the regularity of $H_p$. 
\vs
Set $\mu^{\theta, \delta}= \bar \beta ^{\theta, \delta}\bar m^{\theta, \delta}$. Then,  Lemma \ref{lem.UthetadeltatoUtheta} and Lemma \ref{lem.UthetatoU} imply that the $(\bar m^{\theta, \delta},\bar \alpha^{\theta, \delta}, \bar \mu^{\theta, \delta})$ 's form a minimizing sequence for $U(t_0,m_0)$. Proposition \ref{cor.compactmn} claims the existence of a subsequence (denoted in the same way) and $(m, \alpha,\mu) \in \cA_{t_0,m_0}$ such that the $(\bar m^{\theta, \delta})$'s, $(\bar \alpha^{\theta, \delta}\bar m^{\theta, \delta})$'s and $(\bar \mu^{\theta, \delta})$'s converge in measure to some $ m$, $ \alpha m$ and $ \mu$ respectively. 
\vs
It follows from the proof of Proposition \ref{prop.existsOS} that $( m,  \alpha,  \mu)$ is optimal for $U(t_0,m_0)$ and, by  construction, $ \alpha$ is bounded by $R_0$ in $L^\infty$. 
\vs

We finally prove the H\"{o}lder continuity in time of $U$.
 Fix $(t_0,m_0)\in [0,T]\times \sub$. Choosing $\alpha=\mu=0$ and letting $m$ solve the heat equation with initial condition $m_0$, we have by a dynamic programming argument that, for any $h\in (0,T-t_0)$,  
$$
U(t_0,m_0)\leq \int_{t_0}^{t_0+h}\int_{T^d} L(x,0,m_t)dt + U(t_0+h,m_{t_0 + h}) \leq C\sqrt{h} + U(t_0+h,m_0), 
$$
where for the second inequality we used the regularity of $U$ with respect to the measure argument and the fact that $\bd_1(m_{t_0 + h},m_0) \leq C \sqrt{h}$. 
\vs
We now prove the opposite inequality. Let $(m,\alpha, \mu)$ be optimal for $U(t_0,m_0)$, with $\alpha$ bounded. Then, again, dynamic programming implies that 
\begin{equation}\label{kjhkjazen}
\begin{split}
 U(t_0,m_0)& =    U(t_0 + h, m_{t_0 + h}) + \int_{t_0}^{t_0 + h} \int_{\T^d} L\big( x, \alpha(t,x) , m_t \big) dm_t(x) dt \\
    & \qquad \qquad  +\iint_{[t_0,t_0 + h]\times \T^d} \Psi(x,m_{t-}) \mu(dt,dx). 
\end{split}
\end{equation}
Now let $\tilde{m}$ satisfy 
\begin{align*}
    \partial_t \tilde{m} - \Delta \tilde{m} + \text{div}(\tilde{m} \alpha) = 0 \ \text{in} \ (t_0,T] 
    \ \ \text{and} \ m_{t_0} =m_0,
\end{align*}
and note  that $\tilde m_t \geq m_t$ for any $t$. Then, from the uniform bound of $\alpha$, we find another constant $C$
depending only on the  uniform bound of $\alpha$ such that, for any $t_0\leq s\leq t\leq T$, 
\be\label{ineqtildemst}
    \bd_1(\tilde{m}_t, \tilde{m}_s) \leq C \sqrt{t-s}. 
\ee
It follows, from  \eqref{kjhkjazen} and the boundedness of $L$, that 
\begin{align*}
&    U(t_0 + h, m_0) - U(t_0,m_0) \\
&  \qquad = U(t_0 + h, \wt{m}_{t_0 + h}) - U(t_0, m_0) + U(t_0 + h, m_{t_0}) - U(t_0 + h, \wt{m}_{t_0 + h}) 
    \\
    &\qquad\leq U(t_0 + h, \wt{m}_{t_0 + h}) - U(t_0 + h, m_{t_0 + h}) + U(t_0 + h,m_{t_0}) - U(t_0 + h, \wt{m}_{t_0 + h}) 
    \\
    &\qquad\qquad + Ch - \iint_{[t_0, t_0 + h]\times \T^d} \Psi(x,m_{t-}) d\mu(t,x).
\end{align*} 
Then using that $U$ is $\Psi-$monotone and uniformly Lipschitz continuous  in the $m$ variable and \eqref{ineqtildemst}, we find
\begin{align*}
&    U(t_0 + h, m_0) - U(t_0,m_0) \\
    &\qquad\leq \int_{\T^d} \Psi(x,\wt{m}_{t_0 + h}) d(\wt{m}_{t_0 + h} - m_{t_0 + h})
   - \iint_{[t_0, t_0 + h]\times \T^d} \Psi(x,m_{t-}) d\mu(t,x) + C \sqrt{h}  
    \\
    &\qquad\leq \int_{\T^d} \Psi(x, \wt{m}_{t_0 + h}) d(m_{t_0} - m_{t_0 + h}) - \iint_{[t_0, t_0 + h]\times \T^d} \Psi(x,m_{t-}) d\mu(t,x)
    \\
    &\qquad\qquad+ C \sqrt{h} + C \|\Psi( \cdot, \wt{m}_{t_0 + h})\|_{W^{1,\infty}} \bd_1(\wt{m}(t_0 + h), m_{t_0}) 
    \\
    &\qquad\leq  \int_{\T^d} \Psi( x, \wt{m}_{t_0 + h}) d(m_{t_0} - m_{t_0 + h}) - \iint_{[t_0, t_0 + h]\times \T^d} \Psi(x,m_{t-}) d\mu(t,x) + C \sqrt{h}. 
\end{align*}
That  $\Psi$ is non-increasing in $m$ also yields that 
\begin{align*}
   \iint_{[t_0, t_0 + h]\times \T^d} \Psi(x,m_{t-}) d\mu(t,x) & \geq   \iint_{[t_0, t_0 + h]\times \T^d} \Psi(x,\wt{m}_t) d\mu(t,x)
    \\
    &\geq   \iint_{[t_0, t_0 + h]\times \T^d} \Psi(x,\wt{m}_{t_0 + h}) d\mu(t,x) - C \sqrt{h}, 
\end{align*}
where in the last line we used the fact that  the total variation of $\mu$ is bounded by $1$  and that $\bd_1(\wt{m}(t_0 + h), \wt{m}_t) \leq C \sqrt{h}$. 
\vs
Combining the above estimates and using that, in view of \eqref{ath1}, the time-independent test function
\begin{align*}
    \phi(t,x) = \Psi(x,\wt{m}(t_0 + h))
\end{align*}
is of class $C^2$, 
we have 
\begin{align*}
    U(t_0 + h, m_{0}) - U(t_0,m_0) \leq C \sqrt{h} + \int_{\T^d} \phi(t_0,x) d(m_{t_0} - m_{t_0 + h})(x) - \iint_{[t_0, t_0 + h]\times \T^d} \phi(t,x) d\mu(t,x).
\end{align*}

Finally, the equation satisfied by $m$ implies that 
\begin{align*}
    \int_{\T^d} \phi d(m_{t_0 + h} - m_{t_0})  &  = \int_{t_0}^{t_0+h}\int_{\T^d} \Big(\partial_t \phi +\Delta \phi + \alpha \cdot  D\phi\Big) m- \iint_{[t_0,t_0+h]\times\T^d} \phi \mu.
\end{align*}
Since $\phi$ is $C^2$ and $\alpha$ is bounded, we find 
\begin{align*}
    \int_{\T^d} \phi d(m_{t_0 + h} - m_{t_0}) = - \iint_{[t_0,t_0  + h] \times \T^d} \phi(t,x) d\mu(t,x) + O(h), 
\end{align*}
and, thus, 
\begin{align*}
    U(t_0 + h, m_{0}) - U(t_0,m_0) \leq C \sqrt{h}. 
\end{align*}
\end{proof}

To conclude we prove the following result which removes the need to assume \eqref{assump.Gsmooth}.
\begin{corollary} \label{cor.regU}
Assume   \eqref{summarize}. Then the conclusions of Proposition \ref{prop.reguU} still hold.
\end{corollary}

In other words, the estimate~\eqref{ulip} in Theorem~\ref{thm.main.lip} holds. 

\begin{proof}
    We apply Lemma \ref{lem.envelopeLip} and then Proposition \ref{prop.mollification} to produce a sequence of terminal conditions $(G^k)_{k\in\N}$ which are Lipschitz uniformly in $k$, satisfy \eqref{assump.Gsmooth}, and such that $G^k \to G_{\Psi}$ uniformly. 
 \vs
    Let $U^k$ be defined exactly like $U$, but with $G^k$ replacing $G$. Then Proposition \ref{prop.reguU} shows that there exists a constant $C > 0$ which is independent of $k$ such that, for all  $ m,n \in \sub$  and  $t,s \in [0,T]$,
    \begin{align*}
        \big|U^k(t,m) - U^k(s,n) \big| \leq C \Big( \bd(m,n) + |t-s|^{1/2} \Big). 
    \end{align*}
    Since the $U^k$ converges pointwise to $U$, we deduce that $U$ satisfies the same bound. 
   \vs
    Similarly, Proposition \ref{prop.reguU} shows that there is a constant $R_0 > 0$ such that for each $k \in \N$ and each $(t_0,m_0) \in [0,T) \times \sub$, there exists an optimizer $(m^k,\alpha^k, \mu^k) \in \cA_{t_0,m_0}$ for the problem defining $U^k(t_0,m_0)$ which satisfies $\|\alpha^k\|_{\infty} \leq R_0$.
\vs   
     A compactness  argument similar to  the proof of Proposition \ref{prop.reguU} yields that, along a subsequence, the $(m^k,\alpha^k, \mu^k)$'s  converge to an optimizer $(m,\alpha,\mu)$ for the problem defining $U(t_0,m_0)$, which still satisfies the bound $\| \alpha \|_{\infty} \leq R_0$. 
\end{proof}

\section{Viscosity solutions}  \label{sec.viscosity}

In this section, we prove a comparison principle for \eqref{eq.HJstoppingPSI}. Unfortunately, this will require the Hamiltonian $H$ to be globally Lipschitz. This is the reason for the Lipschitz condition for  the uniqueness result in Theorem \ref{thm.uniqueness}. On the other hand, for the comparison result we will not require $H$ to derive from a Lagrangian via \eqref{ath200}, that is, we do not need $H$ to be convex in $p$, and we will not need the regularity conditions appearing in \eqref{phl1}. Thus, in this section we assume that
\be\label{ath100}
H \text{ is globally Lipschitz continuous, and bounded below},
\ee
in addition to the conditions \eqref{phl2}, \eqref{phl3} on $\Psi$ and $G$, respectively. For notational convenience, we record this as the following hypothesis:
\be\label{hyp.comparison}
\text{ \eqref{phl2}, \eqref{phl3}, and \eqref{ath100} all hold.}
\ee
\vs
In what follows, we  first introduce the definition of viscosity sub- and super-solutions to \eqref{eq.HJstoppingPSI} and prove a comparison principle when $H$ is Lipschitz. Then we show that the value function $U$ is a viscosity solution of \eqref{eq.HJstoppingPSI}, and, finally, complete the proof of the uniqueness result, Theorem \ref{thm.uniqueness}. 

\begin{definition} \label{def.testfunction}
    A map $\Phi : [0,T] \times \sub \to \R$ is a smooth test function,  if $\partial_t \Phi$ and $\dfrac{\delta \Phi}{\delta m}$ exist and are continuous, $\dfrac{\delta \Phi}{\delta m}(t,m, \cdot)\in C^2$ for any $(t,m)$ and, for any $t\in [0,T]$, the maps
    \begin{align*}
     \sub\times \T^d \ni (m,x) \mapsto D_x\frac{\delta \Phi}{\delta m}(t,m,x)
     \; \text{and}\;      \sub\times \T^d \ni (m,x) \mapsto D^2_x \dfrac{\delta \Phi}{\delta m}(t,m,x) 
    \end{align*}
    are Lipschitz continuous and bounded, uniformly with respect to $t$.
\end{definition}
For the definition of the super-solution we need  an additional property of the test functions. 
\begin{definition} \label{def.locallydecreasing}
Let $\Phi$ be a smooth test function, and $(\ov t, \ov m) \in [0,T) \times \sub$. For $\gamma \in \R$, we say that $\Phi$ is locally $\gamma$-$\Psi$-non-increasing at $(\ov t, \ov m)$, if there exist $\eps > 0$ such that, if $|t - \ov t| + \bd(m,\ov m) < \eps$, then,  for any $n \leq m$, 
\begin{align*}
    \Phi(t,m) \leq \Phi(t,n) + \int_{\T^d} \Psi( x,m) d(m - n)(x) - \gamma (m - n)(\T^d).
\end{align*}
A test function  $\Phi$ is said to be locally $\Psi$-decreasing at $(\ov t, \ov m)$, if $\Phi$ is $\gamma$-$\Psi$-non-increasing at $(\ov t, \ov m)$ for some $\gamma > 0$. 
\end{definition}

We turn now to the notions of viscosity sub- and super-solutions to \eqref{eq.HJstoppingPSI}.

\begin{definition}\label{def.viscositysoln} (i)~A  continuous map $V:[0,T]\times \sub \to \R$ is a viscosity sub-solution to  \eqref{eq.HJstoppingPSI}, if 
it is  $\Psi-$non-increasing,  $V(T,\cdot)\leq G_{\Psi}$, and, for any smooth test function $\Phi:[0,T]\times \sub(\T^d) \to \R$ such that  $V-\Phi$ has a maximum at some  $(\ov t, \ov m)\in [0, T)\times \sub(\T^d)$, 
$$
-\partial_t \Phi(\ov t, \ov m) - \int_{\T^d}  \Delta_x \dfrac{\delta \Phi}{\delta m}(\ov t,\ov m,x) d \ov m(x)  +\int_{\T^d} H \Big(x, D_x \dfrac{\delta \Phi}{\delta m}(\ov t,\ov m,x),\ov m \Big)d \ov m(x)  \leq0.
$$
(ii)~A continuous map $V:[0,T]\times \sub \to \R$ is a viscosity super-solution to \eqref{eq.HJstoppingPSI},  if $V(T,\cdot) \geq G_{\Psi}$, and, for any smooth test function $\Phi$ and $(\ov t, \ov m) \in [0,T) \times \sub$ such that $V - \Phi$ has a minimum at $(\ov t, \ov m)$ and $\Phi$ is locally $\Psi$-decreasing at $(\ov t ,\ov m)$,
\begin{align*}
    - \partial_t \Phi(\ov t, \ov m) - \int_{\T^d} \Delta_x \frac{\delta \Phi}{\delta m}(\ov t, \ov m,x) d \ov m(x) + \int_{\T^d} H \Big(x, D_x \frac{\delta \Phi}{\delta m}(\ov t, \ov m,x), \ov m \Big) d \ov m(x)  \geq 0.
\end{align*}
(iii)~ A continuous map $V : [0,T] \times \sub \to \R$ is a viscosity solution to \eqref{eq.HJstoppingPSI} if it is both a viscosity sub-solution and a viscosity super-solution.
\end{definition}

\begin{remark}
        We comment briefly on the idea behind the definition of viscosity super-solution. Recall that formally, we only expect the value function $U$ to be a super-solution in the domain $\mathcal{D}$, defined in \eqref{phl11.1}. If $\Phi$ is locally $\Psi$-decreasing at $(\ov t, \ov m)$ and $U - \Phi$ has a minimum at $(\ov t, \ov m)$, then there exists $\delta > 0$ such that for any $m' \in \sub$ with $m' < m$, we have 
        \begin{align*}
            U(\ov t, \ov m) - U(\ov t, m') \leq \Phi(\ov t, \ov m) - \Phi(\ov t, m') \leq \int_{\T^d} \Psi(\ov m, x) d(\ov m - m')(x) - \delta (\ov m - m')(\T^d).
        \end{align*}
        In particular, since $\delta > 0$ and $m' < m$, this means that
        \begin{align*}
            U(\ov t, \ov m) < U(\ov t, m') + \int_{\T^d} \Psi(\ov m, x) d(\ov m - m')(x), 
        \end{align*}
        i.e. $(\ov t, \ov m) \in \mathcal{D}$. In other words, the fact that the test function $\Phi$ is strictly $\Psi$-decreasing at $(\ov t, \ov m)$ and $U - \Phi$ is minimized at $(\ov t, \ov m)$ guarantees that $(\ov t, \ov m) \in \mathcal{D}$. The requirement that $\Phi$ is \textit{locally} strictly $\Psi$-decreasing at $(\ov t, \ov m)$, i.e. strictly $\Psi$-decreasing in a neighborhood of $(\ov t, \ov m)$, is useful for certain perturbative arguments used in establishing the comparison principle.
\end{remark}

The main result of this section is the following comparison principle.

\begin{theorem} \label{thm.comparison} Assume \eqref{hyp.comparison} and let $V^-$ and $V^+$ be respectively a viscosity sub-solution and a viscosity super-solution to \eqref{eq.HJstoppingPSI}. Then $V^- \leq V^+$. 
\end{theorem}

\subsection*{A less regular class of test functions} To prove the comparison result, we need to work with less regular test functions. The following proposition explains how to pass from smooth test functions to the class of test functions used in the proof of Theorem~\ref{thm.comparison}. 
 
\begin{proposition} \label{prop.subsol.equiv} (i)~Assume \eqref{hyp.comparison} and let $V$ be a viscosity sub-solution to \eqref{eq.HJstoppingPSI}. There exists a constant $C$ depending only on the Lipschitz constant of $H$ such that, if $[0,T] \times H^{-1} \ni (t,q) \mapsto \Phi(t,q)\in  \R$ is $C^1$ in $(t,q)$, with the derivative $D_{H^{-1}} \Phi$ being Lipschitz continuous and of linear growth in $q$ uniformly in $t$, and if, for some $\delta >  0$, 
$$(t,m)\to V(t,m)-\Phi(t,m)-\delta \|m\|^2_2$$ has a maximum at a point $(\bar t, \bar m)\in [0, T)\times (H^{-1}\cap \sub)$, then $\bar m\in H^1$ and 
\be\label{lsfdkedfnfrfPSI.0}
\begin{split}
& -\partial_t \Phi(\bar t, \bar m) +\int_{\T^d}  D_x D_{H^{-1}}\Phi(t,\bar m,x)\cdot D\bar m(x)dx +\delta \int_{\T^d} |D\bar m|^2 dx \\
& \qquad\qquad +\int_{\T^d} H \big(x, D_x D_{H^{-1}}\Phi(\bar t,\bar m,x), \bar m \big) d \bar m(x)  \leq C\delta \|\ov{m}\|_2^2.
\end{split}
\ee
(ii)~Assume \eqref{hyp.comparison} and let $U$ is a viscosity super-solution to \eqref{eq.HJstoppingPSI}.  There exists a constant $C$ depending only on the Lipschitz constant of $H$ such that, if $\Phi_1$ is a smooth test function,  $\Phi_2 : [0,T] \times H^{-1} \to \R$ is non-increasing and satisfies the regularity condition placed on $\Phi$ in (i), 
 $(\ov{t}, \ov{m}) \in [0,T) \times (L^2 \cap \sub)$ is a minimum point of 
 $$(t,m)\to V(t,m)-\Phi_1(t,m) - \Phi_2(t,m) +\delta \|m\|^2_2,$$
 for some $\delta \in (0,1)$, 
 and, finally,  $\Phi_1$ is locally $\Psi$-decreasing at $(\ov t, \ov m)$,  then, $\ov m \in H^1$ and 
\be
\begin{split}\label{lsfdkedfnfrfPSI}
&-\partial_t \Phi_1(\bar t, \bar m) - \partial_t \Phi_2(\ov t, \ov m) + \int_{\T^d} \Big( D_x \frac{\delta \Phi_1}{\delta m}(t, \ov m,x) + D_x D_{H^{-1}} \Phi_2(t,\bar m,x) \Big)\cdot D\bar m(x)dx\\
&    -\delta \int_{\T^d} |D\bar m|^2 dx+\int_{\T^d} H\bigg(x, \Big( D_x\frac{ \delta \Phi_1}{\delta m}(t, \ov m,x) + D_x D_{H^{-1}} \Phi_2(t,\bar m,x) \Big),\bar m\bigg) d\bar m(x)\\
&\qquad \qquad  \geq -C\delta \|\ov{m}\|_2^2.
\end{split}
\ee
\end{proposition}

\begin{proof} Since the proof of (i) is a simpler version of the proof of (ii), here we present the arguments for the latter.
\vs
Modifying, if necessary,  $\Phi_1$ by a smooth perturbation, we may assume without loss of generality that the  minimum  $(\bar t, \bar m)$ is unique; note that this can be done while maintaining the fact that $\Phi_1$ is locally $\Psi$-decreasing at $(\ov t, \ov m)$. 
\vs
Let $(\rho_{\eta_1})_{0 < \eta_ < 1}$ and $(\rho_{\eta_2})_{0 < \eta_2 < 1}$ be standard approximations to the identity 
and consider the optimization problem 
\begin{align} \label{opt.eta1.eta2}
   M_{\eta_1,\eta_2} \coloneqq  \inf_{t,m} \Big\{ V(t,m) - \Phi_1(t,m) - \Phi_2(t,m * \rho_{\eta_1}) + \delta \|m * \rho_{\eta_2}\|_2^2 \Big\}.
\end{align}
\vs
We now claim that if $(t_{\eta_1,\eta_2}, m_{\eta_1,\eta_2})$ is any optimizer for $M_{\eta_1,\eta_2}$, then $\lim_{\eta_1 \to 0} \lim_{\eta_2 \to 0} (t_{\eta_1,\eta_2},m_{\eta_1,\eta_2}) = (\ov t, \ov m)$.  This follows from a straightforward compactness argument, which we now sketch. First, one can use the compactness of the metric space $\big(\sub, \bd\big)$, the continuity of $(t,m) \mapsto \Phi_2\big(t, m* \rho_{\eta_1}\big)$ with respect to $\bd$ for fixed $\eta_1$, and the lower semi-continuity of $m \mapsto \|m\|_2^2$ with respect to $\bd$, to deduce that 
\begin{align*}
    M_{\eta_1,\eta_2} \xrightarrow{ \eta_2 \downarrow{0} } M_{\eta_1} \coloneqq \inf_{t,m} \Big\{ V(t,m) - \Phi_1(t,m) - \Phi_2(t,m * \rho_{\eta_1}) + \delta \|m \|_2^2 \Big\},
\end{align*}
and for any sequence $\eta_2^n \downarrow 0$, any limit point $(t_{\eta_1},m_{\eta_2})$ of $\big(t_{\eta_1,\eta_2^n},m_{\eta_1,\eta_2^n}\big)$ with respect to $\bd$ is a minimizer for the problem defining $M_{\eta_1}$. 

Next, if $(t_{\eta_1},m_{\eta_1})$ denotes any optimizer of the problem defining $M_{\eta_1}$, one can us the fact that the sub-level sets of
\begin{align*}
    \big(t,m\big) \mapsto V(t,m) - \Phi_1(t,m) - \Phi_2\big(t, m * \rho_{\eta_1} \big) + \delta \|m * \rho_{\eta_1}\|_2^2
\end{align*}
are bounded in $L^2$, hence with respect to the weak topology of $L^2$. Thus another compactness argument, together with the lower semi-continuity of 
\begin{align*}
\big(t,m\big) \mapsto V(t,m) - \Phi_1(t,m) - \Phi_2\big(t, m\big) + \delta \|m \|_2^2
\end{align*}
with respect to the topology of $H^{-1}$, hence with respect to the weak topology of $L^2$, shows that
\begin{align*}
    M_{\eta_1} \xrightarrow{ \eta_1 \downarrow{0} } M \coloneqq \inf_{t,m} \Big\{ V(t,m) - \Phi_1(t,m) - \Phi_2(t,m) + \delta \|m \|_2^2 \Big\},
\end{align*}
and that any limit point $(\hat{t},\hat{m})$ of $(t_{\eta_1},m_{\eta_1})$ is an optimizer for the problem defining $M$. But $(\ov t, \ov m)$ is the unique optimizer for this problem, so $(t_{\eta_1},m_{\eta_1}) \to (\ov t, \ov m)$.

Combining this with the fact that every subsequential limit point of $(t_{\eta_1,\eta_2},m_{\eta_1,\eta_2})$ as $\eta_2 \to 0$ is an optimizer for $M_{\eta_1}$, we deduce that $$ \lim_{\eta_1 \to 0} \lim_{\eta_2 \to 0} \big(t_{\eta_1,\eta_2},m_{\eta_1,\eta_2}\big) = (\ov t, \ov m),$$
as claimed.

In particular, if we choose first $\eta_1$ and then $\eta_2$ small enough, then 
$$t_{\eta_1,\eta_2} < T \ \ \text{and} \ \ \Phi_1 \  \text{is locally $\Psi$-decreasing at} \  (t_{\eta_1,\eta_2}, m_{\eta_1,\eta_2}).$$
The mollified test function
\begin{align*}
 [0,T]\times   \sub  \ni (t,m) \mapsto \Phi_{\eta_1,\eta_2} =\Phi_1(t,m) + \Phi_2(t,m * \rho_{\eta_1}) - \delta \|m * \rho_{\eta_2}\|_2^2
\end{align*}
is smooth and  
\begin{align*}
    \frac{\delta \Phi_{\eta_1, \eta_2}}{\delta m}(t,m,\cdot) = \frac{\delta \Phi_1}{\delta m}(t,m,\cdot ) + \rho_{\eta_1} * D_{H^{-1}} \Phi_2(t,m * \rho_{\eta_1}, \cdot) - 2\delta \rho_{\eta_2} * \rho_{\eta_2} * m.
\end{align*}
Moreover, since $\Phi_1$ is locally $\Psi$-decreasing at $(t_{\eta_1,\eta_2},m_{\eta_1,\eta_2})$ (provided we choose $\eta_1$ and then $\eta_2$ small enough), and $\Phi_2$ and $m \mapsto - \| m * \rho_{\eta_2}\|_2^2$ are both non-increasing, we see that $\Phi_{\eta_1,\eta_2}$ is locally $\Psi$-decreasing at $(t_{\eta_1,\eta_2},m_{\eta_1,\eta_2})$ for $\eta_1$ and then $\eta_2$ small enough. 
\vs
From the definition of viscosity super-solution, we find that, at $(t,m) = (t_{\eta_1,\eta_2}, m_{\eta_1,\eta_2})$, we have
\begin{align*}
    &- \partial_t \Phi_1(t,m) -  \partial_t \Phi_2(t,m * \rho_{\eta_1})
    \\
    &\quad - \int_{\T^d} \Delta_x \Big( \frac{\delta \Phi_1}{\delta m}(t,m, \cdot) + \rho_{\eta_1} * D_{H^{-1}} \Phi_2(t, m *\rho_{\eta_1},x) - 2 \delta \rho_{\eta_2} * \big(\rho_{\eta_2} * m\big) \Big) dm
    \\
    &\quad + \int_{\T^d} H\bigg(x, D_x\Big( \frac{\delta \Phi_1}{\delta m}(t,m,x) + \rho_{\eta_1} * D_{H^{-1}} \Phi_2(t,m * \rho_{\eta_1}, \cdot) - 2\delta \rho_{\eta_2} * \rho_{\eta_2} * m \Big) ,m \bigg) dm
     \geq 0.
\end{align*}

Notice that 
\begin{align*}
    \int_{\T^d} \Delta_x \Big(\rho_{\eta_1} * D_{H^{-1}} \Phi_2(t,m * \rho_{\eta_1},\cdot)\Big) dm &= 
   \int_{\T^d} \Delta_x \Big( D_{H^{-1}} \Phi_2(t,m * \rho_{\eta_1},\cdot)\Big) \rho_{\eta_1} *m  dx\\
    &= - \int_{\T^d} D_x \Big(\rho_{\eta_1} * \frac{\delta \Phi}{\delta m}(t,m * \rho_{\eta_1},\cdot)\Big) D_x\rho_{\eta_1} * m dx, 
\end{align*}
and, similarly, 
\begin{align*}
    \int_{\T^d} \Delta_x \Big(\rho_{\eta_2} * \big(\rho_{\eta_2} * m\big) \Big) dm = -  \int_{\T^d} |D \big( \rho_{\eta_2} * m \big)|^2 dx.
\end{align*}

Finally, using the Lipschitz continuity of $H$, we get
  
\begin{align*}
 &  \int_{\T^d} H\bigg(x, D_x\Big( \frac{\delta \Phi_1}{\delta m}(t,m,x) + \rho_{\eta_1} * D_{H^{-1}} \Phi_2(t,m * \rho_{\eta_1}, \cdot) - 2\delta \rho_{\eta_2} * \rho_{\eta_2} * m \Big) ,m \bigg) dm
    \\
    &{ \leq} \int_{\T^d}  H\Big( x, D_x\Big( \frac{\delta \Phi_1}{\delta m}(t,m,x) + \rho_{\eta_1} * D_{H^{-1}} \Phi_2(t,m * \rho_{\eta_1}, \cdot) ,m \Big) dm + C \delta \int_{\T^d} \Big|  \rho_{\eta_2} * D_x(\rho_{\eta_2} * m) \Big| dm
    \\
    &{\leq}  \int_{\T^d}  H\Big( x, D_x\Big( \frac{\delta \Phi_1}{\delta m}(t,m,x) + \rho_{\eta_1} * D_{H^{-1}} \Phi_2(t,m * \rho_{\eta_1}, \cdot) ,m \Big) dm  + C \delta \int_{\T^d} \rho_{\eta_{2}} * \Big|D_x(\rho_{\eta_2} * m) \Big| dm
    \\
    &=  \int_{\T^d}  H\Big( x, D_x\Big( \frac{\delta \Phi_1}{\delta m}(t,m,x) + \rho_{\eta_1} * D_{H^{-1}} \Phi_2(t,m * \rho_{\eta_1}, \cdot) ,m \Big) dm  + C \delta \int_{\T^d}\Big|D_x(\rho_{\eta_2} * m) \Big| m * \rho_{\eta_2} dx
    \\
    & \leq   \int_{\T^d}  H\Big( x, D_x\Big( \frac{\delta \Phi_1}{\delta m}(t,m,x) + \rho_{\eta_1} * D_{H^{-1}} \Phi_2(t,m * \rho_{\eta_1}, \cdot) ,m \Big) dm  
    \\
    &\qquad +  \delta \int_{\T^d}\Big|D_x(\rho_{\eta_2} * m) \Big|^2 dx {+} C \delta \int_{\T^d} \Big|m * \rho_{\eta_2}\Big|^2 dx. 
\end{align*} 
Combining  the last three computations  we find that  $(t,m) = (t_{\eta_1,\eta_2},m_{\eta_1,\eta_2})$ satisfies
\be\label{ubound}
\begin{split}
  - \partial_t \Phi_1(t,m) &- \partial_t \Phi_2(t,m* \rho_{\eta_1}) - \int_{\T^d} \Delta_x \frac{\delta \Phi_1}{\delta m}(t,m,x) dm(x)\\
 & + \int_{\T^d} D_x \Big(\rho_{\eta_1} * \frac{\delta \Phi_2}{\delta m}(t,m * \rho_{\eta_1},\cdot)\Big) D_x\rho_{\eta_1} * m 
- \delta \int_{\T^d} \Big| D_x \rho_{\eta_2} * m \Big|^2 dx    \\
& + \int_{\T^d}  H\Big( x, D_x\Big( \frac{\delta \Phi_1}{\delta m}(t,m,x) + \rho_{\eta_1} * D_{H^{-1}} \Phi_2(t,m * \rho_{\eta_1}, \cdot) ,m \Big) dm \\
  &\quad \geq -C \delta \int_{\T^d} \Big|m * \rho_{\eta_2}\Big|^2 dx.
\end{split}
\ee
Since  $\delta > 0$ is fixed, we see that there is a constant $C > 0$ such that, for each $\eta_1,\eta_2 > 0$, 
 $\|\rho_{\eta_2} * m_{\eta_1,\eta_2}\|_2^2 \leq C$, and, thus, from \eqref{ubound}, 
\begin{align*}
    \|D_x\big(\rho_{\eta_2} * m_{\eta_1,\eta_2}\big)\|_2^2 \leq C.
\end{align*}
As a consequence, we note that, for all $\eta_1$ small enough, there exists $(t_{\eta_1},m_{\eta_1})$ such that $(t_{\eta_1},m_{\eta_1})$ is an optimizer for 
\begin{align} \label{eta1.problem}
    \inf_{t,m} \Big\{ U(t,m) - \Phi_1(t,m) - \Phi_2(t,m * \rho_{\eta_1}) + \delta \|m\|_2^2 \Big\}, 
\end{align}
and, along an appropriate subsequences, in the limit $\eta_2\to 0$, 
\begin{align*}
 m_{\eta_1,\eta_2} \to  m_{\eta_1} \ \text{weak-}*\ \ \text{and} \ \  \rho_{\eta_2} * m_{\eta_1,\eta_2} \to m_{\eta_1} \ \text{weakly in $H^1$ and  strongly in $L^2$}.
\end{align*}

Since the weak-$*$ convergence of $m_{\eta_1,\eta_2}$ to $m_{\eta_1}$ implies that $\rho_{\eta_1} * m_{\eta_1, \eta_2}$ converges to $\rho_{\eta_1} * m_{\eta_1}$ in $C^k$ for any $k$, we note that, as $\eta_2\to 0$, 
\begin{equation} \label{linearterm}
\begin{split}
    \int_{\T^d} D_x & \Big( \rho_{\eta_1} * D_{H^{-1}} \Phi_2( \rho_{\eta_1} * m_{\eta_1,\eta_2}, \cdot) \Big) \cdot D_x \rho_{\eta_1} * m_{\eta_1,\eta_2} \\
 &   \to \int_{\T^d} D_x \Big( \rho_{\eta_1} *  D_{H^{-1}} \Phi_2 (\rho_{\eta_1} * m_{\eta_1}, \cdot) \Big) \cdot D_x \rho_{\eta_1} * m_{\eta_1}. 
\end{split}
\end{equation}

%

In view of the uniform convergence, as $\eta_2\to 0$, 
\begin{align*}
    D_x \rho_{\eta_1} *  D_{H^{-1}} \Phi_2 (t,m_{\eta_1,\eta_2} * \rho_{\eta_1},\cdot) \to D_x \rho_{\eta_1} *  D_{H^{-1}} \Phi_2 (t,m_{\eta_1} * \rho_{\eta_1},\cdot)  
\end{align*}
and the weak-$*$ convergence $m_{\eta_1,\eta_2} \to m_{\eta_1}$, we get 
\begin{equation} \label{hamiltonianterm}
\begin{split}
& \lim_{\eta_2 \to 0} \int_{\T^d}  H\Big( x, D_x\Big( \frac{\delta \Phi_1}{\delta m}(t_{\eta_1,\eta_2},m_{\eta_1,\eta_2},x) + \rho_{\eta_1} * D_{H^{-1}} \Phi_2(t,m_{\eta_1,\eta_2} * \rho_{\eta_1}, \cdot) ,m_{\eta_1,\eta_2} \Big) dm_{\eta_1,\eta_2}   \\ 
& \qquad\qquad =  \int_{\T^d}  H\Big( x, D_x\Big( \frac{\delta \Phi_1}{\delta m}(t_{\eta_1},m_{\eta_1},x) + \rho_{\eta_1} * D_{H^{-1}} \Phi_2(t,m_{\eta_1} * \rho_{\eta_1}, \cdot) ,m_{\eta_2} \Big) dm_{\eta_1}.
\end{split}
\ee
We  also note that because $\rho_{\eta_2} * m_{\eta_1,\eta_2} \to m_{\eta_1}$ weakly in $H^1$, we also have
\begin{align} \label{lscterm}
    \int_{\T^d} \Big| D_x  m_{\eta_1} \Big|^2 dx \leq \liminf_{\eta_2 \to 0} \int_{\T^d} \Big|D_x \rho_{\eta_2} * m_{\eta_1,\eta_2}\Big|^2 dx. 
\end{align}
Combining \eqref{linearterm}, \eqref{hamiltonianterm}, \eqref{lscterm} and the fact that  $m_{\eta_1,\eta_2} * \rho_{\eta_2}$ converges strongly in $L^2,$ we find that $(t,m) = (t_{\eta_1},m_{\eta_1})$ satisfies $m \in H^1$, and
\begin{equation} \label{ubound2}
\begin{split}
  - \partial_t \Phi_1(t,m) &- \partial_t \Phi_2(t,m* \rho_{\eta_1}) - \int_{\T^d} \Delta_x \frac{\delta \Phi_1}{\delta m}(t,m,x) m(dx) \\
  &+ \int_{\T^d} D_x \Big(\rho_{\eta_1} * D_{H^{-1}} \Phi_2(t,m * \rho_{\eta_1},\cdot)\Big) D_x\rho_{\eta_1} * m - \delta \int_{\T^d} \Big| D_x  m \Big|^2 dx 
 \\
  & + \int_{\T^d}  H\Big( x, D_x\Big( \frac{\delta \Phi_1}{\delta m}(t,m,x) + \rho_{\eta_1} * D_{H^{-1}} \Phi_2(t,m * \rho_{\eta_1}, \cdot) ,m \Big) dm  \\
  &\quad \geq -C \delta \int_{\T^d} m^2 dx.
\end{split}
\ee

Since $(t_{\eta_1},m_{\eta_1})$ is an optimizer for \eqref{eta1.problem}, there is  an independent of $\eta_1$ constant $C$ such that $\|m_{\eta_1}\|_2^2 \leq C$, and, thus,   \eqref{ubound2} yields that $m_{\eta_1}$ is bounded in $H^1$. Thus, setting $(t,m) = (\ov{t},\ov{m})$ for simplicity, we find that, as $\eta_1\to 0$, 
\begin{align*}
    m_{\eta_1} \to   m \ \text{weakly in $H^1$ and  strongly in $L^2$} \ \ \text{and} \ \   \rho_{\eta_1} * m_{\eta_1} \to m \ \text{weakly in $H^1$ and  strongly in $L^2$}.
\end{align*}
It follows that, as $\eta_1\to 0$, 
\[ D_{H^{-1}} \Phi_2 (t,m_{\eta_1} * \rho_{\eta_1},\cdot) \to  D_{H^{-1}} \Phi_2(t,m,\cdot) \ \text{in} \ H^1.\]
Since 
\begin{align*}
    \| \rho_{\eta_1} * & D_x  D_{H^{-1}} \Phi_2(t,m_{\eta_1} * \rho_{\eta_1},\cdot) - D_x D_{H^{-1}} \Phi_2(t,m,\cdot)\|_2 
    \\
    &\leq \| \rho_{\eta_1} * D_x  D_{H^{-1}} \Phi_2(t,m_{\eta_1} * \rho_{\eta_1},\cdot) - \rho_{\eta_1} * D_xD_{H^{-1}} \Phi_2 (t,m,\cdot)\|_2 
    \\
    &\qquad \qquad + \|  \rho_{\eta_1} * D_x D_{H^{-1}} \Phi_2(t,m,\cdot) -D_x D_{H^{-1}} \Phi_2(t,m,\cdot) \|_2
    \\
    &\leq C \|  D_x  D_{H^{-1}} \Phi_2(t,m * \rho_{\eta_1},\cdot) -  D_x D_{H^{-1}} \Phi_2(t,m,\cdot)\|_2\\
    & \qquad \qquad + \|  \rho_{\eta_1} * D_x D_{H^{-1}} \Phi_2(t,m,\cdot) -D_x D_{H^{-1}} \Phi_2 (t,m,\cdot) \|_2, 
\end{align*}
and, for any fixed $f \in L^2$,  $\eta_1 * f \to f$ in $L^2$ as $\eta_1\to 0$, it follows that 
\[ D_x  \Big( \rho_{\eta_1} * D_{H^{-1}} \Phi_2(t,m_{\eta_1} * \rho_{\eta_1},\cdot) \Big)\to  D_x D_{H^{-1}} \Phi_2(t,m,\cdot) \ \text{strongly in $L^2$}.\]
Together with the strong convergence in $L^2$ and the weak convergence in $H^1$ of $m_{\eta_1}$ to $m$,  this is enough to conclude that 
\begin{align*}
    \int_{\T^d} D_x \Big(\rho_{\eta_1} * D_{H^{-1}} \Phi_2(t,m_{\eta_1} * \rho_{\eta_1},\cdot)\Big) D_x\rho_{\eta_1} * m_{\eta_1} dx \to  
     \int_{\T^d} D_x D_{H^{-1}} \Phi_2(t,m ,\cdot) \cdot D_x m dx, 
\end{align*}
as well as
\begin{align*}
    &\int_{\T^d} H \Big(x, D_x \frac{\delta \Phi_1}{\delta m}(t_{\eta_1},m_{\eta_1},x) + \rho_{\eta_1} * D_x \rho_{\eta_1} * D_{H^{-1}} \Phi_2 (t,m_{\eta_1} * \rho_{\eta_1},\cdot),m_{\eta_1} \Big) m_{\eta_1} dx 
    \\
    &\quad \to \int_{\T^d} H \Big(x, D_x \frac{\delta \Phi_1}{\delta m}(t,m,x) +  D_x D_{H^{-1}} \Phi_2 (t,m,\cdot),m \Big) m dx .
\end{align*}
Handling the rest of terms in a similar way  we conclude that 
\begin{align*} 
   &-\partial_t \Phi_1(\bar t, \bar m) - \partial_t \Phi_2(\ov t, \ov m) + \int_{\T^d} \Big( D_x \frac{\delta \Phi_1}{\delta m}(t, \ov m,x) + D_x D_{H^{-1}} \Phi_2(t,\bar m,x) \Big)\cdot D\bar m(x)dx \notag \\
& -\delta \int_{\T^d} |D\bar m|^2 dx  +\int_{\T^d} H\bigg(x, \Big( D_x\frac{ \delta \Phi_1}{\delta m}(t, \ov m,x) + D_x D_{H^{-1}} \Phi_2(t,\bar m,x) \Big),\bar m\bigg)d \bar m(x)  \geq -C\delta \|\ov{m}\|_2^2,
\end{align*}
which completes the proof.

\end{proof}

\subsection*{The proof of Theorem~\ref{thm.comparison}} \label{subsec.comparisonproof}
With Proposition \ref{prop.subsol.equiv}  in hand, we turn to the proof of the comparison principle. For this, we will make use of the penalization $\Phi:H^{-1}\to \R$ defined by 
\be\label{defPhiPsi}
\Phi(\mu) = \sup_{f\in H^1, \ f\leq 0} \Big\{ \langle \mu, f\rangle_{H^{-1},H^1} -\frac12 \|f\|_{H^1}^2 \Big\}. 
\ee
 It is immediate that,  for each $\mu\in H^{-1}$, there exists a unique maximizer $\hat{f}_{\mu}$ for $\Phi(\mu)$ which is characterized by the obstacle problem
\begin{align*}
\langle \mu, h-\hat{f}_{\mu} \rangle_{H^{-1},H^1} - \langle \hat f_\mu,h -\hat{f}_{\mu} \rangle_{H^1} \leq 0 \ \ \text{for all} \ \  h\in H^1 \ \text{such that} \  h\leq 0.
\end{align*}
Taking $\nu=0$ and $\nu= 2\hat{f}_{\mu}$ in the obstacle problem gives 
\be\label{formPhiPsi1}
\Phi(\mu)= \frac12\|\hat f_\mu\|_{H^1}^2= \frac12 \langle \mu, \hat f_\mu \rangle_{H^{-1},H^1}.
\ee
The obstacle problem is therefore equivalent to the inequality
\be\label{slkdfklkeknjzerPsi}
\langle \mu, h \rangle_{H^{-1},H^1} - \langle \hat f_\mu,h \rangle_{H^1} \leq 0 \ \ \text{for all} \ \  h\in H^1 \ \text{such that} \  h\leq 0,
\ee
coupled with equality \eqref{formPhiPsi1}.
\vs
In the next lemma we list and prove two key properties of the map $\Phi$.

\begin{lemma}\label{ath300} The map  $\Phi$ is of class $C^{1,1}$ and non-increasing and  
$
D_{H^{-1}} \Phi(\mu) = \hat f_\mu \leq 0. 
$
Moreover, if  $\mu\in L^2$ and $\Phi(\mu)=0$, then $\mu\geq0$. Finally, for all $\mu \in H^1$, if $\mu_-=\min\{\mu, 0\}$, then 
\be\label{slkdfklkeknjzerTERPsi1}
\int_{\T^d} D\hat f_\mu\cdot D\mu \geq-2\Phi(\mu)  +\|\mu_-\|^2_2. 
\ee
\end{lemma}

\begin{proof}
We first claim that $\mu \to \hat f_\mu$ is Lipschitz continuous from $H^{-1}$ to $H^1$. Indeed, fix  $\mu^1, \mu^2\in H^{-1}$. Then  $f^1= \hat f_{\mu^1}$ and $f^2= \hat f_{\mu^2}$  satisfy \eqref{slkdfklkeknjzerPsi} and \eqref{formPhiPsi1} and, thus,
\begin{align*}
\|f^1-f^2\|_{H^1}^2 & = \|f^1\|_{H^1}^2+ \|f^2\|_{H^1}^2- 2 \langle f^1,f^2 \rangle_{H^1} \\
& \leq \langle \mu^1,  f^1 \rangle_{H^{-1},H^1}+\langle \mu^2,  f^2\rangle_{H^{-1},H^1}-\langle \mu^1, f^2 \rangle_{H^{-1},H^1}-\langle \mu^2, f^1 \rangle_{H^{-1},H^1} \\
& = \langle \mu^1-\mu^2, f^1-f^2 \rangle_{H^{-1},H^1} \leq \|\mu^1-\mu^2\|_{H^{-1}}\|f^1-f^2\|_{H^1}. 
\end{align*}
It follows that $\|f^1-f^2\|_{H^1}\leq \|\mu^1-\mu^2\|_{H^{-1}}$.
\vs
Next, we claim that $\Phi$ is differentiable with $D_{H^{-1}} \Phi(\mu) = \hat{f}_{\mu}$, so that in particular $\Phi$ is $C^{1,1}$. We start by noting that  it is easy to check from the definition of $\Phi$ that, for all $\mu,\nu \in H^{-1}$,
\begin{align*}
   \Phi(\mu) \geq \langle \mu - \nu, \hat{f}_{\nu} \rangle_{H^{-1},H^1} + \Phi(\nu).
\end{align*}
In particular, reversing the role of $\nu$ and $\mu$, 
\begin{align*}
    \Phi(\mu) - \Phi(\nu) &\leq \langle \mu - \nu, \hat{f}_{\mu} \rangle_{H^{-1},H^1} \leq \langle \mu - \nu, \hat{f}_{\nu} \rangle_{H^{-1},H^1} + \langle \mu - \nu, \hat{f}_{\mu} - \hat{f}_{\nu} \rangle_{H^{-1},H^1}
    \\
    &\leq \langle \mu - \nu, \hat{f}_{\nu} \rangle_{H^{-1},H^1} + \| \mu - \nu\|_{H^{-1}}^2.
\end{align*}
That $D_{H^{-1}} \Phi(\mu) = \hat{f}_{\mu}$ follows since combining the previous two inequalities gives
\begin{align*}
    \Big| \Phi(\mu) - \Phi(\nu) - \langle \mu - \nu, \hat{f}_{\nu} \rangle_{H^{-1},H^1} \Big| \leq \|\mu - \nu\|_{H^{-1}}^2.
\end{align*}

\vs
Next, suppose that $\Phi(\mu)=0$. Then, for any $f\in H^1$ such that $f\leq 0$, we have 
 $$
 \int_{\T^d} \mu f \leq \frac12 \|f\|_{H^1}^2. 
 $$
 Replacing $f$ by $\ep f$ with $\ep>0$, dividing by $\ep$ and letting $\ep\to 0$, we find that, for every $f\in H^1$ such that $f\leq 0$,
 $$
  \int_{\T^d} \mu f \leq 0. 
$$
Since, by approximation, this inequality holds for any $f\in L^2$ such that $f\leq 0$, it follows  that $\mu\geq0$. 
\vs 
To prove \eqref{slkdfklkeknjzerTERPsi1}, we introduce the approximation
$$
\max_{f\in H^1} \Big[\int_{\T^d} f\mu - \frac12 \|f\|_{H^1}^2 -\frac{1}{2\ep} \int_{\T^d}  (f_+)^2\Big].
$$
Let $f^\ep$ be the unique maximizer of the above problem. It is immediate that the $f^\ep$'s converge, as $\ep \to 0$ and  weakly in $H^1$,  to  $\hat f_\mu$ and 
\be\label{kjsdnfkfnzkrPsi1}
f^\ep - \Delta f^\ep + \frac{1}{\ep} f^\ep_+= \mu.
\ee
Multiplying the last equation by  $\frac{1}{\ep}f^\ep_+$ and integrating over $\T^d$ gives 
$$
\int_{\T^d} \big[ \frac{1}{\ep}(f^\ep_+)^2 + \frac{1}{\ep}|Df^\ep_+|^2 + (\frac{1}{\ep}f^\ep_+)^2\big] dx = \int_{\T^d}\mu \frac{1}{\ep}f^\ep_+ dx \leq \int_{\T^d} \mu_+\frac{1}{\ep}f^\ep_+ dx,
$$
and thus, 
\begin{align} 
\label{fepsbound} \|\frac{1}{\ep}f^\ep_+\|_2\leq \|\mu_+\|_2.
\end{align}
Let $\hat \nu_\mu$ be a weak limit in $L^2$ of $(\frac{1}{\ep}f^\ep_+)$. Then $\hat \nu_\mu\in L^2$, $\hat \nu_\mu \geq 0$ and $\|\hat \nu_\mu\|_2\leq \|\mu_+\|_{2}$. Moreover, we can pass to the $\ep\to $ limit in \eqref{kjsdnfkfnzkrPsi1} to obtain 
$$
\hat f_\mu - \Delta \hat f_\mu + \hat \nu_\mu = \mu.
$$
Multiplying this equality by $\mu$ and  integrating gives, after using \eqref{formPhiPsi1} and the positivity of $\hat \nu_\mu$ for the first inequality and Cauchy-Schwarz for the second one and the bound $\|\hat \nu_\mu\|_2\leq \|\mu_+\|_{2}$ for the last one, 
\begin{align*}
\int_{\T^d} D\hat f_\mu\cdot D\mu &  = \int_{\T^d} -\hat f_\mu \mu - \hat \nu_\mu  \mu +\mu^2 \geq -2\Phi(\mu) -\int_{\T^d} \hat \nu_\mu \mu_+ +\|\mu\|_2^2\\ 
& \geq -2\Phi(\mu) -\| \hat \nu_\mu\|_2\| \mu_+\|_2 +\|\mu\|_2^2 \geq  -2\Phi(\mu)  +\|\mu_-\|^2,
\end{align*}

The proof is now complete. 

\end{proof}

 We continue with  the proof of the comparison. 

\begin{proof}[The proof of Theorem~\ref{thm.comparison}] We argue by contradiction assuming that $\max \big(V^- - V^+)>0$.
Since $V^-$ and $V^+$ are continuous,  we can choose $\delta,\theta>0$ small enough such that
$$
M_{\delta, \theta}:= \max_{(t,m)\in [0,T]\times \sub } V^-(t,m)-V^+(t,m)  -  2\delta \|m \|_2^2 -\theta(T-t) - \theta m(\T^d) >0. 
$$
In addition, as $V^- \leq V^+$ in $\{T\}\times \T^d$, any maximum point $(t_\delta,m_\delta)$ for $M_{\delta, \theta}$ satisfies $t_\delta<T$. We do not include the dependence of the maximum point on $\theta$, because $\theta$ will be fixed throughout the argument.
\vs
Fix now $\ep>0$ and let  $( t_{\delta,\ep}, m_{\delta,\ep},  s_{\delta,\ep},  n_{\delta,\ep})\in ([0,T]\times \sub)^2$ be a maximum point of 
\begin{equation*}
\begin{split}
(t,m,s,n)& \ \to \ V^-(t,m)-V^+(s,n) -\frac{1}{\ep}\Phi(m-n) - \int_{\T^d} \Psi(x,m) (m-n)(dx)\\
&\qquad  -\frac{1}{2\ep}(t-s)^2 -\delta( \|m\|^2_2+ \|n\|^2_2) -\theta(T-t) - \theta n(\T^d),
\end{split}
\end{equation*}
where $\Phi$ is given by \eqref{defPhiPsi}.
Note that, since  $\Phi(0)=0$,  
\begin{equation}\label{izaukehbsrdfngBISPsi}
\begin{split}
& V^-(t_{\delta,\ep},m_{\delta,\ep})- V^+(s_{\delta,\ep},n_{\delta,\ep}) - \int_{\T^d} \Psi(x,m_{\delta,\ep}) (m_{\delta,\ep}-n_{\delta,\ep})(dx) -\frac{1}{\ep}\Phi(m_{\delta,\ep}-n_{\delta,\ep}) \\
& \qquad \qquad -\frac{1}{2\ep}(t_{\delta,\ep}-s_{\delta,\ep})^2 -\delta( \|m_{\delta,\ep}\|^2_2+ \|n_{\delta,\ep}\|^2_2) -\theta(T-t_{\delta,\ep}) - \theta n_{\delta,\ep}(\T^d)
\geq M_{\delta, \theta}.
\end{split}
\ee
In addition, since  $U$, $V$ and $\Psi$ are bounded and $\Phi\geq0$, we have 
\begin{align} \label{penalizations.est}
\Phi( m_{\delta,\ep}- n_{\delta,\ep})\leq C\ep, \qquad |t_{\delta,\ep}-s_{\delta,\ep}|\leq C\ep^{1/2}, \qquad \| m_{\delta,\ep}\|_2^2+ \|n_{\delta,\ep}\|_2^2\leq C\delta^{-1}. 
\end{align}
We claim that any limit point $(t_\delta,m_\delta,s_\delta,n_\delta)$, as $\ep\to 0$ and in the weak topology in $L^2$, of $( t_{\delta,\ep}, m_{\delta,\ep},  s_{\delta,\ep},  n_{\delta,\ep})$ must satisfy $m_\delta=n_\delta$ and $s_\delta=t_\delta<T$, where $(t_\delta, m_\delta)$ is a maximum point in the expression which defines $M_{\delta, \theta}$. Note that such limit points exist in view of \eqref{penalizations.est}. 
\vs
The equality $s_\delta=t_\delta$ is clear. Since  $\Phi(m_\delta-n_\delta)=0$ and $m_\delta-n_\delta\in L^2$, it follows from Lemma~\ref{ath300}, that $m_\delta-n_\delta\geq 0$. 
\vs
If 
$m_\delta\neq n_\delta$, then passing to the limsup in \eqref{izaukehbsrdfngBISPsi}, we obtain 
$$
V^-(t_\delta,m_\delta)-V^+(t_\delta,n_\delta)  - \int_{\T^d} \Psi(x,m_\delta) (m_\delta-n_\delta) -\delta ( \|m_\delta\|^2_2 +\|n_\delta \|^2_2)-\theta(T-t_\delta) - \theta n_{\delta}(\T^d) \geq M_{\delta, \theta}.
$$
Since  $V^-(t,m)$ is $\Psi-$non-increasing and $m_\delta\geq n_\delta$, we find 
$$
 V^-(t_\delta,m_\delta)\leq V^-(t_\delta,n_\delta)  + \int_{\T^d} \Psi(x,m_\delta) (m_\delta-n_\delta)(dx) ,
$$
while, as $m_\delta\geq n_\delta\geq 0$ with $m_\delta\neq n_\delta$, $\|m_\delta\|_2^2>\|n_\delta\|_2^2$. Thus,
$$
V^-(t_\delta,n_\delta)-V^+(t_\delta,n_\delta)  -2 \delta \|n_\delta \|^2_2 -\theta(T-t_\delta) - \theta n_{\delta}(\T^d) > M_{\delta, \theta}, 
$$
a contradiction to the definition of $M_{\delta, \theta}$. 
\vs

Hence $m_\delta=n_\delta$ and $(t_\delta,m_\delta)$ is a maximum point in the definition of $M_{\delta, \theta}$ and, therefore,  $t_\delta<T$. 
\vs
We also note  for later use that 
\be\label{ath310}
\frac{1}{2\ep}\Phi(m_{\delta,\ep}-n_{\delta,\ep}) \to 0 \ \text{as } \ep\to0 \ \ \text{and} \ \ \delta \|m_\delta\|_2^2 \to 0   \ \text{as} \ \delta \to0.
\ee

From now on  we choose $\ep>0$ small enough such that $s_{\delta,\ep}, t_{\delta,\ep}<T$ and, to simplify the presentation, set $f_{\delta,\ep}= D_{H^{-1}} \Phi(m_{\delta, \eps} - n_{\delta, \eps}) = \hat f_{m_{\delta,\ep}-n_{\delta,\ep}}$; here we use the notation from Lemma \ref{ath300}.
\vs

Since  $V^-$ is a sub-solution to  \eqref{eq.HJstoppingPSI}, we have, in view of  Proposition~\ref{prop.subsol.equiv}, that $ m_{\delta,\ep}\in H^1$ and 
\begin{align*}
& \theta -\frac{ t_{\delta,\ep}- s_{\delta,\ep}}{\ep} +\int_{\T^d}  (\frac{Df_{\delta,\ep}}{\ep} + D\xi_{\delta,\ep})\cdot D m_{\delta,\ep}  +\delta \int_{\T^d} |D m_{\delta,\ep}|^2 \\
& \qquad\qquad  +\int_{\T^d} H(x, \frac{Df_{\delta,\ep}}{\ep} +D\xi_{\delta,\ep},m_{\delta,\ep}) m_{\delta,\ep}  \leq C\delta\| m_{\delta,\ep}\|_2^2,
\end{align*}
where 
$$
\xi_{\delta, \ep}(x) =  \int_{\T^d} \frac{\delta \Psi}{\delta m}(y,m_{\delta,\ep},x) (m_{\delta,\ep}-n_{\delta,\ep})(dy) +
\Psi(x,m_{\delta,\ep}) .
$$
\vs
Next, we use Proposition \ref{prop.subsol.equiv} and apply the super-solution test for $V^+$. 
For this, we need to know that the smooth test function
\begin{align} \label{decreasing}
    (s,n) \mapsto  \int_{\T^d} \Psi(x,m_{\eps,\delta}) dn - \theta  n(\T^d)
\end{align}
is locally $\Psi$-decreasing at $n_{\eps,\delta}$ for all $\eps$ small enough. To this end, note that for any $n \in \sub$ and $n' \leq n$, we have
\begin{align} \label{comp.nonincreasing}
    \int_{\T^d} & \Psi(x, m_{\delta, \eps}) n(dx) - \theta n(\T^d) =  \int_{\T^d} \Psi( x, m_{\delta, \eps}) n'(dx) +  \int_{\T^d} \Psi(x, m_{\delta, \eps}) (n - n')(dx) - \theta n(\T^d)
    \nonumber \\
    &= \int_{\T^d} \Psi(x, m_{\delta, \eps}) n'(dx) +  \int_{\T^d} \Big( \Psi(x, m_{\delta, \eps}) - \Psi(x, n) \Big)  (n - n')(dx)
  \nonumber   \\
    &\qquad + \int_{\T^d}\Psi( x, n)   (n - n')(dx) - \theta n'(\T^d) - \theta (n - n')(\T^d).
 \end{align}
  Now recall that $\bd(m_{\delta,\ep}, n_{\delta,\ep})\to 0$ as $\ep\to 0$. Thus, for $\eps$ small enough, any $n$ close enough to $n_{\delta, \eps}$, we have 
  \begin{align*}
     \int_{\T^d} \Big( \Psi(x, m_{\delta, \eps}) - \Psi(x, n) \Big)  (n - n') < \frac{\theta}{2} (n-n')(\T^d).
  \end{align*}
Plugging this into \eqref{comp.nonincreasing}, we see that, for $\eps$ small enough, the map in \eqref{decreasing} is indeed locally $\Psi$-decreasing at $(s_{\delta, \eps}, n_{\delta, \eps})$. 
\vs
Since  the map $n \mapsto - \Phi(m-n)$ is $C^{1,1}$ in $H^{-1}$ and non-increasing, and $V^+$ is a super-solution, we use Proposition \ref{prop.subsol.equiv} to deduce that $ n_{\delta,\ep} \in H^1$ and 
\begin{align*}
& -\frac{ t_{\delta,\ep}- s_{\delta,\ep}}{\ep}+\int_{\T^d}  (\frac{D f_{\delta,\ep} }{\ep}+ D\zeta_{\delta, \ep}) \cdot D n_{\delta,\ep}  -\delta \int_{\T^d} |D n_{\delta,\ep}|^2  \\
& \qquad \qquad +\int_{\T^d} H(x, \frac{D f_{\delta,\ep} }{\ep}+D\zeta_{\delta, \ep},n_{\delta,\ep})  n_{\delta,\ep}   \geq 
-C\delta \| n_{\delta,\ep}\|_2^2,
\end{align*}
where 
$$
\zeta_{\delta, \ep}(x)=  \Psi(x,m_{\delta,\ep}).
$$
Thus
\begin{align*}
& \theta+\int_{\T^d} \frac{ D f_{\delta,\ep}}{\ep} \cdot D( m_{\delta,\ep}- n_{\delta,\ep})-\int_{\T^d} \Delta \xi_{\delta, \ep}m_{\delta, \ep} 
+\int_{\T^d} \Delta \zeta_{\delta, \ep}n_{\delta, \ep}
 +\delta \int_{\T^d}( |D m_{\delta,\ep}|^2+ |D n_{\delta,\ep}|^2)  \\
& \qquad +\int_{\T^d} H(x, \frac{Df_{\delta,\ep}}{\ep}+D\xi_{\delta, \ep},m_{\delta,\ep}) m_{\delta,\ep}-\int_{\T^d} H(x, \frac{Df_{\delta,\ep}}{\ep}+D\zeta_{\delta, \ep},n_{\delta,\ep})n_{\delta,\ep}  \\
& \qquad  \leq C\delta \big( \|m_{\delta,\ep}\|_2^2+\|n_{\delta,\ep}\|_2^2). 
\end{align*}
Using \eqref{slkdfklkeknjzerTERPsi1} we have 
\begin{align*}
&\int_{\T^d} \frac{ D f_{\delta,\ep}}{\ep} \cdot D( m_{\delta,\ep}- n_{\delta,\ep}) \geq -\frac{2}{\ep}\Phi(m_{\delta,\ep}-n_{\delta,\ep}) + \frac{1}{\ep}\|(m_{\delta,\ep}-n_{\delta,\ep})_-\|^2_2, 
\end{align*}
while, by the regularity of $\Psi$, we get 
\begin{align*}
 -\int_{\T^d} \Delta \xi_{\delta, \ep}m_{\delta, \ep}  +\int_{\T^d} \Delta \zeta_{\delta, \ep}n_{\delta, \ep} & 
\geq -\int_{\T^d} \Delta \xi_{\delta, \ep}(m_{\delta, \ep} -n_{\delta, \ep} ) +\int_{\T^d} \Delta (\zeta_{\delta, \ep}-\xi_{\delta, \ep})n_{\delta, \ep} \\ 
& \geq -C \bd(m_{\delta, \ep},n_{\delta, \ep}) \; \geq \; -C\int_{\T^d} |m_{\delta, \ep}-n_{\delta, \ep}|. 
\end{align*}
Finally, since  $H$ is bounded from below and Lipschitz,  we find 
\begin{align*}
& \int_{\T^d} H\Big(x, \frac{Df_{\delta,\ep}}{\ep}+D\xi_{\delta, \ep}, m_{\delta,\ep} \Big) m_{\delta,\ep}-\int_{\T^d} H\Big(x, \frac{Df_{\delta,\ep}}{\ep}+D\zeta_{\delta, \ep}, n_{\delta,\ep}\Big)n_{\delta,\ep}\\ 
& \qquad \geq  \int_{\T^d} H\Big(x, \frac{Df_{\delta,\ep}}{\ep}+D\xi_{\delta, \ep},m_{\delta,\ep} \Big) (m_{\delta,\ep}-n_{\delta,\ep})_+-
 \int_{\T^d} H\Big(x, \frac{Df_{\delta,\ep}}{\ep}+D\xi_{\delta, \ep}, m_{\delta,\ep} \Big) (m_{\delta,\ep}-n_{\delta,\ep})_- \\
 & \qquad\qquad\qquad -C \bd( m_{\delta,\ep},n_{\delta,\ep}) - C\| D\xi_{\delta, \ep}-D\zeta_{\delta, \ep}\|_\infty \\
 &\qquad  \geq -C\int_{\T^d} (m_{\delta,\ep}- n_{\delta,\ep})_+ -C\int_{\T^d} (m_{\delta,\ep}- n_{\delta,\ep})_- (1+ \frac{|Df_{\delta,\ep}|}{\ep}+|D\xi_{\delta, \ep}|) - C\bd(m_{\delta, \ep},n_{\delta, \ep})  \\ 
&\qquad  \geq -C\int_{\T^d} |m_{\delta,\ep}- n_{\delta,\ep}| -\frac{1}{2\ep} \|(m_{\delta,\ep}- n_{\delta,\ep})_-\|_2^2 -\frac{C}{\ep} \|Df_{\delta,\ep}\|^2_2.  
\end{align*} 
Recalling \eqref{formPhiPsi1}, we can collect the four strings of inequalities above to get 
\begin{equation}\label{sldkjfeljhbnePsi}
\begin{split}
&  \theta  -\frac{C}{\ep} \Phi(m_{\delta,\ep}-n_{\delta,\ep}) -C\int_{\T^d} |m_{\delta,\ep}- n_{\delta,\ep}| +\delta \int_{\T^d}( |D m_{\delta,\ep}|^2+ |D n_{\delta,\ep}|^2) \\
& \qquad \qquad  \leq C\delta ( \|m_{\delta,\ep}\|_2^2+\|n_{\delta,\ep}\|_2^2). 
\end{split}
\ee
\vs
Recalling that \eqref{ath310}  and  \eqref{sldkjfeljhbnePsi} imply 
that the $m_{\delta,\ep}$'s and $n_{\delta,\ep}$'s are bounded in $H^1$ and 
that  the $m_{\delta,\ep}$'s and $n_{\delta,\ep}$'s  have the same weak limit, as $\ep\to 0$,  in $L^2$,  
yields that the  $m_{\delta,\ep}$'s and $n_{\delta,\ep}$'s  converge strongly, again as $\ep\to 0$,  in $L^2$ to some $ m_\delta\in H^1$. 
\vs
We also proved that the $t_{\delta,\ep}$'s and $s_{\delta,\ep}$'s converge to some $ t_\delta$ and that $( t_\delta,  m_\delta)$ is a maximum point in the definition of $M_{\delta, \theta}$. 
\vs
Passing to the limit in \eqref{sldkjfeljhbnePsi} we get 
$$
\theta \leq C\delta\| m_\delta\|_2^2,
$$
and find a contradiction by choosing $\delta$ small enough since  $\delta\| m_\delta\|_2^2$ tends to $0$. 

\end{proof}

\subsection*{The value function is a viscosity solution}

We now show that the value function $U$ is a viscosity solution.

\begin{proposition} \label{prop.valsub} Assume \eqref{summarize}. Then, the value function $U$ is a viscosity solution of \eqref{eq.HJstoppingPSI}. 
\end{proposition}

\begin{proof} We already know that $U$ is continuous and $\Psi-$non-increasing. The argument for the subsolution property is standard (see for instance the companion paper \cite{CJSabsorption}), as we can choose flows of measures  satisfying, for $\mu=0$, the equality 
$$
\partial_t m-\Delta m +{\rm div}(m\alpha) = 0 \ \text{in} \ (\ov t, T) \times \T^d  \ \ \text{and} \ \ m_{t_0}=m_0.
$$

We now show that $U$ is a super-solution. Let  $(\bar t, \bar m)$ be a minimum of $(t,m)\to U( t,m)-\Phi( t,m)$, where $\Phi$ is a smooth test function which is locally $\Psi$-decreasing at $(\ov t, \ov m)$. Using Corollary \ref{cor.regU}, we can find $(m,  \alpha,\mu)$ which is optimal for $U(\ov t, \ov m)$ and such that $\alpha$ is bounded.
\vs
For $h>0$, the dynamic programming principle says that 
\be\label{oijnfdlkjnPSI}
U(\bar t,\bar m) = \int_{\bar t}^{\bar t+h}  \int_{\T^d} L(x,\alpha,m_{t}) dm_t(x)  dt + \int_{[\ov t, \ov t + h] \times \T^d } \Psi(x,m_{t-}) d\mu(t,x) + U(\bar t+h, m_{\bar t + h}).
\ee
Sending $h \to 0$ and using Lemma \ref{lem.cdlag} shows that
\begin{align*}
    U(\ov{t},\ov{m}) - U(\ov{t}, m_{\ov t}) = \int_{\T^d} \Psi(x,\ov{m}) \mu(\{\ov t\},  dx) = \int_{\T^d} \Psi(x,\ov{m}) d(\ov{m} - m_{\ov t}).
\end{align*}
We also  have, for some $\delta > 0$. that 
\begin{align*}
    U(\ov t, \ov m) - U(\ov t, m_{\ov t}) \leq \Phi(\ov t, \ov m) - \Phi(\ov t, m_{\ov t}) \leq \int_{\R^d} \Psi(x,\ov{m}) d(\ov{m} - m_{\ov t}) - \delta (\ov{m} - m_{\ov t})(\T^d), 
\end{align*}

Subtracting the two inequalities, we find that $(\ov{m} - m_{\ov t})(\T^d) = 0$, and, thus,  $\ov{m} = m_{\ov t}$. 
\vs
Next, using the fact that $U - \Phi$ is minimized at $(\ov t, \ov m)$, we find that  
\begin{equation} \label{phbar.ineq}
\begin{split}
    \Phi(\ov t, \ov m) &\geq \int_{\bar t}^{\bar t+h}  \int_{\T^d} L\big(x,\alpha(t,x),m_t \big) dm_t(x)  dt \\
    &\qquad +\int_{[\ov t, \ov t + h] \times \T^d} \Psi(x,m_{t-})d\mu(t,x) + \Phi(\bar t+h, m_{\bar t + h}).
\end{split}
\ee
\vs
Let $\tilde{m}$ be  the solution of 
\begin{align*}
    \partial_t \wt{m} - \Delta \wt{m} + \text{div}(\wt m \alpha) = 0\ \text{in} \ (\ov t, T) \times \T^d \ \ \text{and}  \ \  \wt{m}_{t_0} = \ov m. 
\end{align*}
We note that, since $\wt{m}_t \geq m_t$, and, in view of the continuity of  $m$ at $\ov t$, $\big(\wt{m}_t - m_t \big)(\T^d) \to 0$ as $t\to \ov t$,  
it follows that, if   $\bd_{\text{TV}}$ is the total variation metric, 
\begin{align}\label{takis4000}
    \lim_{t \to \ov t} \; \bd_{\text{TV}}(\wt{m}_t,m_t)=0. 
\end{align}

Next, we observe that,  in view of the fact that $\Phi$ is locally $\Psi$-decreasing at $(\ov{t}, \ov{m})$, we have, for $h$ small enough and some $\delta > 0$, that 
\begin{align*}
    \Phi(\ov t + h, m_{\ov t + h}) &\geq \Phi(\ov t + h, \wt{m}_{\ov t + h}) - \int_{\T^d} \Psi(x, \wt{m}_{\ov t + h} ) d(\wt{m}_{t+ h} - m_{t+h})(x) 
+ \delta (\wt{m}_{\ov t + h} - m_{\ov t + h})(\T^d) 
    \\
    &= \Phi(\ov t + h, \wt{m}_{\ov t + h}) - \int_{\T^d} \Psi(x, \wt{m}_{\ov t + h} ) d(\wt{m}_{t+ h} - m_{t+h})(x) + \delta \mu([\ov t, \ov t + h] \times \T^d).
\end{align*}
Combining the last inequality  with \eqref{phbar.ineq} yields 
\begin{align*}
   \Phi(\bar t,\bar m) &\geq \int_{\bar t}^{\bar t+h}  \int_{\T^d} L\big(x,\alpha(t,x),m_t \big) dm_t(x) dt +\int_{[\ov t,\bar t+h] \times \T^d} \Psi(x,m_{t-})\mu
   \\
   &\qquad - \int_{\T^d} \Psi(x, \wt{m}_{\ov t+h})d(\wt{m}_{\ov t + h} - m_{\ov t + h})(x)
  + \Phi(\bar t+h, \wt{m}_{\ov t + h})  + \delta \mu([\ov t, \ov t + h] \times \T^d).
\end{align*}
Using the equations satisfied by $m$ and $\wt{m}$ gives 
\begin{align*}
    \int_{\T^d} \Psi(x, \wt{m}_{\ov t +h}) & d(\wt{m}_{\ov t + h} - m_{\ov t + h}) = \int_{\ov t}^{\ov t + h} \int_{\T^d} \big( \Delta_x \Psi( \wt{m}_{\ov t + h}, x )  + D_x \Psi(\wt{m}_{\ov t + h},x)  \cdot \alpha(t,x)  \big)  d(\wt{m}_t - m_t)(x) dt
    \\
    &\quad +   \int_{[\ov t,\ov t + h] \times \T^d} \Psi( x, \wt{m}_{\ov t + h}) d\mu 
    \leq o(h) +  \int_{[\ov t, \ov t + h] \times \T^d} \Psi( x, m_{\ov t + h}) d\mu.
\end{align*}
Hence
\begin{align*}
  &\int_{[\ov t, \ov t + h] \times \T^d} \Psi(x, m_{t-})\mu(dt,dx)  - \int_{\T^d} \Psi(x, \wt{m}_{\ov t+h})d(\wt{m}_{\ov t + h} - m_{\ov t + h})  + \delta \mu([\ov t, \ov t + h] \times \T^d)
  \\
  &\quad \geq - o(h) + \int_{[\ov t, \ov t + h] \times \T^d} \Big( \delta +  \Psi(x,m_{t-}) - \Psi(x, m_{\ov t + h}) \Big) d\mu 
  \\
  &\quad \geq - o(h),
\end{align*}
where we used the fact that the second term is positive for small $h$ by continuity. 
\vs 
In particular, we find that 
\begin{align*}
   \Phi(\bar t,\bar m) &\geq \int_{\bar t}^{\bar t+h}  \int_{\T^d} L(x,\alpha,m) m dt  + \Phi(\bar t+h, \wt{m}_{\bar t+h}) - o(h),
\end{align*}
and, in view of \eqref{takis4000}, 
we conclude that
\begin{align*}
  \Big|  \int_{\ov t}^{\ov t + h} \int_{\T^d} L(x, \alpha, m) m dt -  \int_{\ov t}^{\ov t + h} \int_{\T^d} L(x, \alpha, \wt{m}) \wt{m} dt \Big| = o(h), 
\end{align*}
and, thus, 
\begin{align*}
   \Phi(\bar t,\bar m) &\geq \int_{\bar t}^{\bar t+h}  \int_{\T^d} L(x,\alpha,\wt{m}) \wt{m} dt  + \Phi(\bar t+h, \wt{m}_{\bar t + h}) - o(h). 
\end{align*}
Expanding $\Phi$ along the regular flow $\tilde m$ yields
\begin{align*}
    o(h) &\geq \int_{\ov{t}}^{\ov t + h} \bigg( \partial_t \Phi(t, \wt{m}_{t}) + \int_{\T^d} \Delta_x \frac{\delta \Phi}{\delta m} (t, \wt{m}_{t},x ) d\wt{m}_{t}(d) + \int_{\T^d} \Big( L(x, \alpha, \wt{m}) \\ &+ D_x \frac{\delta \Phi}{\delta m}(t, \wt{m}_{t},x ) \cdot \alpha \Big) d\wt{m}_t(x) \bigg) dt 
    \\
    &\geq \int_{\ov{t}}^{\ov t + h} \bigg( \partial_t \Phi(t, \wt{m}_{t}) + \int_{\T^d} \Delta_x \frac{\delta \Phi}{\delta m} (t, \wt{m}_{t},x )d \wt{m}_{t}(x) - \int_{\T^d} H\Big(x, D_x \frac{\delta \Phi}{\delta m} (t, d\wt{m}_{t}, )  , \wt{m}_t \Big)   \bigg)dt 
\end{align*}
Dividing by $h$ and sending $h \to 0$ gives the result. 

\end{proof}

\subsection*{The proof of Theorem~\ref{thm.uniqueness}}

We start with  a corollary of Theorem \ref{thm.comparison}.
\begin{corollary} \label{cor.truncation}
    Suppose that \eqref{phl2} and \eqref{phl3} hold, and that $H : \T^d \times \R^d \times \cP(\T^d) \to \R$ is just locally Lipschitz continuous. If $U$ and $V$ are respectively  a viscosity sub- and super-solution to  \eqref{eq.HJstoppingPSI}  
    and both $U$ and $V$ are Lipschitz continuous in $m$, uniformly in $t$, then $U \leq V$. 
\end{corollary}
\begin{proof}
    The idea is to replace $H$ by 
    \begin{align}  \label{def.Hr}
        H^R(x,p,m) = H\big(x, \pi^R(p), m \big)
    \end{align}
    where $\pi^R : \R^d \to \R^d$ is projection onto the ball of radius $R$.
    \vs
     It is straightforward to show that if $\phi : \sub \to \R$ is Lipschitz with respect to $\bd$, and, if $\phi - \psi$ has a maximum or a minimum at $m \in \sub$ for some smooth test function $\psi$, then 
     $$| D_x \frac{\delta \psi}{\delta m}(m,x) | \leq \text{Lip}(\phi; \bd) \ \text{ for  all $x$ in the support of $m$}.$$
      It follows  that, if $U$ and $V$ are Lipschitz in $m$, then they will be respectively  viscosity sub- and super-solution to \eqref{eq.HJstoppingPSI} with $H^R$ replacing $H$ for all $R$ large enough. 
      \vs
      Since $H^R$ is globally Lipschitz (and bounded), the claim follows applying Theorem \ref{thm.comparison}.
     
\end{proof}

\begin{proof}[Proof of Theorem \ref{thm.uniqueness}]
  This is an immediate consequence of Proposition \ref{prop.valsub} and Corollaries \ref{cor.regU} and \ref{cor.truncation}.
\end{proof}

\section{The mean field limit}  \label{sec.convergence}

We consider here the mean field limit behavior of \eqref{stopping.gen2}.

\begin{proof}[Proof of Theorem \ref{thm.convergence}] 
With $C$  the constant in  Theorem~\ref{thm.main.lip}, we define $\wt{V}^N : [0,T] \times \sub \to \R$ by 
$$
\wt{V}^N(t,m)= \sup_{K, \bx\in (\T^d)^K} V^{N,K}(t,\bx) - C\bd(m^{N,K}_{\bx}, m). 
$$
Then  Theorem~\ref{thm.main.lip} implies that $\tilde V^N(t,m^{N,K}_{\bx} ) =V^{N,K}(t,\bx)$  for all $K\in \{0, \dots, N\}$ and all  $(t,\bx) \in [0,T]\times (\T^d)^K$ and that 
\begin{align*}
    |\wt{V}^N(t,m) - \wt{V}^N(s,n)| \leq C \Big(\bd(m,n) + |t-s|^{1/2} \Big).
\end{align*}
Hence, the  $\tilde V^N$'s are  compact for the uniform convergence. 
\vs

Let $V$ be any cluster point of this sequence. For simplicity of notation, we  still denote by $\wt{V}^N$ the converging subsequence. We claim  that $V$ is a solution to  the HJ equation \eqref{eq.HJstoppingPSI}. 
\vs
That  $V(T,m) = G_{\Psi}(m)$ follows from Lemma \ref{lem.gnk.conv}. 

\vs

To show that  $V$ is $\Psi$-non-increasing, we fix $t_0\in [0,T]$ and  $m_0\leq n_0$, and aim to show  that
\begin{align} \label{V.nondecreasing}
    V(t_0,n_0) \leq V(t_0,m_0) + \int_{\T^d} \Psi(x, n_0)  (n_0 - m_0)(dx). 
\end{align}
By the continuity of $V$, we can assume that $\int_{\T^d} n_0<1$. 
Then  Lemma \ref{lem.approxm0} below states that, if  $\nu_0= n_0-m_0$,  there exist $K^N,M^N\in \{0, \dots, N\}$, $\bx_N \in (\T^d)^{K^N}$ and $\by_N\in (\T^d)^{M^N}$ such that the $m^{N,K^N}_{\bx_N}$'s and the  $m^{N,M^N}_{\by_N}$'s converge respectively  to $m_0$ and $\nu_0$. It follows that the  $m^{N, K^N+M^N}_{(\bx_N, \by_N)}= m^{N,K^N}_{\bx_N} +m^{N,M^N}_{\by_N}$'s  converge to $n_0$. 

\vs
From  the definition of the $(V^{N,K})$'s, we have 
$$
 V^{N, K^N+M^N}(t, m^{N, K^N+M^N}_{(\bx_N, \by_N)}) \leq V^{N,K^N}(t, m^{N,K^N}_{\bx_N}) + \frac{1}{N} \sum_{i = 1}^{M^N} \Psi\big(y_N^j, m^{N, K^N+M^N}_{(\bx_N, \by_N)}\big). 
$$
Letting $N\to \infty$, we obtain \eqref{V.nondecreasing}.
\\

We now check that $V$ is a sub-solution. Let $\Phi:(0,T)\times \sub \to \R$ be a smooth test function in the sense of Definition \ref{def.testfunction}, such that $V-\Phi$ has a strict maximum at $(t_0,m_0)$. Then there exist (up to a subsequence) $K^N$ and $(t^N, {\bx^N})$, a local maximum of $(t,\bx)\to V^{N,K^N}(t,\bx)-\Phi(t,m^{N,K^N}_{\bx})$ such that $t^N\to t_0$ and  $m^{N,K^N}_{\bx^N}\to m_0$. 
\vs
Set $w(t,\bx)=\Phi(t,m^{N,K^N}_{\bx})$. Then Lemma \ref{lem.C11} below implies that $w(t^N,\cdot)$ is $C^{1,1}$. Moreover, for each $\ep>0$, the map $\bx\to V^{N,K^N}(t^N,\bx)-w(t^N, \bx)-\ep|\bx -\bx^N|^2$ has a strict maximum at $\bx^N$. 
\vs

Thus (see \cite[Lemma A.3]{usersguide}), for any $\delta>0$ small, there is $\bp^N\in (\R^d)^{K^N}$ such that $|\bp^N|<\delta$, the map $\bx\to V^{N,K^N}(t^N,\bx)-w(t^N, \bx)-\ep|\bx -\bx^N|^2-\bp^N\cdot (\bx-\bx^N)$ has a maximum at some  $\by^{N,\delta}$ point of second order differentiability of $w$ and $\by^{N,\delta}\to \bx^N$ as $\delta \to 0$.

\vs

It also follows from  Lemma \ref{lem.C11} and by the equation satisfied by $V^{N,K}$ that, for some constant $C$ which is independent of $N$, 
\begin{align*}
 &   - \partial_t \Phi (t^N,m^{N,K^N}_{\bx^N})-\frac{1}{N} \sum_{i = 1}^{K^N} (\Delta_{y} \frac{\delta \Phi}{\delta m} (t^N, m^{N,K^N}_{\bx^N},x^{N,i}) + 2\ep d)     \\
&     
      + \frac{1}{N} \sum_{i = 1}^{K^N} H\Big( x^{N,i},D_x \frac{\delta \Phi}{\delta m}(t^N, m^{N,K}_{\by^{N,\delta}},\by^{N,\delta,i}))+2\ep N (\by^{N,\delta,i}-\bar x^{N,i})+N\bp^{N,i},x^{N,i}\big), m_{\bx^N}^{N,K^N} \Big)\; \leq \; \frac{C d}{N}. 
\end{align*}
\vs
Letting first $\delta \to 0$, then $\ep\to 0$ and, finally, $N\to \infty$ we obtain  
\begin{align*}
 &  - \partial_t \Phi (t_0,m_0)-\int_{\T^d} \Delta_{y} \frac{\delta \Phi}{\delta m} (t_0, m_0,x) m_0(dx) 
      + \int_{\T^d} H\Big( x,D_{m} \Phi (t_0, m_0,x\big), m_0\Big) m_0(dx) \leq 0. 
\end{align*}

Next we show that  $V$ is a super-solution. Let $\Phi:(0,T)\times \sub \to \R$ be a smooth test function, such that $V-\Phi$ has a minimum at $(t_0,m_0)$ and $\Phi$ is locally $\Psi$-decreasing at $(t_0,m_0)$. We can assume without loss of generality that this minimum is strict, since otherwise we can replace $\Phi$ by $\tilde \Phi:=\Phi+\ep((t-t_0)^2+ \|\cdot - m_0\|_{H^{-s}}^2)$ for $\ep\geq 0$ and  $s>d/2 + 2$, which is such that $V-\tilde \Phi$ has a strict minimum at $(t_0,m_0)$, and remains locally $\Psi$-decreasing at $(t_0,m_0)$ for $\eps$ small enough. 
\vs

Then there exist,  up to  subsequences still denoted by $N$,   $K^N\in\{0, \dots, N\}$ and $(t^N, {\bx^N})$ local minimum of $(t,\bx)\to V^{N,K^N}(t,\bx)-\Phi(t,m^{N,K^N}_{\bx})$ such that $t_N \to t_0$ and  $m^{N,K^N}_{\bx_N}\to m_0$. 
\vs
Arguing as in the proof of the sub-solution property, we find  
\begin{align}\label{zielrgfkg2}
 &  \max\Bigl\{  - \partial_t \Phi (t^N,m^{N,K^N}_{\bx^N})-\frac{1}{N} \sum_{i = 1}^{K^N} \Delta_{y} \frac{\delta \Phi}{\delta m} (t_N, m^{N,K^N}_{\bx_N},x_{N}^i)    +\frac{Cd}{N} \notag  \\
& \qquad\qquad    
      + \frac{1}{N} \sum_{i = 1}^K H\Big( x_N^i ,D_{m} \Phi (t_, m^{N,K^N}_{\bx_N},x_N^i \big),m_{\bx_N}^{N,K^N} \Big) \; ,\\
      & \qquad\qquad  \max_{S \subset [K^N]\backslash\{\emptyset\}} \Big( V^{N,K^N}(t_N,\bx_N)-V^{N,K^N-|S|}(t_N,\bx_N^{-S})  - \frac{1}{N} \sum_{i \in S} \Psi( x_N^i, m_{\bx_N}^{N,K_N}) \Big) \Bigr\} \geq  0 . \notag
\end{align}
Next we claim that, for $N$ large enough, 
\be\label{zielrgfkg}
\max_{S \subset [K^N]\backslash\{\emptyset\}} \Big( V^{N,K^N}(t_N,\bx_N)-V^{N,K^N-|S|}(t_N,\bx_N^{-S})  - \frac{1}{N} \sum_{i \in S} \Psi(x_N^i, m_{\bx_N}^{N,K_N}) \Big)  <0. 
\ee
Indeed, suppose that,  for some $N \in \N$ and $S \subset [K_N]\backslash\{\emptyset\}$, we have
$$
V^{N,K^N}(t_N,\bx_N)-V^{N,K^N- |S| }(t_N,\bx_N^{-S}) \geq \frac{1}{N} \sum_{ i \in S} \Psi(x_N^i, m^N_{\bx_N}) . 
$$
Using the optimality of  $(t_N, \bx_N)$, we have therefore 
$$
\Phi(t_N,m_{\bx_N}^N)-\Phi(t_N,m_{\bx_N^{-S}}^{N,K_N - |S|}) \geq \frac{1}{N} \sum_{ i \in S} \Psi( m_{\bx_N}^N, x_N^i).
$$
For large enough $N$, this clearly contradicts the fact that $\Phi$ is locally $\Psi$-decreasing at $(t_0,m_0)$, and, thus,  \eqref{zielrgfkg} holds. 
\vs

It follows that, in view of  \eqref{zielrgfkg2}, we have 
\begin{align*}
 &- \partial_t \Phi (t^N,m^{N,K^N}_{\bx^N})-\frac{1}{N} \sum_{i = 1}^{K^N} \Delta_{y} \frac{\delta \Phi}{\delta m} (t^N, m^{N,K^N}_{\bx^N},x^{N,i})   +\frac{Cd}{N}  \notag  \\
& \qquad    \qquad
      + \frac{1}{N} \sum_{i = 1}^K H\big( x^{N,i},D_{m} \Phi (t^N, m^{N,K^N}_{\bx^N},x^{N,i}\big), m_{\bx^N}^{N,K^N}) \geq 0.
      \end{align*}
Sending $N\to\infty$, we find the required inequality. 

\end{proof}

We are left to prove the technical lemmata used in the proof above.

\begin{lemma}\label{lem.approxm0} Given $m_0\in \sub $, there exist sequences $K^N\in \{0, \dots, N\}$ and $\bx^N \in (\T^d)^{K^N}$ such that the $(m^{N, K^N}_{\bx^N})$'s converge to $m_0$.
\end{lemma}

\begin{proof}  We assume that $m_0\neq 0$ since otherwise one simply chooses $K^N=0$ for any $N$. 
\vs
Set $\hat m_0=m_0/\int_{\T^d} m_0$ and fix $\ep>0$. It is then  known that there exists $K_0$ such that, for all $K\geq K_0$, there exists $\bx^K\in (\T^d)^K$ such that 
$$
\bd_1(m^K_{\bx^K}, \hat m_0)\leq \ep.
$$
\vs
We now choose $N_0$ large enough such that, if $N\geq N_0$, then there exists $K^N\in \{K_0, \dots, N\}$ such that $$|K^N/N- \int_{\T^d} m_0|\leq \ep. $$ 
\vs
It follows that 
\begin{align*}
\bd (m^{N, K^N}_{\bx^{K^N}}, m_0) &= \frac{K^N}{N} \bd (m^{K^N}_{\bx^{K^N}},  \frac{N}{K^N} m_0) 
 \leq \frac{K^N}{N}\bd_1 (m^{K^N}_{\bx^{K^N}}, \hat m_0) + \frac{K^N}{N}\bd( \frac{N}{K^N} m_0, \hat m_0) \\ 
& \leq \ep +  \frac{K^N}{N}  \left| \frac{N}{K^N}\int_{\T^d}m_0 -1\right| = \ep + \left| \int_{\T^d} m_0-\frac{K^N}{N}\right| \leq 2\ep. 
\end{align*}
This completes the proof.

\end{proof}

For the next lemma we recall that $C^{1,1}$ maps have a second order expansion at a.e. point. By a small abuse of notation we denote by $D^2_{x^ix^i}w(\bar{\bx})$ the second order term in the expansion.

\begin{lemma}\label{lem.C11} Assume that $\Phi:[0,T]\times \sub\to \R$ is smooth in the sense of Definition \ref{def.testfunction}. Then there exists $C$ such that, for all  $N\geq 1$, $K\in \{1, \dots, N\}$ and $\bar t\in [0,T]$, the map $w(\bx)= \Phi(\bar t, m^{N,K}_{\bx})$ is $C^{1,1}$ and for all points $\bar{\bx}\in (\T^d)^K$ of second order differentiability of $w$ we have 
\be\label{derivative1}
 D_{x^i} w(\bar{\bx}) = \frac{1}{N} D_x \frac{\delta \Phi}{\delta m}(\bar t, m^{N,K}_{\bar{\bx}}, \bar x^i) \\
\ee
and 
\be\label{derivative2}
 \Bigl| D^2_{x^ix^i}w(\bar{\bx}) - \frac{1}{N} D^2_x \frac{\delta \Phi}{\delta m}(\bar t, m^{N,K}_{\bar{\bx}}, \bar x^i,\bar x^i)\Bigr| \leq \frac{C}{N}. 
\ee
\end{lemma}


\begin{proof} The proof is an easy application of the definition of the derivative and the Lipschitz regularity of this derivative. To simplify the notation, we omit the dependence of $\Phi$ on the time variable, which plays no role here, since all the estimates are uniform with respect to this variable. 
\vs
The fact that $w$ is $C^1$ with a derivative given by \eqref{derivative1} is standard (see \cite{CDLL}). In addition, $w$ is $C^{1,1}$ because the map $D_x \frac{\delta \Phi}{\delta m}$ is globally Lipschitz. 
\vs
The definition of the derivative implies that, for $m_s= (1-s)m^{N,K}_{(\bar{\bx}^{-i},x^i)}+s m^{N,K}_{\bar{\bx}}$,
\begin{align*} 
 w((\bar{\bx}^{-i},x^i))-w(\bar{\bx}) &=  \Phi( m^{N,K}_{(\bar{\bx}^{-i},x^i)}) -\Phi( m^{N,K}_{\bar{\bx}}) = \int_0^1 \int_{\T^d} \frac{\delta \Phi}{\delta m} (m_s, y) d(m^{N,K}_{(\bar{\bx}^{-i},x^i)}-m^{N,K}_{\bar{\bx}})(y)ds\\
& = \frac{1}{N} \sum_{i=1}^K \int_0^1(\frac{\delta \Phi}{\delta m} (m_s, x^i)-\frac{\delta \Phi}{\delta m} (m_s, \bar x^i))ds
\end{align*}
That  $x\to D^2_{x}\frac{\delta \Phi}{\delta m} (m,x)$ is Lipschitz uniformly in $m$ yields that 
\begin{align*}
&\Bigl| \frac{\delta \Phi}{\delta m} (m_s, x^i)-\frac{\delta \Phi}{\delta m} (m_s, \bar x^i)
- D_x\frac{\delta \Phi}{\delta m} (m_s, \bar x^i)(x^i-\bar x^i) -\frac12 D^2_x\frac{\delta \Phi}{\delta m}(m^{N,K}_{\bar{\bx}}, \bar x^i)(x^i-\bar x^i)\cdot (x^i-\bar x^i)\Bigr| \\
&\qquad \qquad \leq C |x^i -\bar x^i|^3
\end{align*}

Since  $m\to D_{x}\frac{\delta \Phi}{\delta m} (m,x)$ and $m\to D^2_{x}\frac{\delta \Phi}{\delta m} (m,x)$ are Lipschitz uniformly in $x$, we also have 
\begin{align*}
& \Bigl| D_x\frac{\delta \Phi}{\delta m} (m_s, \bar x^i)\cdot (x^i-\bar x^i) -D_x\frac{\delta \Phi}{\delta m} (m^{N,K}_{\bar x}, \bar x^i)\cdot (x^i-\bar x^i)\Bigr|  \\
& + \Bigl|  D^2_x\frac{\delta \Phi}{\delta m}(m^{N,K}_{\bar{\bx}}, \bar x^i)(x^i-\bar x^i)\cdot (x^i-\bar x^i)-
 D^2_x\frac{\delta \Phi}{\delta m}(m^{N,K}_{\bar{\bx}}, \bar x^i)(x^i-\bar x^i)\cdot (x^i-\bar x^i)\Bigr| \\ 
& \qquad \qquad   \leq \frac{C}{N} |x^i-\bar x^i| \bd(m^{N,K}_{(\bar{\bx}^{-i},x^i)}, m^{N,K}_{\bar{\bx}}) \leq \frac{C}{N^2} |x^i-\bar x^i|^2. 
\end{align*}
\vs
Collecting the results above gives \eqref{derivative2}. 

\end{proof}

\appendix

\section{Mollifications and $\Psi$-monotone envelopes}

\label{sec.mollification}

\subsection*{$\Psi$-monotone envelopes}

We prove here some technical facts about the maps 
\begin{align*}
    G_{\Psi} : \sub \to \R \ \text{and} \ G_{\Psi}^{N,K} : (\T^d)^K \to \R \ \text{for  $N\in\N$ and $K \in \{1,...,N\}$}, 
\end{align*}
defined from $G : \sub \to \R$ and $\Psi : \T^d \times \sub \to \R$ by \eqref{def.monotoneenvelope} and \eqref{def.discretemonotoneenvelop}, respectively.

\begin{lemma} \label{lem.envolopeismonotone}
    The functions $G_{\Psi}$ is $\Psi$-non-increasing in the sense of \eqref{def.psinon-increasing}. Similarly, the functions $G_{\Psi}^{N,K}$ are $\Psi$-non-increasing, in the sense that, for any $N \in \N$, $K \in \{1,...,N\}$, $\bx \in (\T^d)^K$, and $S \subset [K]$, we have
    \begin{align*}
        G_{\Psi}^{N,K}(\bx) \leq G_{\Psi}^{N,K-|S|}(\bx^{-S}) + \frac{1}{N} \sum_{i \in S} \Psi(x^i, m_{\bx}^{N,K}). 
    \end{align*}
\end{lemma}

\begin{proof}
    Fix $m,m' \in \sub$ with $m' \leq m$ and choose  $m'' \leq m'$ such that
    \begin{align*}
        G_{\Psi}(m') = G(m'') + \int_{\T^d} \Psi(x,m') d(m' - m'').  
    \end{align*}
    Since the map  $m \mapsto \Psi(x,m)$ is non-increasing, it follows that 
    \begin{align*}
        G_{\Psi}(m) &\leq G(m'') + \int_{\T^d} \Psi(x,m) d(m - m'') 
        \\
        &\leq G(m'') + \int_{\T^d} \Psi(x,m') d(m' - m'') + \int_{\T^d} \Psi(x,m) d(m-m') 
        \\
        &= G_{\Psi}( m') + \int_{\T^d} \Psi(x,  m) d(m - m'). 
    \end{align*}
    
 Thus $G_{\Psi}$ is $\Psi$-non-increasing. The proof of the $\Psi$-non-increasing property of $G_{\Psi}^{N,K}$ is similar. 

\end{proof}

\begin{lemma} \label{lem.envelopeLip}
    If $\Psi$ and $G$ are  Lipschitz continuous, then so is $G_{\Psi}$.
    Moreover, $G^{N,K}_{\Psi}$ is Lipschitz continuous uniformly in $N$, in the sense that there exists a constant $C$ such that, for each $N \in \N$, $J,K \in \{1,...,N\}$, $\bx \in (\T^d)^K$ and $ \by \in (\T^d)^J$,
      \begin{align*}
        |G^{N,K}_{\Psi}(\bx) - G^{N,J}_{\Psi}(\by)| \leq C \bd(m_{\bx}^{N,K}, m_{\by}^{N,J}).
    \end{align*}
\end{lemma}

\begin{proof}
   First, we assume  that $m$ and $n$ have the same mass, that is,  $m(\T^d) = n(\T^d)$ and recall that,  if $\bd_1$ is the 1-Wasserstein distance between $m$ and $n$, then
   \begin{align*}
       \bd(m,n) \leq \bd_1(m,n) \leq \text{diam}(\T^d) \bd(m,n), 
   \end{align*}
   Let $\pi$ be an optimal coupling  between $m$, and $n$, that is,  $\pi \in \cP(\T^d \times \T^d)$ is such that 
   \begin{align*}
       m(A) = \pi(A \times \T^d), \quad n(A) = \pi (\T^d \times A) \ \text{and} \ \bd_1(m,n) = \int_{\T^d} \int_{\T^d} |x-y| d\pi(x,y), 
   \end{align*}
 and denote  by  $(\pi_x)_{x \in \T^d}$ its  disintegration  with respect to the first marginal, that is, 
   \begin{align*}
       d\pi(x,y) = dm(x) d\pi_x(y) \ \text{and} \  dn(y) = \int_{\T^d} d\pi_x(y) dm(x) . 
   \end{align*}
   Next, let $m' \in \sub$ be such that 
   \begin{align*}
        G_{\Psi}(m) = G(m') + \int_{\T^d} \Psi(x,m) d(m - m')(x),  
   \end{align*}
and define $n'$ by 
   \begin{align*}
        dn'(y) = \int_{\T^d} d\pi_x(y) dm'(x) . 
   \end{align*}
   Notice that $n' \leq n$, and, for any $1$-Lipschitz $f : \T^d \to \R$, 
   \begin{align*}
       \int_{\T^d} &f d(m' - n') = \int_{\T^d} f(x)  dm'(x) - \int_{\T^d} \int_{\T^d} f(y) d\pi_x(y) dm'(x) 
       \\
       &\leq \int_{\T^d} \int_{\T^d} |x-y| d\pi_x(y) dm'(x) 
       \leq \int_{\T^d} \int_{\T^d} |x-y| d\pi_x(y) dm(x)  = \bd_1(m,n). 
   \end{align*}
   
  It follows that  $$\bd(m',n') \leq \bd_1(m,n) \leq \text{diam}(\T^d) \bd(m,n),$$
  and, thus, 
   \begin{align*}
        &G_{\Psi}(n) \leq G(n') + \int_{\T^d} \Psi(x,n) d(n - n')(x) 
        \\
        &\quad \leq G(m') + \int_{\T^d} \Psi(x,m) d(n - n')(x) + C \bd(m,n)
    \leq G(m') + \int_{\T^d} \Psi(x,m) d(m-m') + C \bd(m,n)
        \\
        &\quad= G_{\Psi}(m) + C \bd(m,n). 
   \end{align*}
   The above  shows that,  for any $m,n$ with $m(\T^d) = n(\T^d)$, 
   \begin{align*}
       |G_{\Psi}(m) - G_{\Psi}(n)| \leq C \bd(m,n). 
   \end{align*}
   Now consider the case  $m \leq n$, which implies that  $\bd(m,n) = n(\T^d) - m(\T^d)$. In view of Lemma \ref{lem.envolopeismonotone}, we have 
   \begin{align*}
       G_{\Psi}(n) \leq G_{\Psi}(m) + \int_{\T^d} \Psi(x,n) d(n-m) \leq G_{\Psi}(m) + C (n-m)(\T^d) = G_{\Psi}(m) + C \bd(m,n). 
   \end{align*}
   Choose $n' \leq n$ such that 
   \begin{align*}
      G_{\Psi}(n) = G(n') +  \int_{\T^d} \Psi(x,n) (n - n')(dx), 
   \end{align*}
  and let  $m' = \big( m - (n - n') \big)_+$. We note that 
   \begin{align*}
       n' - m' = n' - \big( m - (n - n') \big)_+ = \big( n - (n - n') \big)_+ - \big( m - (n - n') \big)_+,
   \end{align*}
   and, hence, $0 \leq n' - m' \leq n - m$, so that 
   \begin{align*}
       \bd(n',m') = n'(\T^d) - m'(\T^d) \leq n(\T^d) - m(\T^d). 
   \end{align*}
 It follows that 
   \begin{align*}
       G_{\Psi}(m) &\leq G(m') + \int_{\T^d} \Psi(x,m) d (m - m')
       \\
       &\leq G(n')  + \int_{\T^d} \Psi(x,n) d (n - n') + C \bd(m,n) + C \bd(m',n')
       \\
       &\leq G_{\Psi}(n) + C \bd(m,n). 
   \end{align*}
   At this stage, we have shown that there is a constant $C$ such that 
   \begin{align*}
       |G_{\Psi}(m) - G_{\Psi}(n)| \leq C \bd(m,n)
   \end{align*}
   holds if either $m(\T^d) = n(\T^d)$ or $m \leq n$, and, hence, in view of  Lemma \ref{lem.dequrho},  $G_{\Psi}$ is Lipschitz. 
   \vs
   The argument for the Lipschitz property of $(G_{\Psi}^{N,K})$ is similar. 

\end{proof}

\begin{lemma} \label{lem.gnk.conv}
Suppose that  $\Psi$ and $G$ are Lipschitz. Then, as $N\to \infty$,  $G^{N,K}_{\Psi} \to G_{\Psi}$ uniformly, in the sense that 
\begin{align*}
    \lim_{N \to \infty} \max_{K = 1,...,N} \sup_{\bx \in (\T^d)^k} \Big| G_{\Psi}(m_{\bx}^{N,K}) - G_{\Psi}^{N,K}(\bx) \Big| = 0.
\end{align*}
\end{lemma}

\begin{proof}
    We begin  proving the  weaker statement that, for each $m \in \sub$, there exist  sequences $K_N \in \{1,...,N\}$ and $\bx_N \in (\T^d)^{K_N}$ such that, as $N\to \infty$,  
    \begin{align*}
        \bd(m, m_{\bx}^{N,K_N}) \to 0 \ \ \text{and} \ \  \big| G_{\Psi}(m) - G^{N,K}_{\Psi}(m_{\bx_N}^{N,K_N}) \big| \to 0.
    \end{align*}
    \vs
    Let $m' \leq m$ be such that 
    \begin{align*}
        G_{\Psi}(m) = G(m') + \int_{\T^d} \Psi(x,m) d(m-m')(x),
    \end{align*}
    and choose $J_N^1, J_N^2$, $\by_N \in (\T^d)^{J_N^1}$ and $\bz_N \in (\T^d)^{J_N^1}$ such that, as $N\to \infty$, 
    \begin{align*}
       K_N = J_N^1 + J_N^2 \leq N, \quad m_{\by_N}^{N,J_N^1} \to (m -m') \ \text{and} \  m_{\bz_N}^{N,J_N^2} \to m'. 
    \end{align*}
   \vs
     If we set $\bx_N = (\by_N, \bz_N) \in (\T^d)^{K_N}$, it follows that 
    \begin{align*}
        m_{\bx_N}^{N,K_N} \to m, 
    \end{align*}
    and also 
    \begin{align*}
        &G_{\Psi}^{N,K_N}(\bx_N) \leq G\big(m_{\bz_N}^{N,J_N^2}\big) + \frac{1}{N} \sum_{i = 1}^{J_N^1} \Psi\big(y_N^i, m_{\bx_N}^{N,K_N}\big)
        \\
        &\quad \leq G(m') + \int_{\T^d} \Psi(x,m) d(m-m')(x) + C \bd\big(m, m_{\bx_N}^{N,K_N}\big)+ C \bd\big(m-m', m_{\by_N}^{N,J^1_N}\big) + C \bd\big(m', m_{\bz_N}^{N,J_N^2}\big). 
    \end{align*}    
    \vs
  Thus
    \begin{align*}
        \limsup_{N \to \infty} \big( G_{\Psi}^{N,K_N}(\bx_N) - G_{\Psi}(m) \big) \leq 0.
    \end{align*}
    \vs
    Using the Lipschitz continuity of $G_{\Psi}$ and the definition of $G_{\Psi}^{N,K_N}$, we also know that 
    \begin{align*}
    G_{\Psi}(m) \leq G_{\Psi}\big(m_{\bx_N}^{N,K_N}\big) + C \bd\big(m,m_{\bx_N}^{N,K_N}\big) \leq G_{\Psi}^{N,K_N}(\bx_N) + C \bd\big(m,m_{\bx}^{N,K_N}\big), 
    \end{align*}
    so that 
        \begin{align*}
        \liminf_{N \to \infty} \big( G_{\Psi}^{N,K_N}(\bx_N) - G_{\Psi}(m) \big) \geq 0, 
    \end{align*}
   which completes the proof of the assertion.
 \vs 
    Next, following the reasoning in the proof of Theorem \ref{thm.convergence}, the uniform Lipschitz estimates on $(G^{N,K}_{\Psi})_{K= 1,...,N}$  in Lemma \ref{lem.envelopeLip}  together with the Arzela-Ascoli theorem imply that, for each subsequence $(N_j)_{j \in \N}$, there is a further subsequence $(N_{j_k})_{k \in \N}$ and a Lipschitz function $F : \sub \to \R$, which, a-priori, may depend on the subsequence,  such that 
    \begin{align*}
        \lim_{k \to \infty} \max_{K = 1,...,N_{j_k}} \sup_{\bx \in (\T^d)^K} \big| G^{N_{j_k},K}_{\Psi}(\bx) - F(m_{\bx}^{N_{j_k},K}) \big| = 0.
    \end{align*}
    \vs
    But then the first step of the proof   implies that $F = G_{\Psi}$, and thus  we deduce that,   for each subsequence $(N_j)_{j \in \N}$, there is a further subsequence $(N_{j_k})_{k \in \N}$ such that 
    \begin{align*}
        \lim_{k \to \infty} \max_{K = 1,...,N_{j_k}} \sup_{\bx \in (\T^d)^K} \big| G^{N_{j_k},K}_{\Psi}(\bx) - G_{\Psi}(m_{\bx}^{N_{j_k},K}) \big| = 0.
    \end{align*}
  The proof is now complete.
  
\end{proof}

\subsection*{Mollification}
We prove the following result. 

\begin{proposition} \label{prop.mollification}
    Suppose that $G : \sub \to \R$ is $\Psi$-non-increasing and Lipschitz continuous. Then, there exists a sequence $(G^k)_{k \in \N}$ such that 

    (i)~ each $G^k$ is $\Psi$-non-increasing,
    \vs
    (ii)~ $G^k$ is Lipschitz continuous, uniformly in $k$, 
    \vs
   (iii)~ $G^k$ is $C^1$, with $\dfrac{\delta G^k}{\delta m}$ Lipschitz continuous,  and
   \vs
   (iv)~$G^k \to G$ uniformly on $\sub$.
    \end{proposition}

\begin{proof}
    Following the mollification argument in \cite{MZ24}, there is a constant $C > 0$ such that, for each $\delta > 0$, there exists a finite smooth partition of unity $\{\phi_i\}_{i = 1,...,n}$ on $\T^d$ with the property that 
    \begin{align*}
       \text{diam}\big( \text{supp}(\phi_i) \big) \leq \delta \ \ \text{ and } \ \ 
        |D \phi_i| \leq \frac{C}{\delta}. 
    \end{align*}
   Choose $x_i \in \text{supp}(\phi_i)$ for each $i$ and define 
  the map  $T_{\delta} : \sub \to \sub$ by 
   \begin{align*}
       T_{\delta} m = \sum_{i = 1}^n \langle \phi_i, m \rangle \delta_{x_i}. 
   \end{align*}
   We note that, for any $f \in W^{1,\infty}$ with $\|f\|_{1,\infty} \leq 1$, we have 
   \begin{align*}
       \int_{\T^d} f d\big( m - T_{\delta} m\big) &= \int_{\T^d} \Big( f(x) - \sum_{i = 1}^n f(x_i) \phi_i(x) \Big) dm(x) \leq \sup_{x  \in \T^d} \Big|f(x) - \sum_{i = 1}^n f(x_i) \phi_i(x) \Big|
       \\
 &=  \sup_{x  \in \T^d} \Big| \sum_{i = 1}^n f(x)\phi_i(x) - \sum_{i = 1}^n f(x_i) \phi_i(x) \Big|
       \leq \sup_{x \in \T^d} \sum_{i = 1}^n |f(x) - f(x_i)| \phi_i(x) \leq \delta, 
   \end{align*}
   where the last inequality uses the fact that, if $\phi_i(x) > 0$, then $x \in \text{supp}(\phi_i)$ and $ |x- x_i| \leq \delta.$
   \vs
   It follows that 
   \begin{align}
       \bd(m, T_{\delta} m ) \leq \delta.
   \end{align}
   Next, we notice that, for $m,n \in \sub$, and $f \in W^{1,\infty}$ with $\|f\|_{1,\infty} \leq 1$, we have 
   \begin{align} \label{Tdelta.Lip}
     \int f d\big( T_{\delta} m - T_{\delta} n \big) = \int_{\T^d} \Big(\sum_{i =  1}^n f(x_i) \phi(x) \Big) d(m - n), 
   \end{align}
   and 
   \begin{align*}
      \Bigl| D_x \Big(\sum_{i =  1}^n f(x_i) \phi_i(x) \Big)\Bigr|  = \Bigl|\sum_{i = 1}^n f(x_i) D_x \phi_i(x) \Bigr|
= \Bigl| \sum_{i = 1}^n \big(f(x_i) - f(x)\big) D_x \phi_i(x)\Bigr|
 \leq \delta \frac{C}{\delta} = C, 
   \end{align*}
   where we used the facts that  $|x_i - x| \leq \delta$ if $D_x \phi_i(x) \neq 0$, $|D_x \phi_i(x) | \leq C/\delta$ and 
   \begin{align*}
       \sum_{i - 1}^n D_x \phi_i(x) = D_x \big( \sum_{i = 1}^n \phi_i(x) \big) = D_x 1 = 0.
   \end{align*}
   \vs
   Coming back to \eqref{Tdelta.Lip}, we deduce that 
   \begin{align*}
        \int f d\big( T_{\delta} m - T_{\delta} n \big) \leq C\bd(m,n), 
   \end{align*}
   and, hence, 
   \begin{align} \label{Tdelta.lip.2}
       \bd(T_{\delta} m, T_{\delta} n) \leq C \bd(m,n).
   \end{align}
  Next, we take  a standard approximation to the identity $(\rho_{\eta})_{\eta>0}$ on $\R$ with $\text{supp}(\rho_{\eta}) \subset [-\eta, \eta]$ and, for fixed $\delta$, we set
   \begin{align*}
     r_{\delta} = \min_{i = 1,...,n}  \int_{\T^d} \phi_i dx,
   \end{align*}
    and then, for each $0 < \eta < r_{\delta}$, we define $G_{\delta, \eta} : \sub \to \R$ by 
   \begin{align*}
      G_{\delta, \eta}(m) = \int_{\R^n} G\Big( \sum_{i = 1}^n \big(\langle \phi_i, (1-  \frac{\eta}{r_{\delta}}) m + \frac{\eta}{r_{\delta}} \text{Leb} \rangle + y_i \big) \delta_{x_i}  \Big) d\rho_{\eta}^{\otimes n}(\by).
   \end{align*}
   It follows that 
   \be\label{A6}
   \begin{split}
    &  \big| G_{\delta, \eta}(m)  - G(m) \big|\\
    & \leq \int_{\R^n} \Big| G\Big( \sum_{i = 1}^n \big(\langle \phi_i, (1-  \frac{\eta}{r_{\delta}}) m + \frac{\eta}{r_{\delta}} \text{Leb} \rangle + y_i \big) \delta_{x_i}  \Big) - G(m) \Big| d\rho_{\eta}^{\otimes n}(\by) 
    \\
      & \leq \int_{\R^n} \bigg( \bd\Big( \sum_{i = 1}^n \big(\langle \phi_i, (1-  \frac{\eta}{r_{\delta}}) m + \frac{\eta}{r_{\delta}} \text{Leb} \rangle + y_i \big) \delta_{x_i} ,T_{\delta}\big( (1 - \frac{\eta}{r_{\delta}}) m + \frac{\eta}{r_{\delta}} \text{Leb}\big) \Big) d\rho_{\eta}^{\otimes n}(\by)
      \\
      &\qquad  + \bd\Big( T_{\delta}\big( (1 - \frac{\eta}{r_{\delta}}) m + \frac{\eta}{r_{\delta}} \text{Leb}\big),T_{\delta}\big( m \big)\Big) + \bd\Big( T_{\delta}\big( m \big), m \Big)
  \leq C \Big( \eta + \frac{\eta}{r_{\delta}} + \eta  + \delta \Big).
   \end{split}
   \ee
   Moreover, in view of \eqref{Tdelta.lip.2}, 
   \begin{align*}
     & \Big| G_{\delta, \eta}(m) - G_{\delta,\eta}(n) \Big| 
      \\
      &\leq \int_{\R^n} \Big| G\Big( \sum_{i = 1}^n \big(\langle \phi_i, (1-  \frac{\eta}{r_{\delta}}) m + \frac{\eta}{r_{\delta}} \text{Leb} \rangle + y_i \big) \delta_{x_i}  \Big)
      \\
      &\qquad \qquad - G\Big( \sum_{i = 1}^n \big(\langle \phi_i, (1-  \frac{\eta}{r_{\delta}}) n + \frac{\eta}{r_{\delta}} \text{Leb} \rangle + y_i \big) \delta_{x_i}  \Big) \Big| d\rho_{\eta}^{\otimes n}(\by)
      \\
      &\leq \int_{\R^n} \bd\Big( (1 - \frac{\eta}{r_{\delta}}) T_{\delta}m + \sum_{i = 1}^n y_i \delta_{x_i}, (1 - \frac{\eta}{r_{\delta}}) T_{\delta}n + \sum_{i = 1}^n y_i \delta_{x_i} \Big) d\rho_{\eta}^{\otimes n}(\by)
      \\
      &= \bd \Big( (1 - \frac{\eta}{r_{\delta}}) T_{\delta}m , (1 - \frac{\eta}{r_{\delta}}) T_{\delta}n \Big) 
      \leq \bd\Big( T_{\delta} m, T_{\delta} n \Big) \leq C \bd(m,n), 
   \end{align*}
   \vs
Finally, again following \cite{MZ24}, it is straightforward to check that $G_{\delta, \eta}$ is smooth, and, in particular, that  $\dfrac{\delta G_{\delta, \eta}}{\delta m}$ is globally Lipschitz. 
\vs
   Next, for $m' < m$, we use the fact that $G$ is $\Psi$ non-increasing to get 
   \begin{equation} \label{Gdeltaeta.nondecreasing} 
   \begin{split} 
       G_{\delta, \eta}&(m) = \int_{\R^n} G\Big( \sum_{i = 1}^n \big(\langle \phi_i, (1-  \frac{\eta}{r_{\delta}}) m + \frac{\eta}{r_{\delta}} \text{Leb} \rangle + y_i \big) \delta_{x_i}  \Big) d\rho_{\eta}^{\otimes n}(\by) 
       \\
       &\leq  \int_{\R^n}  \bigg( G\Big( \sum_{i = 1}^n \big(\langle \phi_i, (1-  \frac{\eta}{r_{\delta}}) m' + \frac{\eta}{r_{\delta}} \text{Leb} \rangle + y_i \big) \delta_{x_i}  \Big)
       \\
       & \quad + \int_{\T^d} \Psi\Big(x, \sum_{i = 1}^n \big(\langle \phi_i, (1-  \frac{\eta}{r_{\delta}}) m + \frac{\eta}{r_{\delta}} \text{Leb} \rangle + y_i \big) \delta_{x_i} \Big) d\Big( \sum_{i = 1}^n \langle \phi_i, (1 - \frac{\eta}{r_{\delta}}) (m - m') \rangle \delta_{x_i}\Big)(x) \bigg) d\rho_{\eta}^{\otimes n}(\by) 
       \\ 
       &=  
       G_{\delta, \eta}(m') 
       \\ 
       &\quad  + \big(1 - \frac{\eta}{r_{\delta}} \big) \int_{\R^n} \int_{\T^d} \Big( \sum_{i = 1}^n \Psi\big(x_i, \sum_{i = 1}^n \big( \langle (1-  \frac{\eta}{r_{\delta}}) m + \frac{\eta}{r_{\delta}} \text{Leb} \rangle + y_i\big)\delta_{x_i} \big) \phi_i(x) \Big) d(m - m')(x) d\rho_{\eta}^{\otimes n}(\by).
   \end{split}
   \ee
   Then, notice that 
   \begin{align*}
       \Big| &\Psi(x, m)  - \sum_{i = 1}^n \phi_i(x) \Psi\big(x_i,\sum_{i = 1}^n  \big( \langle (1-  \frac{\eta}{r_{\delta}}) m + \frac{\eta}{r_{\delta}} \text{Leb}, \phi_i \rangle + y_i\big)\delta_{x_i} \big) \Big|
       \\
       &\leq C \delta + \Big|\Psi(x,m) - \Psi\big(x, \sum_{i = 1}^n  \big( \langle (1-  \frac{\eta}{r_{\delta}}) m + \frac{\eta}{r_{\delta}} \text{Leb}, \phi_i \rangle + y_i\big)\delta_{x_i} \big) \Big|
       \\
       &\leq C \delta + \bd\Big(m,  \sum_{i = 1}^n  \big( \langle (1-  \frac{\eta}{r_{\delta}}) m + \frac{\eta}{r_{\delta}} \text{Leb}, \phi_i \rangle + y_i\big)\delta_{x_i} \Big) 
       \\
       &\leq C\Big( \delta + \sum_{i = 1}^n |y_i| + \eta + \frac{\eta}{r_{\delta}} \Big) + \bd\Big(m, T_{\delta}m \Big)
       \leq C \Big( \delta + \sum_{i = 1}^n |y_i| + \eta + \frac{\eta}{r_{\delta}} \Big). 
   \end{align*}
   Coming back to \eqref{Gdeltaeta.nondecreasing}, we get, for  $C$ independent of $\delta$ and $\eta$, 
   \begin{align*}
        G_{\delta, \eta}(m) \leq G_{\delta, \eta}(m') + \int_{\T^d} \Psi(x,m) d(m - m') + C \Big( \delta +  (1 + n) \eta + \frac{\eta}{r_{\delta}} \Big) (m - m')(\T^d).
   \end{align*}
   Thus, setting 
   \begin{align*}
       \wt{G}_{\delta, \eta}(m) = G_{\delta, \eta}(m) - C \Big( \delta +  (1 + n) \eta + \frac{\eta}{r_{\delta}} \Big) m(\T^d), 
   \end{align*}
   we see that $\wt{G}_{\delta, \eta}$ is $\Psi$-non-increasing. 
   \vs
   Moreover, clearly 
   \begin{align*}
        \big| \wt{G}_{\delta, \eta}(m) - G(m) \big| &\leq  \big| G_{\delta, \eta}(m) - G(m) \big| + C \Big( \delta +  (1 + n) \eta + \frac{\eta}{r_{\delta}} \Big)
        \leq C \Big( \delta +  (1 + n) \eta + \frac{\eta}{r_{\delta}} \Big), 
   \end{align*}
   and 
   \begin{align*}
       \big|\wt{G}_{\delta, \eta}(m) - \wt{G}_{\delta, \eta}(n) \big| &\leq \big|{G}_{\delta, \eta}(m) - {G}_{\delta, \eta}(n) \big| +   C \Big( \delta +  (1 + n) \eta + \frac{\eta}{r_{\delta}} \Big) (m - m')(\T^d) 
       \\
       &\leq C \Big( 1 + \delta +  (1 + n) \eta + \frac{\eta}{r_{\delta}} \Big) \bd(m,n).
   \end{align*}
   We can thus build the desired approximation by choosing first $\delta$ and then $\eta$ small enough. 

\end{proof}

\bibliographystyle{alpha}

\newcommand{\etalchar}[1]{$^{#1}$}

\end{document}